\documentclass[12pt]{amsart}
\usepackage[dvipsnames]{xcolor}
\usepackage{url,graphicx,xypic,amssymb}
\usepackage{tikz}

\usepackage{hyperref}
\usepackage[nameinlink,capitalize]{cleveref}
\newcommand\myshade{85}
\colorlet{mylinkcolor}{violet}
\colorlet{mycitecolor}{YellowOrange}
\colorlet{myurlcolor}{Aquamarine}

\hypersetup{
  linkcolor  = mylinkcolor!\myshade!black,
  citecolor  = mycitecolor!\myshade!black,
  urlcolor   = myurlcolor!\myshade!black,
  colorlinks = true,
}

\usepackage[margin=1.4in]{geometry}

\numberwithin{equation}{section}

\newtheorem{theorem}{Theorem}[section]
\newtheorem*{theorem*}{Theorem}
\newtheorem{proposition}[theorem]{Proposition}
\newtheorem{lemma}[theorem]{Lemma}
\newtheorem{corollary}[theorem]{Corollary}
\newtheorem*{corollary*}{Corollary}
\newtheorem{conjecture}[theorem]{Conjecture}

\newtheorem{Atheorem}{Theorem}

\theoremstyle{definition}
\newtheorem{definition}[theorem]{Definition}

\theoremstyle{definition}
\newtheorem{example}[theorem]{Example}

\newtheorem{remark}[theorem]{Remark}

\def\CC{\mathbf{C}}
\def\K{\mathcal{K}} % for K-polynomial

\def\PP{\mathbf{P}}

\def\QQ{\mathbf{Q}}
\def\RR{\mathbf{R}}

\def\ZZ{\mathbf{Z}}

\def\rk{\operatorname{rank}}

\def\conv{\operatorname{conv}}

\def\GL{\mathrm{GL}}

\def\Sym{\mathrm{Sym}}
\def\Hilb{\mathrm{Hilb}}
\def\codim{\mathrm{codim}}

\def\Trop{\operatorname{Trop}}
\def\tr{\operatorname{tr}}
\def\ea{\operatorname{ea}}
\def\link{\operatorname{link}}

\def\CS{{$\operatorname{{CS}}^*$}}
\def\rec{\operatorname{rec}}

\newcommand{\newword}[1]{\emph{#1}}

\title{The external activity complex of a pair of matroids}

\author{Andrew Berget}
\address{Western Washington University, Bellingham, WA, USA}
\email{andrew.berget@wwu.edu}

\author{Alex Fink}
\address{Queen Mary University of London, London, UK}
\email{a.fink@qmul.ac.uk}

\begin{document}
\maketitle

\begin{abstract}
  We introduce the Schubert variety of a pair of linear subspaces in $\CC^n$
  and the external activity complex of a pair of not necessarily realizable matroids.
  Both of these generalize constructions of Ardila et al., which occur when one of the linear spaces is one-dimensional. We
  prove that our external activity complex is Cohen-Macaulay and deduce a formula for its
  $K$-polynomial in terms of exterior powers of the dual tautological
  quotient classes of matroids. As a consequence, we deduce a non-negative formula for
  the matroid invariant $\omega(M)$ of Fink, Shaw, and Speyer
  in terms of certain homology groups of links within an external activity complex,
  proving the 2005 tropical $f$-vector conjecture of Speyer.
\end{abstract}

\section{Introduction}
In his thesis \cite{speyerThesis} Speyer
studied the subdivision of a hypersimplex into matroid base polytopes.
His interest in these subdivisions arose because, when regular, they parametrize tropical linear spaces,
which have come to be recognized as fundamental objects in tropical geometry.
A tropical linear space is a polyhedral complex that generalizes the tropicalization of a linear subspace of~$K^n$, where $K$ is a valued field. 
Speyer conjectured an upper bound on the number of faces by dimension:
\begin{conjecture}[The tropical $f$-vector conjecture \cite{speyerThesis}]\label{conj:f-vector conjecture}
  The number of $(n-i)$-dimensional interior faces in a subdivision
  of $\Sigma(r,n)$ into matroid base polytopes is at most
  \[
    \frac{(n-i-1)!}{(r-i)!\,(n-r-i)!\,(i-1)!}.
  \]
\end{conjecture}
%\begin{conjecture}[The tropical $f$-vector conjecture \cite{speyerThesis}]\label{conj:f-vector conjecture}
%  The number of bounded $(n-i)$-dimensional faces in a tropical linear space is at most
%  \[
%    \frac{(n-i-1)!}{(r-i)!\,(n-r-i)!\,(i-1)!}.
%  \]
%\end{conjecture}
The conjecture also bounds the complexity of other objects parametrized by matroid base polytope subdivisions,
including the very stable pairs of Hacking and Keel \cite{hackingKeel}, and Kapranov's Lie complexes \cite{kapranov}, as
explained in \cite{speyer}. 
In that paper, Speyer proves his conjecture for
subdivisions of $\Sigma(r,n)$ into base polytopes
of matroids realizable over $\CC$.
The crucial ingredient in his proof is an application of the cohomology
vanishing theorem of Kawamata--Viehweg applied to a smooth proper
variety constructed from $L$.

In this paper we prove \cref{conj:f-vector conjecture}. Our proof
constructs a variety associated to a pair of linear susbspaces
$L_1,L_2 \subset \CC^n$, which we Gr\"obner deform to a simplicial
complex $\Delta_w(M_1,M_2)$ that can be defined purely in terms of the
matroids $M_k$ of the linear spaces $L_k$. This complex is the
\textit{external activity complex} of the pair $(M_1,M_2)$, and is the
central figure of our work. In place of Speyer's use of
Kawamata--Viehweg vanishing, we prove the Cohen-Macaulay property for
our complexes which furnishes the needed positivity. Our proof of this
property proceeds in a unique way, weaving together the geometry of the realizable
case with tropical intersection theory and commutative algebra. In
this introduction, which culminates in the proof of
\cref{conj:f-vector conjecture}, we give an account of how we arrive
at our results.

A good starting point for this story is the
% The
\newword{Schubert variety of a linear space} $L \subset \CC^n$
, which 
is the closure $Y_L$ of $L$ in $\CC^n \subset (\PP^1)^n$. This variety
was first studied by Ardila and Boocher \cite{ardilaBoocher} and
shortly thereafter appeared in the work of Huh and Wang \cite{huhWang}
where it was used in the resolution for realizable matroids
of the longstanding Dowling-Wilson conjecture
(leading to its complete resolution in \cite{BHMPW}). 
We position Schubert varieties of linear spaces
as members of a class of varieties that come from the tautological bundles of linear spaces \cite{BEST} 
by a procedure
known as \newword{Kempf collapsing} \cite{kempf-collapsing}.

\subsection{Collapsing of vector bundles} The input to a collapsing is a
subbundle $E$ of a trivial bundle
$\underline{\CC^m}:=\CC^m \times X\to X$, where $X$ is a smooth
projective variety. The collapsing of $E$ is its image under the
projection map to $\CC^m$. In \cite{kempf-collapsing}, Kempf was
concerned with the orbit under an algebraic group $G$ of a $P$-stable
subspace $V$ of a representation $W$ of $G$, where $P\subseteq G$ is a
parabolic subgroup; this orbit is a collapsing of $G \times^P V$,
which is a subbundle of $\underline{W}\to G/P$. He showed that in
characteristic zero, these orbits are closed, normal, Cohen-Macaulay
and have rational singularities when the collapsing map is
birational. Many other well known varieties occur as collapsings and
enjoy most of these desirable properties. Examples include
determinantal varieties \cite{lascoux}, rank varieties of symmetric
and skew symmetric tensors and nilpotent orbit closures \cite{weyman},
matrix Schubert varieties \cite{km} and quiver loci
\cite{knutson-shimozono}.

Let $T$ be the $n$-torus $(\CC^\times)^n$, which acts on $\CC^n$ by
inverse scaling of coordinates. For a linear space $L \subset \CC^n$,
we let $\mathcal{S}_L$ denote the tautological subbundle of $L$;
this is a bundle over the permutohedral toric variety $X_n$,
the unique $T$-equivariant subbundle of
$\underline{\CC^n} = \CC^n \times X_n$ whose fiber over the identity of the dense torus of~$X_n$ 
is $L$. Let $\mathcal{Q}_L$ denote the corresponding quotient
bundle $\underline{\CC^n}/\mathcal{S}_L$. These are, respectively, the \textit{tautological sub- and quotient bundle} of $L$ \cite{BEST}. We perform collapsing on
$\mathcal{S}_{L_1} \oplus \mathcal{S}_{L_2}$\,, 
obtaining a variety which we call the
(affine) \newword{Schubert variety} of the pair of linear spaces
$(L_1,L_2)$; we denote this by $\hat Y_{L_1,L_2}$. The quotient of
$\hat Y_{L_1,L_2}$ by $T$ is a subvariety of $(\PP^1)^n$ which we
denote by $Y_{L_1,L_2}$. In the case when $L_2$ is a general one
dimensional subspace, $Y_{L_1,L_2} = Y_{L_1}$ is the Schubert
variety of~$L_1$ as studied by Ardila and Boocher.

In general, many interesting invariants of a subbundle $E$ and its
associated quotient bundle $F=\underline{\CC^m}/E$ are reflected in the collapsing
of $E$, including higher cohomology of symmetric powers of $E^\vee$
and exterior powers of $F^\vee$. A systematic treatment of such
results is given in \cite[Chapter 5]{weyman}. We contribute two small
ways that the geometry of a bundle is conveyed in its collapsing in
\cref{prop:K,prop:C}, which relate the $K$-polynomial and the
multidegree of a collapsing to Euler characteristics of exterior
powers, respectively Chern classes, of $F^\vee$. 
These invariants are more manageable than the entire minimal free resolution or
sequence of cohomology groups involved in Weyman's results.

Our first result (recapitulated in \cref{sec:schubert}) is purely geometric, and serves as a necessary stepping stone towards our main and significantly more general results below.
Compare item~(3) especially with \cref{prop:K}.
\begin{Atheorem}\label{Athm:geometric}
  Let $L_1,L_2$ be a pair of linear subspaces of $\CC^n$ with
  respective matroids $M_1$ and $M_2$, whose affine Schubert variety
  is $\hat Y_{L_1,L_2} \subset \CC^n \times \CC^n$. Then
  \begin{enumerate}
  \item $\hat Y_{L_1,L_2}$ is irreducible, normal and has rational singularities.
  \item If the dimension of $\hat Y_{L_1,L_2}$ is
    $n+\rk(M_1) + \rk(M_2) - 1$, then
    $\mathcal{S}_{L_1} \oplus \mathcal{S}_{L_2}$ is a rational
    resolution of singularities of $Y_{L_1,L_2}$.
  \item The $\ZZ^2$-graded $K$-polynomial of $Y_{L_1,L_2}$ depends
    only on $M_1$ and $M_2$ and is equal to
    \[
      \sum_{i,j \geq 0} \textstyle\chi(X_n,\bigwedge^i [\mathcal{Q}_{M_1}^\vee] \cdot \bigwedge^j [\mathcal{Q}_{M_2}^\vee]) (-U_1)^i (-U_2)^j.
    \]
  \end{enumerate}
\end{Atheorem}
Here $[\mathcal{Q}_M]$ is the tautological quotient class of the matroid $M$,
defined in \cite{BEST}. This is an element of the Grothendieck group
$K_0(X_n)$ designed to imitate $\mathcal{Q}_L$
for a matroid $M$ on~$[n]$ which need not be realizable:
when $M$ is in fact
realized by a linear space $L \subset \CC^n$ we have
$[\mathcal{Q}_M] = [\mathcal{Q}_L]\in K_0(X_n)$.
Likewise a class $[\mathcal{S}_M]$ will appear in \Cref{ssec:f-vector conjecture}.

Brion showed that if the $K$-theory class of a subvariety with rational singularities within a flag variety 
is expanded in the Schubert basis, then
the coefficients in the expansion are of alternating sign \cite[Theorem~1]{brionpos}. 
This sign pattern has become accepted as a notion of positivity for $K$-theory classes.
Translating to $K$-polynomials, we have
that for a normal subvariety $Y \subset \CC^n \times \CC^n$ with
rational singularities and $\ZZ^2$-graded $K$-polynomial $\K(Y;U,V)$, the coefficient
of $U^i V^j$ in $\K(Y;1-U,1-V)$ has sign $(-1)^{i+j-\codim(Y)}$. We
will refer to this as \textit{$K$-theoretic positivity} of (the class of) $Y$. 
It implies that the terms of highest total degree in $K(Y;U,V)$
also have a predictable sign. When this is applied to the Schubert
variety of a pair of linear spaces, we obtain the following result: If
$\hat Y_{L_1,L_2}$ has dimension $n+\rk(M_1) +\rk(M_2) - 1$, then
\[\textstyle(-1)^{\rk(M_1) +\rk(M_2) - 1}\chi(X_n,[\bigwedge^i
  \mathcal{Q}_{M_1}^\vee] \cdot [\bigwedge^j \mathcal{Q}_{M_2}^\vee])
  \geq 0.\]

In many noteworthy examples, we have a stock of existing results to compute the
intersection numbers and Euler characteristics in
\cref{prop:C,prop:K}, and help inform our understanding of the
$K$-polynomial or degree of a collapsing. For example, when
investigating determinantal varieties, the Borel-Weil-Bott theorem can
be used to compute the relevant Euler characteristics, and this can be
used to explicitly describe the minimal free resolution of these
varieties. For the Schubert variety of one linear space $Y_L$, the
relevant degree and $K$-polynomial can be extracted from the Tutte
polynomial of the matroid of $L$ using \cite[Theorems
A,~D]{BEST} (see \cref{ex:S_L ++ O(-beta)}). However, for the Schubert variety of a pair of linear
spaces we have no such combinatorial formulas and must resort to other
means. In what follows, the flow of information is reversed, and we
learn about intersection numbers and Euler characteristics of bundles
by studying $K$-polynomials and multidegrees of collapsings.

\subsection{Gr\"obner degenerations} The idea of combining Kempf
collapsing of vector bundles and Gr\"obner degenerations is latent in
the study of matrix Schubert varieties by Knutson and Miller
\cite{km}. In their case, one takes $X=\Omega_w$ to be a Schubert variety
in the full flag variety $\GL_n(\CC)/B$. The flag variety has a full flag of
tautological subbundles
$\mathcal{S}_1 \subset \mathcal{S}_2 \subset \dots
\subset\mathcal{S}_n$ where $\mathcal{S}_i$ has rank $i$. We let $E$
be the restriction of $\bigoplus_{i=1}^n \mathcal{S}_i$ to
$\Omega_w$. The collapsing of $E$ to the space of $n$-by-$n$ matrices
is the matrix Schubert variety $Y_w$. In \cite{km}, a Gr\"obner basis
is computed for the ideal of $Y_w$ and the multidegree and
$K$-polynomial are shown to be the Schubert polynomial and
Grothendieck polynomials indexed by~$w$; this verifies these polynomials as being canonical representatives of the classes of Schubert varieties. The Gr\"obner bases of the
ideals of matrix Schubert varieties are reflected in various known
formulas for the Schubert and Grothendieck polynomials in terms of
combinatorial objects such as pipe dreams.

In spite of the remarkable success of Knutson and Miller's approach,
the idea of using a Gr\"obner degeneration of a Kempf collapsing to
inform our understanding of vector bundles and their Euler
characteristics, or cohomology groups or Chern classes, has not
received much attention. One explanation is that in general Betti
numbers increase upon taking a Gr\"obner degeneration (although
$K$-polynomials are unchanged). 
We take up the idea in this paper, using initial degenerations of~$Y_{L_1,L_2}$
to understand tautological bundles of linear spaces 
and tautological classes of matroids. 
Remarkably, the Betti numbers of~$Y_{L_1,L_2}$ do not
change on passing to the initial ideal. 

To state our theorem on Gr\"obner degenerations we need an important auxiliary matroid, 
the \newword{diagonal Dilworth truncation} $D(M_1,M_2)$, 
studied in depth in \cref{sec:D}.
For a pair of matroids $(M_1,M_2)$ on the same ground
set $[n]$ and without loops in common, we let $D(M_1,M_2)$ be the
matroid on $[n]$ whose circuits are those minimal nonempty subsets $C$ with the
property that $\rk_{M_1}(C) + \rk_{M_2}(C) = |C|$. 
The $w$-initial decomposition of~$C$ is $I_1 \sqcup I_2$ if $I_j$ is
independent in $M_j$ and, under the monomial weight order~$w$, 
the monomial $x_{I_2} y_{I_1}$ achieves maximum weight over all such decompositions.
\begin{Atheorem}\label{Athm:initial ideal}
  Let $A = \CC[x_1,\dots,x_n,y_1,\dots,y_n]$. 
  Let $M_1$ and $M_2$ be the matroids of
  $L_1$ and $L_2$, and assume that $M_1$ and $M_2$ have no loops in
  common. The prime ideal $I(L_1,L_2) \subset A$ of $Y_{L_1,L_2}$ 
  has a universal Gr\"obner basis consisting of
  determinants indexed by circuits of $D(M_1,M_2)$. 
  The initial ideal $I_w(L_1,L_2)$ of~$I(L_1,L_2)$ under any given
  monomial weight order $w$ is equal to
  \[
    ( x_{I_2} y_{I_1} : I_1 \sqcup I_2 \textup{ is the $w$-initial decomp.\ of a circuit of $D$}).
  \]
  The $\ZZ^n$-graded Betti numbers of $Y_{L_1,L_2}$ are equal to those of $A/I_w(L_1,L_2)$.
\end{Atheorem}
This result will serve as a bridge between the earlier geometric result and the combinatorial results to come.

We write $I_w(M_1,M_2):=I_w(L_1,L_2)$, 
as the initial ideal depends only on the matroids $M_1, M_2$.
Let $\Delta_w(M_1,M_2)$ be the the Stanley-Reisner complex
associated to $I_w(M_1,M_2)$, whose facets we will show are
indexed by the bases of $D(M_1,M_2)$.  
We call $\Delta_w(M_1,M_2)$ the
\newword{external activity complex} of the pair $(M_1,M_2)$, borrowing
the name that Ardila and Boocher \cite{ardilaBoocher} gave the complex
% $\Delta_w(M_1,U_{1,n})$
they obtained from the Schubert variety of a
single linear space.

Of course, the definition $\Delta_w(M_1,M_2)$ makes sense for any pair
of matroids without loops in common. In \cref{sec:D} we will define
$\Delta_w(M_1,M_2)$ by its facets; only after considerable effort will we
deduce the minimal nonfaces described above (\Cref{prop:ideal in general}). 

\subsection{Pairs of matroids and our main results}
For pairs of matroids $(M_1,M_2)$
realized by linear spaces $(L_1,L_2)$, the $K$-polynomial of
$\Delta_w(M_1,M_2)$ is that of $Y_{L_1,L_2}$ and described in \cref{Athm:geometric}.
To move beyond this one faces what Adiprasito, Huh and Katz \cite{ahk} call a
``difficult chasm'': There is no reason to expect the
$K$-theoretic positivity phenomenon for the $K$-polynomials of the $Y_{L_1,L_2}$ 
to extend to $K$-polynomials of $\Delta_w(M_1,M_2)$ 
when the pair of matroids is not realizable. The
Stanley-Reisner ring of $\Delta_w(M_1,M_2)$ will never have rational
singularities (it is not a domain); the proof of Brion's result requires generalizations of
the vanishing theorems of Kodaira and Kawamata--Viehweg and does not even
extend to positive characteristic. As a greater obstruction yet, there
are no extant tools to even connect $\Delta_w(M_1,M_2)$ to
$[\mathcal{Q}_{M_1}]$ and $[\mathcal{Q}_{M_2}]$.

We bridge this chasm by developing in \cref{sec:products} certain
tropical cell complexes that are dual to $\Delta_w(M_1,M_2)$, up to
omission of small faces. These complexes are constructed from
intersections of Chern classes of $[\mathcal{Q}_{M_1}]$ and
$[\mathcal{Q}_{M_2}]$; this method is entirely new and offers significant potential to be generalized.
Our main results can now be stated as follows:
\begin{Atheorem}\label{Athm:main intro}
  Let $(M_1,M_2)$ be a pair of matroids on $[n]$ without common loops
  and external activity complex $\Delta_w(M_1,M_2)$.
  \begin{enumerate}
  \item The finely graded $K$-polynomial of $\Delta_w(M_1,M_2)$ is
    bivaluative.
  \item $\Delta_w(M_1,M_2)$ is a Cohen-Macaulay simplicial complex.
  \item The
    $\mathbf{Z}^2$-graded $K$-polynomial is equal to
    \[
      \sum_{i,j} \textstyle\chi\left( \bigwedge^i
        [\mathcal{Q}^\vee_{M_1}]  \bigwedge^j
        [\mathcal{Q}^\vee_{M_2}] \right) (-U_1)^i (-U_2)^j\ .
    \]
  \end{enumerate}
\end{Atheorem}
The Cohen-Macaulay property in the second item
can be used to prove $K$-theoretic positivity for the $K$-polynomial
of $\Delta_w(M_1,M_2)$. The bivaluativity of the $K$-polynomial
in the first item can be used to connect the external
activity complex to the tautological bundles in \cref{Athm:geometric}, via \cref{Athm:initial ideal}, resulting in the third
item. Combining these results with Hochster's formula we arrive at the
following.
\begin{Atheorem}\label{thm:positive euler char formula}
    For any pair of
  matroids $(M_1,M_2)$, if $i+j = n$ then
  \begin{multline*}
        (-1)^{\rk(D(M_1,M_2))} {\textstyle\chi\left( \bigwedge^i [\mathcal{Q}^\vee_{M_1}]
            \bigwedge^j [\mathcal{Q}^\vee_{M_2}] \right)} \\
        = \sum_{B \in \binom{[n]}{i}} \dim_\CC \widetilde{H}_{ \rk(D(M_1,M_2)) -1 } ( \link_{\Delta_w(M_1,M_2)}( x_{[n] \setminus B} y_{B} ) ).
  \end{multline*}
%  Assume that $D(M_1,M_2)$ has rank $\rk(M_1) + \rk(M_2) -1$. If $i+j = n$ then ${\textstyle\chi\left( \bigwedge^i
%      [\mathcal{Q}^\vee_{M_1}]  \bigwedge^j
%      [\mathcal{Q}^\vee_{M_2}] \right)}$ is equal to 
%  \begin{align*}
%    (-1)^{\rk(D(M_1,M_2))} \sum_{B \in \binom{[n]}{i}} \dim_\CC \widetilde{H}_{ \rk(D(M_1,M_2)) -1 } ( \link_{\Delta_w(M_1,M_2)}( x_{[n] \setminus B} y_{B} ) ).
%  \end{align*}
\end{Atheorem}
We take reduced simplicial homology with complex valued coefficients, although rational
coefficients suffice. We believe, but cannot prove, that
$\Delta_w(M_1,M_2)$ is shellable, 
which would allow taking integer coefficients.

\subsection{The tropical $f$-vector conjecture}\label{ssec:f-vector conjecture}
Our approach to Conjecture~\ref{conj:f-vector conjecture} uses a key tool introduced by Speyer in \cite{speyer}.
This is a polynomial invariant of matroids $g_M(t) \in \ZZ[t]$ that has a suitable additivity in matroid polytope subdivisions, 
such that $[t^c]g_M(t)>0$ if the base polytope $P(M)$ of~$M$ has codimension~$c$.
If all coefficients of $g_M(t)$ were nonnegative, 
this fact would imply that $[t^c]g(M)$ bounds from above the number of codimension $c$ cells in a subdivision of~$P(M)$.
In particular, proving this non-negativity implies \Cref{conj:f-vector conjecture}. 

In \cite{speyer} Speyer defines $g_M(t)$ only when $M$ is realizable over~$\CC$, and shows its coefficients non-negative in this case.
In \cite[Section~8]{finkspeyer}, an explicit definition of $g_M(t)$ is given
for arbitrary matroids $M$, but no non-negativity is shown.
The definition there is equivalent to the following:
\[g_M(t)= (-1)^c \sum_i \chi\big(X_n, {\textstyle\bigwedge^i} [\mathcal{S}_M] \cdot
{\textstyle\bigwedge^i} [\mathcal{Q}_M^\vee] \big) (t-1)^i,\] 
where $c$ is the number of
connected components of $M$. We call this \textit{Speyer's $g$-invariant} of $M$.
Speyer, who as noted uses Kawamata--Viehweg vanishing to prove his positivity results, 
remarks that Brion's positivity results cannot be used this way.

Define $\omega(M)$ to be the coefficient of $t^{\rk(M)}$ in
$g_M(t)$. The following formula is immediate from the definition:
\begin{align*}
  \omega(M) &= (-1)^c \chi\big(X_n , {\textstyle\bigwedge^{\rk(M)}} [\mathcal{S}_M] \cdot {\textstyle\bigwedge^{\rk(M)}}[\mathcal{Q}^\vee_M]\big)\\
  &= (-1)^c \chi\big(X_n , {\textstyle\bigwedge^{n-\rk(M)}} [\mathcal{Q}^\vee_M] \cdot {\textstyle\bigwedge^{\rk(M)}}[\mathcal{Q}^\vee_M]\big),
\end{align*}
where $c$ is the number of connected components of $M$. This follows
by taking the $n$\/th exterior power of the relation
$[\mathcal{S}_M] + [\mathcal{Q}_M] = 1$ in $K_0(X_n)$. We note here that Brion's results \cite[Theorem~1]{brionpos}
\textit{can} prove positivity of $\omega(M)$
when $2\rk(M) \leq n$ and $M$ is both connected and realizable over
$\CC$. To see this apply \cref{Athm:geometric} and Brion's results to $Y_{L_1,L_2}/\!\!/(\CC^\times)^2 \subset \PP^{n-1} \times \PP^{n-1}$. The following
result shows that it is sufficient to understand $\omega(M)$ in order
to prove the $f$-vector conjecture.
\begin{theorem*}[Fink--Shaw--Speyer \cite{FSS}]
  If $\omega(N) \geq 0$ for all minors $N$ of~$M$ then all coefficients of $g_M(t)$ are non-negative.
\end{theorem*}
 We now resolve Speyer's conjecture affirmatively.
\begin{Atheorem}\label{thm:omega non-negative}
  For all matroids $M$, $\omega(M) \geq 0$, and thus Speyer's tropical $f$-vector conjecture is true.
\end{Atheorem}
\begin{proof}
  Say that $M$ has rank $r$. We may assume $M$ is loopless since otherwise $g_M(t) \equiv 0$. We may assume that $M$ is connected since $g_{M \oplus M'}(t) = g_M(t) g_{M'}(t)$. By \cref{Athm:main intro}, $\omega(M)$ is $(-1)^{n+c} = (-1)^{n+1}$ times the coefficient of $U_1^{r} U_2^{n-r}$ in $\K(\Delta_w(M,M), U_1,U_2)$. 
  We show in 
  \Cref{lem:D unexpected rank implies omega is zero} that this coefficient is zero when $M$ is connected and $D(M,M)$ has rank less than $2r-1$.
  When $D(M,M)$ has its maximal rank
  $2r-1$ then we apply \cref{thm:positive euler char formula}, which gives a formula for $(-1)^c \omega(M)$.
  The signs in the formula for $\omega(M)$ gather to $(-1)^{(2r-1)+1}=+1$ and we obtain
  \begin{align}\label{eq:omega non-negative formula}
    \omega(M) = \sum_{B \in \binom{[n]}{r}} \dim_\CC \widetilde{H}_{2r-2 } ( \link_{\Delta_w(M,M)}( x_{[n] \setminus B} y_{B} ) ),
  \end{align}
  which is manifestly non-negative.
\end{proof}

\subsection{Relation to previous work}
Weyman's so-called
\emph{geometric method} \cite[Chapter 5]{weyman} constructs the
minimal free resolution of a Kempf collapsing $Y_E$ that is normal
with rational singularities and birational to the bundle $E$.
Our primary goal in starting this work
was to use these tools to understand cohomology of tautological bundles of linear spaces
by studying free resolutions.
In a series of earlier works \cite{chow,moc,ratsing}
we had studied the collapsing of $\mathcal{S}_L^{\oplus (\dim L)}$,
but the combinatorics of the
collapsed varieties proved too difficult to understand at the time. 
We therefore turned to simpler bundles over the permutohedral variety, and
the first case of interest was the sum
$\mathcal{S}_L \oplus \mathcal{O}(-\beta)$. In this way, we were led
to rediscover Ardila and Boocher's construction of (the multicone of)
the Schubert variety of a linear space \cite{ardilaBoocher}. The jump
from realizable matroids to all matroids via an initial degeneration
was started by Ardila and Boocher, and was studied in depth by
Ardila, Castillo and Samper in \cite{acs}, who give a detailed presentation
of shelling orders of the associated Stanley-Reisner complex.

In a recent paper, Eur \cite{eur} considers the cohomology of
tautological bundles of linear spaces. He shows
that $\bigwedge^i \mathcal{S}_L$ and $\bigwedge^j \mathcal{Q}_L$ and
their duals, as well as $\Sym^\ell(\mathcal{Q}_L)$, have vanishing
higher cohomology over arbitrary fields. This extends in two ways our result
\cite[Theorem~5.1]{ratsing}, showing that any Schur functor applied to
$\mathcal{S}_L^\vee$ or $\mathcal{Q}_L$ has no higher cohomology for
$L \subset \CC^n$:
first by working over arbitrary fields, whereas we use
a characteristic zero hypothesis in a crucial way;
second by covering $\bigwedge^i \mathcal{S}_L$, which is not globally generated.
In this vein, one of our later
results, \cref{thm:higher cohomology}, gives a matroidal formula in terms of
$\Delta_w(M_1,M_2)$ for the higher cohomology of
$\bigwedge^i \mathcal{Q}_{L_1}^\vee \otimes \bigwedge^j
\mathcal{Q}_{L_2}^\vee$, where $L_i \subset \CC^n$. If one takes our
formula as \emph{the definition} of the higher cohomology of the
classes
$\bigwedge^i [\mathcal{Q}_{M_1}^\vee] \cdot \bigwedge^j
[\mathcal{Q}_{M_2}^\vee]$, which of course are not objects of a kind
for which cohomology groups are standardly defined, then we have a
bivaluative extension of the case when $(M_1,M_2)$ is a pair of
matroids realizable over $\CC$. Since it is unclear if
$H^p(X_n(k), \bigwedge^i \mathcal{Q}^\vee_L )$ is a matroid invariant
when $k$ is a field of positive characteristic, we cannot necessarily
recover or extend Eur's results.

The wonderful model $W_L$ of $L \subset \CC^n$ is a
compactification of $\PP L \cap \PP T$ inside the permutohedral variety
$X_n$. Our results can be used to give insight into tautological
bundles of linear spaces over $W_L$ and tautological classes of
matroids in $K_0(W_L)$ or $K_0(M)$ \cite{eurlarson, llpp}. We will
address these generalizations in a future work. Of additional interest
are the tautological bundles of matroids over the stellahedral
variety, introduced in \cite{ehl}. Motivated by this construction, the
first author and Morales introduced in \cite{BM} an \textit{augmented}
external activity complex of a matroid and proved its shellability. 

It is shown by Eur and Larson \cite[Theorem~5.4]{eurlarson} that
$\omega(M) \geq 0$ for matroids which are \textit{simplicially
  positive}. These are matroids $M$ for which a certain divisor class
$\mathcal{L}_{P(M^\perp)}$ in the Chow ring of $M$, derived from the
matroid base polytope of $M^\perp$, have a non-negative expansion in
the basis of simplicial generators. This includes infinitely many
non-realizable matroids, and thus covers many cases that Speyer's
result did not. \cite[Theorem~1.2]{eurlarson} gives a 
type of $K$-theoretic positivity that is different from the type
encountered in our work.

Ferroni and Schr\"oter \cite{ferroniSchroter} proved that every coefficient of $g_M(t)$ is non-negative when $M$ is a sparse paving matroid, providing substantial evidence for Speyer's tropical $f$-vector conjecture. In a recent work, Ferroni \cite{ferroni} gave an
explicit non-negative formula for each coefficient of Speyer's invariant
$g_M(t)$ when $M$ is a Schubert matroid. 
An interesting problem for future work is to try to combine the formula for $\omega(M)$ that Ferroni gives with our
general formula for $\omega(M)$ in \Cref{eq:omega non-negative
  formula}, perhaps by understanding a shelling of $\Delta_w(M_1,M_2)$.

In the last decade the development of matroid Hodge theory
(\cite{huh2,huh1} and a long series of works following)
has provided many positivity statements, especially log-concavity for sequences,
regarding Chern classes and Euler characteristics.
We contribute one application of this technology in \Cref{cor:log-concave}.
Building a stronger bridge between these techniques and ours ---
maybe by exploiting the similarities between our Kempf collapsings
and the biprojective bundles of \cite[\S9]{BEST} ---
serves as an interesting line of research for the future.

\subsection{Organization}
\Cref{sec:combinatorics,sec:intersection theory} are combinatorial and algebraic background
material, respectively. While they are not entirely comprehensive, our results
draw on a diverse background and what is familiar to one expert reader
may not be familiar to the next. With the background in place, in \cref{sec:D} 
we introduce and lay out the basic properties of the diagonal Dilworth
truncation of a pair of matroids and the external activity complex.  
We address in \cref{sec:schubert} the case of realizable pairs of matroids
with a study of the Schubert variety of a pair of linear spaces, proving
\cref{Athm:geometric} and \cref{Athm:initial ideal}. This material is
then set aside as we consider the intersections of Chern classes of
$[\mathcal{Q}_{M_1}]$ and $[\mathcal{Q}_{M_2}]$ in
\cref{sec:products}. The goal of this section is to relate the
intersection of these Chern classes to $\Delta_w(M_1,M_2)$ and prove
the first item in \cref{Athm:main intro}, which takes some effort. With this in place, the rest
of \cref{Athm:main intro} and the formula in \cref{thm:positive euler char formula} follow quickly in \cref{sec:consequences}.

\subsection*{Acknowledgements} Our approach and motivation for this
work began at an open problem session at the 2023 BIRS workshop
\textit{Algebraic Aspects of Matroid Theory}, following a problem
proposed by Matt Larson. We offer special thanks to Matt for proposing his problem and
for numerous helpful comments, corrections and insights later, including pointing
us to the work of Conca et al.\ on Cartwright-Sturmfels ideals and
insisting we improve earlier versions of our main theorems.  We thank
Chris Eur for an equal role to our own in the initial stages of this
project. We thank Kris Shaw, David Speyer and Primo\v{z} \v{S}kraba
for helpful conversations.

AB was partially supported as a Simons Fellow in Mathematics and by
the Charles Simonyi Endowment at the Institute for Advanced Study.
AF was supported by the Engineering and Physical Sciences Research Council (grant number EP/X001229/1).

\section{Combinatorial background}\label{sec:combinatorics}
In this section, we recall background material on simplicial
complexes, matroids and polytopes.
\subsection{Simplical complexes}\label{ssec:complexes}
A simplicial complex $\Delta$ on a finite set $E$ is a non-empty
collection of subsets of $E$ that is closed under taking subsets. A
subset $\sigma \subset E$ that is in $\Delta$ is called a
\newword{face} of $\Delta$ and a maximal face of $\Delta$ is called a
\newword{facet}. The dimension of a face $\sigma$ is
$\dim(\sigma) = |\sigma|-1$ and  the dimension of $\Delta$,
$\dim(\Delta)$, is the maximum dimension of a facet of $\Delta$.

Since the varieties we will eventually consider will
all be over the complex numbers, we do our algebraic setup
over the field $\CC$ now, with no later loss of applicability.
Let $A = \CC[x_e : e\in E]$. Define the \newword{Stanley-Reisner
  ideal} of $\Delta$ to be
$I(\Delta) = \left( \prod_{e \in \tau} x_e : \tau \textup{ is not a
  face of }\Delta\right)$. The \newword{Stanley-Reisner ring} of
$\Delta$ is $A/I(\Delta)$. The Krull dimension of this ring is
$\dim(\Delta)+1$.

This association of a
square-free monomial ideal to a simplicial complex, via its non-faces, is a
bijection referred to as the \textit{Stanley-Reisner correspondence}.
The faces of $\Delta$ can be recovered from $A/I(\Delta)$ as the
supports of the square-free monomials that are non-zero in the
quotient. We will later abuse notation and describe a simplicial
complex by listing a collection of square-free monomials in $A$ that
give the facets of the complex, the faces then corresponding to 
the divisors of the given monomials.

A simplicial complex $\Delta$ is said to be \newword{Cohen-Macaulay}
(over $\CC$) if its Stanley-Reisner ring is a Cohen-Macaulay ring. A
simplicial complex $\Delta$ is \newword{pure} if all its facets have
the same dimension. We will later need the well-known fact that every
Cohen-Macaulay simplicial complex is pure.

% See \cite[Corollary 5.1.5]{BH} for details.

The \newword{link} of a face $\sigma \in \Delta$ is the simplicial
complex $\link_\Delta(\sigma)$ on $E \setminus \sigma$ with faces
\[ \{ \tau \subset E: \tau \cup \sigma \in
  \Delta, \tau \cap \sigma = \emptyset\}.\]
One can check that a (graded) ring is Cohen-Macaulay in many 
different-looking ways. We summarize here the results we will need.
\begin{theorem}\label{thm:cm conditions}
  Let $I$ be a homogeneous ideal in a complex polynomial ring
  $A$ in $n$ variables with the usual $\ZZ$-grading, and
  assume that $A/I$ has dimension $r$. The following are equivalent.
  \begin{enumerate}
  \item $A/I$ is a Cohen-Macaulay ring.
  \item The minimal free resolution of the $A$-module $A/I$ has length
    $n-r$.
  \item There is a sequence $y_1,\dots,y_r \in A$ of homogeneous
    elements forming a regular sequence on $A/I$.
  \item There is a sequence $y_1,\dots,y_m \in A$ of homogeneous
    elements forming a regular sequence on $A/I$ with
    $A/\left(I + \left( y_1,\dots,y_m\right)\right)$ a Cohen-Macaulay ring.
  \end{enumerate}
  If $I=I(\Delta)$ is a Stanley-Reisner ideal then we
  may add the following.
  \begin{enumerate}
  \item[(5)] (Reisner's criterion) For all $\sigma \in \Delta$ and all $i < \dim \link_\Delta({\sigma})$ we have
    \[ \widetilde{H}_i(\link_\Delta(\sigma), \CC)=0.\]
  \end{enumerate}
\end{theorem}
The first four items occur in \cite[Theorem~13.37]{millerSturmfels} and the fifth occurs as \cite[Theorem~5.53]{millerSturmfels}.

The minimal free resolution of the Stanley-Reisner ring $A/I(\Delta)$
as an $A$-module is an object of fundamental importance in this
paper. The structure of this resolution is known to be governed by the topology of the simplicial complex $\Delta$. For us, this will be most clearly stated in terms of the
\newword{Alexander dual} of $\Delta$, which is the simplicial complex
$\Delta^\vee$ whose faces are the complements of non-faces of
$\Delta$.
\begin{theorem}[Hochster's formula, dual version]\label{thm:hochster}
  The minimal free resolution of $A/I(\Delta)$ as a $\ZZ^E$-graded
  $A$-module,
  \[
    \dots \to F_2 \to F_1 \to F_0 \to A/I(\Delta) \to 0,
  \]
  has
  $F_i = \displaystyle\bigoplus_{\sigma\, :\, \overline{\sigma} \in \Delta^\vee}
  A(-{\sigma})^{\oplus \beta_{i,\sigma}(A/I(\Delta))}$ where
  $\overline{\sigma}$ is the complement of $\sigma$ in $E$,
  $A(-{\sigma})$ is the free $A$-module on one generator in degree
  $-\sum_{j \in \sigma} e_j \in \ZZ^E$ and
  \[
    \beta_{i,\sigma}(A/I(\Delta)) = \dim_\CC \widetilde{H}_{i-2}( \link_{\Delta^\vee}(\overline{\sigma}),\CC).
  \]
\end{theorem}
The
numbers $\beta_{i,\sigma}(A/I(\Delta))$ are the \newword{finely graded
  Betti numbers} of $\Delta$; here we have given  $A$ the \textit{fine grading}, where  $\deg(x_i) = e_i \in \ZZ^E$. 
The \newword{$K$-polynomial} of $\Delta$ in a given positive multigrading
is the unique polynomial $\mathcal{K}(\Delta)$ for which we have the relationship
$\Hilb(A/I(\Delta)) = \mathcal{K}(\Delta) \Hilb(A)$
among Hilbert series in that multigrading. 
In this paper we will use several formulas for this polynomial or various specializations coming from coarsening the fine grading. Here are two formulas for the finely graded $K$-polynomial of a Stanley-Reisner ring.
\begin{proposition}\label{prop:K poly simplicial complex}
  For a simplicial complex $\Delta$ on $E$, the finely graded $K$-polynomial of $\Delta$ is given by
  \begin{align*}
    \mathcal{K}(\Delta) &= \sum_{\sigma \in \Delta} \prod_{j \in \sigma} T_j \cdot \prod_{j \notin \sigma} (1-T_j) \\
    &= \sum_i \sum_{\sigma\, :\, \overline{\sigma} \in \Delta^\vee} (-1)^i \beta_{i,\sigma}(A/I(\Delta)) \prod_{j \in \sigma} T_j  \in \ZZ[ T_j^{\pm 1} : j \in E].
  \end{align*}
\end{proposition}
The first expression follows by considering the $\CC$-basis of monomials for $A/I(\Delta)$ 
and collecting its elements according to their support,
while the second comes from the fact
that the Hilbert series can be computed from a finely graded free
resolution of $A/I(\Delta)$ \cite[Chapter~1]{millerSturmfels}. We will
write $\K(\Delta)$ and $\K(\Delta;T_i)$ for the $K$-polynomial of
$\Delta$ depending on when we want to emphasize the grading.

In general, there can be cancellation in both formulas for
$\mathcal{K}(\Delta)$. However, when $\Delta$ is Cohen-Macaulay there
is a form of positivity for its $K$-polynomial.
% \begin{proposition}\label{prop:CM implies K theory pos}
%   Let $\Delta$ be a simplicial complex on $E$ whose Alexander dual
%   $\Delta^\vee$ is Cohen-Macaulay. Then we  have a
%   cancellation-free formula,
%   \[
%     \mathcal{K}(\Delta) = \sum_{{\sigma} \in \Delta^\vee} (-1)^{\dim(\link_\Delta({\sigma}))} \dim_\CC \widetilde{H}_{\dim(\link_{\Delta^\vee}({\sigma}))}(\link_{\Delta^\vee}({\sigma})) \prod_{j \notin \sigma} T_j.
%   \]
%   In particular, given any monomial of total degree $m$,
%   its coefficient in $\mathcal{K}(\Delta)$ has sign $(-1)^{\dim(\Delta^\vee)-|E|+m}$.
% \end{proposition}
% \begin{proof}
%   This first part is simply a restatement of Hochster's fomula using Reisner's criterion. The second statement follows from the first since  a monomial of total degree $m$ appearing  in the $K$ polynomial comes from a face of $\sigma \in \Delta^\vee$ whose cardinality is $|E|-m$. The dimension of $\link_{\Delta^\vee}(\sigma)$ is then $\dim(\Delta^\vee) + 1 - (|E|-m)-1$.
% \end{proof}

\begin{proposition}\label{prop:CM implies K theory pos}
  Let $\Delta$ be a Cohen-Macaulay simplicial complex on $[n]$. Then we  have a
  cancellation-free formula exhibiting $K$-theoretic positivity:
  \[
    \mathcal{K}(\Delta; 1-T_1,\dots,1-T_n) = \sum_{{\sigma} \in \Delta} (-1)^{\dim(\link_\Delta({\sigma}))} \dim_\CC \widetilde{H}_{\dim(\link_{\Delta}({\sigma}))}(\link_{\Delta}({\sigma})) \prod_{j \notin \sigma} T_j.
  \]
%  In particular, given any monomial of total degree $m$,
%  its coefficient in $\mathcal{K}(\Delta;1-T_i)$ has sign $(-1)^{\dim(\Delta^\vee)-n+m}$.
\end{proposition}
\begin{proof}
  It follows from the Alexander inversion formula
  \cite[Theorem~5.14]{millerSturmfels} that
  $\mathcal{K}(\Delta;1-T_i) = \mathcal{K}(I(\Delta^\vee);T_i)$. Now
  compute the $K$-polynomial of $I(\Delta^\vee)$ using Hochster's
  formula, relying on Reisner's criterion to see that the homology of
  the links in $\Delta$ vanish except in top dimension.
\end{proof}

\subsection{Matroids} 
A matroid $M$ is given by a simplicial complex $\Delta(M)$ on a
non-empty set $E$ whose faces are called \newword{independent sets}
and satisfy the following property: For all $I,J \in \Delta(M)$,
\[
  |I| < |J| \implies \textup{there is }e \in J \setminus I \textup{ with } I \cup \{e\} \in \Delta(M).
\]
We will use many equivalent descriptions of a matroid in what follows.
It is worth emphasizing that the matroid $M$ encompasses more than
simply its collection of independent sets, and thus we really do want
to maintain a separate notation for this complex.
\begin{example}
  Let $L \subset k^n$ be a linear space with $k$ a field. The
  coordinate hyperplanes in $k^n$ are the zero loci of linear functionals
  $e_1^\perp,\dots,e_n^\perp$, which we can restrict to~$L$. The
  \newword{matroid of $L$}, denoted $M(L)$, has ground set $[n]=\{1,\ldots,n\}$ and
  $I \subset [n]$ is independent when $\{ e_i^\perp|_L : i \in I\}$ is
  a linearly independent set. A matroid occurring in this way is said
  to be \textit{realizable over $k$}. Most matroids are not realizable
  over any field. In spite of this, special consideration of matroids
  realizable over $\CC$ will be a critical part of our methods.
\end{example}

The facets of $\Delta(M)$ are called \newword{bases} of $M$. A non-face of
$\Delta(M)$ is called a \newword{dependent set} of $M$ and a minimal
non-face of $\Delta(M)$ is called a \newword{circuit} of $M$. If
$e \in E$ is in every basis of $M$ then $e$ is called a
\newword{coloop} of $M$ and if $e$ is in no basis of $M$ then $M$ is
called a \newword{loop}. Note that a coloop appears in no circuit of~$M$. 
Two elements of $M$ are \newword{parallel} if they form a two element
circuit. 
%If $M'$ is obtained from $M$ by adding new elements all of
%which are parallel to elements of $M$, then $M'$ is a \newword{parallel
%extension} of $M$. If $M^\perp$ is a parallel extension of $(M')^\perp$
%then $M'$ is said to be a \newword{series extension} of~$M$.
A matroid is said to be \newword{connected} if every pair of
elements is contained in a circuit. The \newword{direct sum} of two
matroids is their simplicial join, and a matroid is connected if and
only if it cannot be written as a direct sum. Every matroid $M$ can be
written uniquely as a direct sum
\[
  M_1 \oplus \dots \oplus M_c
\]
where each $M_i$ is connected, the summands being refered to as the
\newword{connected components} of $M$.

If $I$ is independent and
$e \notin I$, then either $I \cup \{e\}$ is independent or
$I \cup \{e\}$ contains a unique circuit involving $e$. This is called
the \newword{fundamental circuit} of $e$ and denoted $C_{I,e}$.

Every matroid $M$ is endowed with a \newword{rank} function
$\rk_M : 2^E \to \mathbf{N}$ with $\rk_M(S)$ equal to the size of the
largest independent set in $S$. 
The value $\rk(M):=\rk_M(E)$ is called the \newword{rank of $M$},
and $\operatorname{corank}(M):=|E|-\rk(M)$ the \newword{corank of~$M$}.
A \newword{flat} of $M$ is a subset of $E$ that is maximal of a given
rank.  Given a matroid $M$ on $E$, the \newword{dual matroid}
$M^\perp$ also has ground set $E$, and its bases are
the complements of the bases of~$M$.

The \newword{contraction} of $M$ by an independent set $I$ is the
matroid $M/I$ whose independent set complex is
$\link_{\Delta(M)}(I)$. More generally, if $S \subset E$ is any
subset, the contraction $M/S$ has ground set $E \setminus S$ and independent sets given by
$J \setminus S$ where $J \in \Delta(M)$ has $|J \cap S| = \rk_M(S)$.
The \newword{restriction} of $M$ to~$S$ is the matroid $M|S$ on~$S$ with $\Delta(M|S)=\{I\in\Delta(M):I\subseteq S\}$.

% Two elements of $M$ are \newword{parallel} if they form a two element
% circuit. If $M'$ is obtained from $M$ by adding new elements all of
% which are parallel to elements of $M$, then $M'$ is a \newword{parallel
% extension} of $M$. If $M^\perp$ is a parallel extension of $(M')^\perp$
% then $M'$ is said to be a \newword{series extension} of~$M$.

If $M$ and $N$ are matroids on $E$,
then $M$ is a \newword{relaxation} of~$N$ if $\Delta(M)\supseteq\Delta(N)$, 
equivalently if $\rk_M(S)\ge\rk_N(S)$ for every $S\subseteq E$.
%(The same relation is also known as $N$ being a \newword{weak image} of~$M$.)

Provan and Billera \cite[Theorem 3.2.1]{provanBillera}
proved that $\Delta(M)$ is vertex-decomposable and hence shellable for any matroid~$M$.
These properties of $\Delta(M)$ imply the following, which will be important in our main results:
\begin{proposition}\label{prop:matroid is cm}
  For any matroid $M$, $\Delta(M)$ is a Cohen-Macaulay complex.
\end{proposition}

\subsubsection{Basis activities}\label{ssec:basis activities} Let $M$ be a matroid on $E$, and assume
that $E$ is totally ordered by $\prec$. Given an independent set $I$
of $M$, we say that $e \notin I$ is \newword{externally active} for
$I$ if $e$ is the minimum element of a circuit in $I \cup \{e\}$ under
$\prec$. If $B$ is a basis of $M$, then every element $e \notin B$ is
in a circuit of $B \cup \{e\}$ and we obtain a partitioning of
$E \setminus B$ as those elements that are externally active for $B$,
and those that are not (sometimes said to be externally passive for $B$). 
The \newword{external activity} statistic of a basis $B$ of~$M$ 
is the number of externally active elements for~$B$; 
for instance $B$ is said to have \newword{external activity
  zero} if no element of $E \setminus B$ is externally active.

There is a dual notion called internal activity for bases: $e \in B$
is \newword{internally active} if $e$ is externally active for the
basis $E \setminus B$ of $M^\perp$. One uses the same ordering on $E$
for both matroids $M$ and $M^\perp$.

These notions can be used to describe the finely graded
Betti numbers of $\Delta(M)$, first appearing in a result of Stanley \cite{stanley}. For $\mathbf{b} \in \ZZ^n$, we denote the finely graded Betti
numbers of $A/I(\Delta(M))$ by $\beta_{i,\mathbf{b}}(M)$ to ease the
proliferation of parentheses.
\begin{proposition}\label{prop:matroid betti numbers}
  Let $M$ be a matroid on $[n]$ with rank $r$. Then $\beta_{i,\mathbf{b}}(M)$
  is zero unless $\mathbf{b} \in \{0,1\}^n$ is the indicator vector of the complement of a flat $F$ of $M^\perp$. In this case, the
  value of the Betti number $\beta_{\rk(M^\perp/F),\mathbf{b}}(M)$ is the number of bases of $M^\perp/F$ that
  have external activity zero.
\end{proposition}
The following result is worth emphasizing for later use.
\begin{corollary}\label{cor:matroid betti numbers}
  For any matroid $M$ on $[n]$ of rank $r$, $\beta_{i,\mathbf{1}}(M)$ is zero unless
  $i = n-r$.
\end{corollary}
\begin{proof}
  If $\mathbf{b}=\mathbf{1}$ then $F = \emptyset$ and $\rk(M^\perp) = n-r$.
\end{proof}

\subsection{Polytopes and valuations}\label{ssec:polytopes}
Let $N_\RR$ be a real vector space containing a privileged full rank lattice $N$.
We call a polyhedron $\sigma$ in $N_\RR$ \newword{rational} if
it has an inequality description 
\begin{equation}\label{eq:Ax<=b}
\sigma=\{x\in N_\RR:\langle x,a_i\rangle\le b_i\mbox{ for }i=1,\ldots,s\},
\end{equation}
where  $a_j\in\operatorname{Hom}(N,\ZZ)$ for all $1 \leq j \leq s$.
That is, we do not insist that the $b_i$ be rational; if $\sigma$ is a cone this makes no difference.
A fan in $N_\RR$ is \newword{rational} if all its cones are, 
\newword{complete} if its support is the whole of~$N_\RR$,
and \newword{unimodular} if each of its cones is generated by a subset of a lattice basis of~$N$. 

Let $M$ be a matroid on $[n]$ of rank $r$. The \newword{independent
  set polytope} of $M$, denoted $IP(M)$, is the convex hull in $\RR^n$
of the indicator vectors $e_I = \sum_{i \in I} e_i \in \RR^n$ of independent
sets of $M$. The \newword{base polytope} of $M$, denoted $P(M)$, is
the intersection of $IP(M)$ with the hyperplane where the sum of the
coordinates is equal to $r$.

\subsubsection{The permutohedron and its fan}
The \newword{permutohedron} $\Pi_n$ is the convex hull in
$\mathbf{R}^n$ of all $n!$ permutations of the vector
$(0,1,\dots,n-1) \in \RR^n$. 
The normal fan of~$\Pi_n$ is the fan in~$\RR^n$ 
whose cones are the regions of the $n$\/th braid arrangement,
this being the arrangement consisting of the $\binom{n}{2}$
hyperplanes $\{x \in \RR^n:x_i=x_j\}$,
$1 \leq i < j \leq n$. 
The cones of the normal fan of~$\Pi_n$ all have lineality space $\RR\mathbf{1}$,
where $\mathbf{1}=(1,\ldots,1)\in\mathbf{R}^n$:
we describe these cones in more detail in a moment.
Let $\Sigma_n$ be the quotient fan in $\RR^n/\RR\mathbf{1}$.  
We call $\Sigma_n$ the \newword{permutohedral fan}.

The permutohedral fan $\Sigma_n$ is a complete rational unimodular fan over the lattice
$\ZZ^n/\ZZ\mathbf{1}$. For any matroid on $[n]$, the normal fan of
$P(M)$ is refined by the normal fan of~$\Pi_n$. 
To ease notation, we will often write $\RR^n/\RR$ instead of $\RR^n/\RR \mathbf{1}$.

Two different but equivalent families of indexing objects for the cones of $\Sigma_n$ will be useful to us.
Each cone in $\Sigma_n$ has the form
\[\sigma_\preceq = \{x+\RR\mathbf{1}\in\RR^n/\RR : x_i\le x_j\mbox{ whenever }i\preceq j\}\]
for a total preorder $\preceq$ on~$[n]$,
i.e.\ a reflexive transitive binary relation on~$[n]$ such that
$i\preceq j$ or $j\preceq i$ or both for any $i,j\in[n]$. 
Total preorders on~$[n]$ are placed in bijection with chains of sets of the shape
\[S_\bullet : \emptyset = S_0 \subsetneq S_1 \subsetneq \cdots \subsetneq S_{\ell(S_\bullet)} = [n]\]
by the rule
\[\{S_1,\ldots,S_{\ell(S_\bullet)}\} = \{\{j:j\succeq i\} : i\in[n]\}.\]
If $\preceq$ maps to $S_\bullet$ under this bijection we write $S_\bullet(\preceq):=S_\bullet$
and define the synonym $\sigma_{S_\bullet}=\sigma_\preceq$.

\subsubsection{Valuations}
Let $\mathbf 1_S$ denote the indicator function of a subset $S\subseteq V$ of a real vector space $V$.
Given a set $\mathcal S$ of subsets of~$V$
and an abelian group $G$,
let $\mathbb I(\mathcal S)$ be the subgroup of functions $V\to\mathbb Z$
generated by the $\mathbf 1_S$ for $S\in\mathcal S$.
A function $f:\mathcal S\to G$ is a \newword{valuation}, or is \newword{valuative},
if there exists a map of abelian groups $\widehat f:\mathbb I(\mathcal S)\to G$
such that $f(S) = \widehat f(\mathbf 1_{S})$ for all $S\in\mathcal S$.
See \cite[Appendix A.1]{ehl} for a survey of how this definition relates to other definitions of ``valuation'' that have been used.

We will be concerned with valuations on matroids, which we interpret through their base polytopes.
If $\mathrm{Mat}_E$ is the set of base polytopes of matroids on~$E$,
and $f:\mathrm{Mat}_E\to G$ is a valuation,
then we also call the matroid function $M\mapsto f(P(M))$ a valuation.

Linear relations among indicator functions 
imply linear relations among the values of any valuation.
For matroids, we will apply this implication to the relation \Cref{prop:Schubert expansion} below,
an expansion of an arbitrary matroid base polytope into base polytopes of Schubert matroids,
which we therefore now introduce.

Given a chain of subsets $S_\bullet : \emptyset = S_0 \subsetneq S_1 \subsetneq \cdots \subsetneq S_k = E$, 
and a vector $a_\bullet = (0, a_1, \ldots, a_k)$ of non-negative integers of the same length $k=:\ell(S_\bullet)$,
the \newword{Schubert matroid} $\Omega(S_\bullet, a_\bullet)$ is the matroid on~$E$
in which a set $I$ is independent if and only if  
$|I \cap S_i| \leq a_i$ for all $i=0,\ldots,k$.
Note that Schubert matroids with distinct indexing data $(S_\bullet, a_\bullet)$ may be equal.
All Schubert matroids are realizable over any sufficiently large field, including over~$\CC$.

For $M$ a matroid on $E$, and $S_\bullet$ as above, write
\[\rk_M(S_\bullet)=(0=\rk_M(S_0), \rk_M(S_1), \ldots, \rk_M(S_k)).\]

\begin{proposition}[{\cite[Theorem 4.2]{df}}]\label{prop:Schubert expansion}
For any matroid $M$ on~$E$,
\[\mathbf 1_{P(M)} = \sum_{S_\bullet} (-1)^{|E|-|S_\bullet|} \ \mathbf 1_{P(\Omega(S_\bullet, \rk_M(S_\bullet)))}.\]
\end{proposition}
Hampe \cite[Theorem 3.12]{hampe} gives the coefficients after equal terms in the above sum are collected.
We will not need these coefficients, but only properties of the matroids appearing,
firstly their realizability:
\begin{corollary}\label{cor:realizable determines valuation}
Let $f$ and $g$ be matroid valuations.
If $f(M)=g(M)$ when $M$ is a Schubert matroid ---
or, \emph{a fortiori}, when $M$ is realizable over~$\CC$ ---
then $f(M)=g(M)$ for all $M$.
\end{corollary}

%We also need that they are relaxations.
%\begin{lemma}\label{lem:Schubert summands are relaxations}
%For $M$ and $S_\bullet$ as above,
%the matroid $\Omega(S_\bullet, \rk_M(S_\bullet))$ is a relaxation of~$M$, 
%and the two have equal rank.
%\end{lemma}
%\begin{proof}
%Write simply $\Omega$ for the Schubert matroid in the statement.
%Since independent sets form a simplicial complex, if $I$ is independent in~$M$ then so is $I\cap S_i$. 
%Thus the inequalities $|I\cap S_i|\le\rk_M(S_i)$ defining $\Delta(\Omega)$ are true for every $I\in\Delta(M)$ by definition of rank,
%i.e.\ $\Delta(M)\subseteq\Delta(\Omega)$.
%The inequality for $i=|S_\bullet|$, with $S_i=E$, shows that $\rk(\Omega)\le\rk(M)$,
%and the opposite inequality follows from being a relaxation.
%\end{proof}

We will also need the fact that the set of loops is constant in linear relations.

\begin{lemma}\label{lem:loops in subdivisions}
There is a direct sum decomposition
\[\mathbb I(\mathrm{Mat}_E) = \bigoplus_{S\subseteq E} G_S\]
such that, for each matroid $M$ on~$E$,
$\mathbf 1_{P(M)}\in G_S$ where $S$ is the set of loops of~$M$.
\end{lemma}

\begin{proof}
An element $i\in E$ is a loop of~$M$ if and only if $P(M)$ is contained in the coordinate hyperplane $\{x\in\RR^E:x_i=0\}$.
In a subdivision of $P(M)$ into matroid polytopes,
no internal cell can lie in a hyperplane $\{x_i=0\}$ which does not contain $P(M)$ itself.
This implies the lemma by \cite[Theorem 3.5]{df}.
\end{proof}

\section{Intersection theory}\label{sec:intersection theory}
In this section we describe the equivariant $K$-theory and Chow rings
that we will use, including the relevant tropical intersection theory. We  also describe the tatutological
bundles and classes of matroids in these $K$ and Chow rings.
\subsection{Equivariant $K$-theory}\label{ssec:K}
Let $X$ be a complex algebraic variety. That is, $X$ is an integral
scheme of finite type over $\CC$. Let $G = (\CC^\times)^m$ be an
algebraic $m$-torus which acts on~$X$. Let $K^G_0(X)$ denote the
Grothendieck group of $G$-equivariant coherent sheaves on $X$, and
$K^0_G(X)$ denote the Grothendieck group of $G$-equivariant vector
bundles on $X$. 
The latter group has a ring structure from the tensor product of vector bundles. 
Assume henceforth that $X$ is smooth. This means the
natural map $K^0_G(X) \to K_0^G(X)$ taking a vector bundle to its
sheaf of sections is an isomorphism of groups. In this way $K^G_0(X)$
is a ring. When $G$ is the trivial group we will simply write
$K_0(X)$.

If $X$ is a point then $K^G_0(X)$ can be identified with the
representation ring of $G$, denoted $R(G)$. Since $G$ is a torus, 
$R(G)$ is the Laurent polynomial ring
\[
  R(G) = \ZZ[ \operatorname{Hom}(G,\CC^\times) ] = \ZZ[U_1^{\pm 1},\dots,U_m^{\pm1}].
\]

If $f: X \to Y$ is a $G$-equivariant proper map then there is a
pushforward $f_*:K^G_0(X) \to K^G_0(Y)$ given by
$f_*[\mathcal{F}] = \sum_i (-1)^i [R^if_* \mathcal{F}]$. When $Y$ is a
point this map is the equivariant Euler characteristic, denoted
$\chi^G : K^G_0(X) \to R(G)$, or simply $\chi: K_0(X) \to \ZZ$ when
$G$ is trivial.

If $H \to G$ is a group homomorphism, there is a natural restriction
map $K^G_0(X) \to K^H_0(X)$ \cite[2.2.1]{handbook}. This map views a
$G$-equivariant sheaf as an $H$-equivariant sheaf in the natural
way. Since all our groups are tori, when $H$ is the trivial group we
have a \textit{surjective} restriction map $K^G_0(X) \to K_0(X)$. 

\subsubsection{The permutohedral variety}
Recall that $\Sigma_n$ is the permutohedral fan from
\cref{ssec:polytopes}, which is a rational, complete, unimodular fan
in $\RR^n/\RR$. Let $X_n$ denote the associated toric variety, which
we call the \newword{permutohedral variety}. The torus
$T = (\CC^\times)^n$ acts on $X_n$ with a dense orbit; however, it is
not the dense torus of $X_n$ since it acts with a stabilizer
$\{(s,\dots,s) : s \in \CC^\times \}$. Let $T'$ denote the dense torus
of $X_n$, which we can identify with $T/\CC^\times$, or with those
$t \in T$ satisfying $t_1 = 1$. We will consider both $K^T_0(X_n)$ and
$K_0(X_n)$ in the sequel.

The fixed points of the action of $T$ on $X_n$ are parametrized by
pemutations $\pi \in S_n$. Since, additionally, $T$ acts on $X_n$ with finitely many
fixed curves, we may apply the theory of equivariant
localization to compute $K^T_0(X_n)$. Fix the notation
$R(T) = \ZZ[T_1^{\pm 1},\dots,T_n^{\pm 1}]$
for elements of the representation ring. Then we may identify
$ K^T_0(X_n) $ as the subring of the direct sum of rings
\[
  \bigoplus_{\pi \in \mathfrak{S}_n} \ZZ[T_1^{\pm 1},\dots,T_n^{\pm 1}]
\]
consisting of those tuples $f = (f_\pi)$ satisfying
$f_\pi - f_{\pi'} \in \left( 1- T_{\pi(i+1)}/T_{\pi(i)} \right)$
whenever $\pi$ and $\pi' = \pi \circ (i,i+1)$ differ by an adjacent
transposition. For a class $\xi \in K^T_0(X_n)$, the local class
$\xi_\pi \in \ZZ[T_1^{\pm 1},\dots,T_n^{\pm n}]$ is the pullback of
$\xi$ along the inclusion of the $T$-fixed point of $X_n$ indexed by
$\pi$.

Over $X_n$ we have a trivial bundle with fiber $\CC^n$, denoted
$\underline{\CC^n} = \CC^n \times X_n$. We endow this with an action
of $T$ that inversely scales coordinates of $\CC^n$.  We will always
consider $\CC^n$ and {$\underline{\CC^n}$} with this action in the
sequel. This agrees with the conventions in \cite{BEST}.

The map $K^T_0(X_n) \to K_0(X_n)$ is surjective, and the kernel
consists of the elements whose local class at $\pi$ has the form
$\xi_\pi(T_1,\dots,T_n) - \xi_\pi(1,\dots,1)$ for all~$\pi$,
where $\xi$ is a class in $K^T_0(X_n)$. 
See \cite[Theorem~2.1]{BEST} for details and reference to more general results.
\subsubsection{Matrices with two rows}
Let $X = \CC^n \times \CC^n$, which we can think of as the space of $2\times n$ matrices. 
Let $S = (\CC^\times)^2$ and let $T$ continue to be the
$n$-torus above.  An element $(s_1,s_2) \in S$ acts on $X$ by scaling
the row $i$ of a matrix by the inverse $s_i^{-1}$, and an element
$t = (t_1,\dots,t_n)\in T$ acts by scaling the $j$th column of a
matrix by $t_i^{-1}$. These actions commute and hence there is an
action of the product $S \times T$ on $X$. There is also an action of
the torus $G = (\CC^\times)^{2n}$ on $X$ by inverse scaling of the
individual matrix entries. Since $X$ is an affine space, $X$
has the equivariant $K$-theory of a point. In spite of these trivial
$K$-rings, we will need to push other $K$-theoretic computations to the
$K$-theory of~$X$, necessitating this perspective.

We take this moment to fix further notation, with
\begin{align*}
  R(S) &= \ZZ[U_1^{\pm 1},U_2^{\pm 1}]\\
  R(G) &= \ZZ[T_{i,j}^{\pm 1} : 1 \leq i \leq 2, 1 \leq j \leq n].
\end{align*}
We have natural group homomorphisms $S \times T \to G$,
$S \hookrightarrow S \times T$ and $T \hookrightarrow S \times T$ and these give rise to
restriction maps in $K$-theory
\[
  \xymatrixcolsep{3em}\xymatrix{
    & & K^S_0(X)\\
    K^G_0(X) \ar[r]^-{T_{i,j} \mapsto U_i T_j} & K^{S \times T}_0(X) \ar[ru]^{T_j \mapsto 1} \ar[rd]_{U_i \mapsto 1}&  \\
    & & K^T_0(X)
  }
\]

\subsubsection{Multigraded $K$-polynomials and multidegrees}
We continue with $X = \CC^n \times \CC^n$, acted on by any of the tori previously described. 
Denote the coordinate ring of~$X$ by $A = \CC[x_1,\dots,x_n,y_1,\dots,y_n]$,
with the $x$-variables the coordinates on the first factor and
the $y$-variables the coordinates on the second.

The torus actions defined on $X$ correspond to different multigradings
on $A$:
\begin{enumerate}
\item The action of $G$ on $X$ gives rise to a $\ZZ^n \times \ZZ^n$
  grading on $A$ --- the fine grading.  Here $x_i$ has degree
  $\deg_{\ZZ^n \times \ZZ^n}(x_i) = (e_i,0) \in \ZZ^n \times \ZZ^n$ and $y_j$ has degree
  $\deg_{\ZZ^n \times \ZZ^n}(y_j) =(0,e_j) \in \ZZ^n \times \ZZ^n$.
\item The action of $S$ on $X$ gives rise to a $\ZZ^2$ grading on $A$,
  where each $x_i$ has degree $\deg_{\ZZ^2}(x_i)=(1,0) \in \ZZ^2$ and each $y_j$ has
  degree $\deg_{\ZZ^2}(y_j) = (0,1) \in \ZZ^2$.
\item The action of $T$ on $X$ gives rise to a $\ZZ^n$ grading on $A$,
  where $x_i$ and $y_i$ both have degree $\deg_{\ZZ^n}(x_i) = \deg_{\ZZ^n}(y_i) = e_i \in \ZZ^n$.
\item The action of $S \times T$ on $X$ gives a $\ZZ^2 \times \ZZ^n$
  grading, which is the product of the two previous gradings. 
  Its degree function is denoted $\deg_{\ZZ^2 \times \ZZ^n}$, as
  expected.
\end{enumerate}
A torus equivariant coherent sheaf on $X$ is equivalent to a
correspondingly multigraded $A$-module. Any finitely generated
multigraded $A$-module $M$ has a multigraded free resolution where the
free modules involved are copies of $A$ up to a global shift in the
multidegree in each copy. The multigraded Hilbert series of $M$ can
thus be written as a rational function $\K(M)\cdot \Hilb(A)$ where
$\K(M)$ is a Laurent polynomial called the \newword{multigraded
  $K$-polynomial} of $M$. This polynomial lives in the appropriate
equivariant $K$-theory ring.
The multigraded $K$-polynomial of a simplicial complex 
from \Cref{ssec:complexes}
is indeed a case of the present definition, with $M=A/I(\Delta)$.
\begin{example}
  We have made our conventions so that the $\ZZ^2 \times \ZZ^n$-graded
  Hilbert series of~$A$ is
  $1/\prod_{i=1}^2 \prod_{j=1}^n (1-U_i T_j)$. This is the also the character of the $S \times T$-module $A$. The $K$-polynomial of~$A$
  is the numerator here, which is the constant polynomial $1$ as expected.
\end{example}
Generally, the indeterminates visible in the $K$-polynomial will allow the reader to recover 
which $K$-theory a particular class belongs to
(equivalently, which multigrading is being used). Occasionally, we
want to emphasize this and write $\K(M;U_1,U_2)$ to signify that
$\K(M;U_1,U_2) \in K^S_0(X)$, etc.

The $K$-polynomial of a multigraded $A$-module $M$ (say, in one of the
above multigradings) can be used to determine the
\newword{multidegree} of~$M$, denoted $\mathcal{C}(M)$, via the following procedure:
Replace each indeterminate $Z$ in its $K$-polynomial with $1-Z$, simplify
the resulting Laurent polynomial and then gather the terms of lowest
remaining degree (which will be the codimension of $M$) to give $\mathcal{C}(M)$. See \cite[Definition~8.45]{millerSturmfels}. 
When $M$ is the Stanley-Reisner ring of a pure simplicial complex we have the following.

\begin{proposition}\label{prop:multidegree}
  Let $I \subset A$ be a square-free monomial ideal whose associated Stanley-Reisner complex $\Delta$ is pure. Then the finely graded multidegree $\mathcal{C}(A/I)$ is the facet-complement enumerator of $\Delta$.
\end{proposition}
\begin{proof}
  This follows from the formula for the $K$-polynomial in \cref{prop:K poly simplicial complex}, along with the general
  procedure described above to extract the multidegree from the $K$-polynomial.
\end{proof}

\subsubsection{Properties conveyed in collapsings}
In this subsection we consider the situation of a Kempf collapsing
from the introduction. Recall that we had a subbundle $E$ of a trivial
bundle $\CC^m \times X$ over a projective variety $X$, and the associated
quotient bundle $F = \CC^m/E$. The two results we give here 
motivate a significant part of our work, most clearly evident in \cref{thm:Kpolyrealizable} and \cref{cor:int cc}.

Let $\hat Y_E$ denote the image of $E$ under the projection map to $\CC^m$, and $Y_E$ the image of this in $\PP^{m-1}$.
The $\ZZ$-graded $K$-polynomial of~$\hat Y_E$ is
the numerator of the $\ZZ$-graded Hilbert series of the
coordinate ring of $\hat Y_E$; it is the class of $\hat Y_E$ in the $\CC^\times$-equivariant $K$-theory of $\CC^m$. By the localization sequence in $K$-theory, and since $\CC^m \setminus \{0\}$ is a $\CC^\times$-torsor over $\PP^{m-1}$, we have maps
\[
  K_0^{\CC^\times}(\CC^m) \to K_0^{\CC^\times}(\CC^m \setminus \{0\}) \stackrel{\sim}{\to} K_0(\PP^{m-1}),
\]
and the class of $\hat Y_E$ is mapped to the class of $Y_E$. In
particular, the class of $Y_E$ is canonically represented by the
$K$-polynomial of $\hat Y_E$.
\begin{proposition}\label{prop:K}
  Assume that $E$ is birational to $\hat Y_E$, and that $\hat Y_E$ is normal and has rational
  singularities.  Then the $K$-polynomial of $\hat Y_E$ is
  \[
    \K(\hat Y_E;U) = \sum_{i \geq 0} \textstyle\chi(X, \bigwedge^i F^\vee ) (-U)^i \in \ZZ[U^{\pm 1}]
  \]
  where the Euler characteristic $\chi(X,-)$ is the alternating sum
  of the dimensions of the cohomology groups of the vector bundle $-$.
\end{proposition}
\begin{proof}
  Let $A$ denote the coordinate ring of $\CC^m$ and $I$ the ideal of $\hat Y_E$. By
  \cite[Theorem~5.1.3]{weyman}, the minimal free resolution of
  $A/I$ as an $A$-module has as its $i$th term $F_i = \bigoplus_j H^j(X,\bigwedge^{i+j} F^\vee) \otimes A(-i-j)$. This is constructed by pushing forward the Koszul resolution of $p^* F$, where $p : \CC^m \times X \to X$ is the bundle map. The $K$-polynomial of $\hat Y_E$ is thus
  \begin{multline*}
    \sum_i  (-1)^i \sum_j \dim H^j(X,{\textstyle\bigwedge^{i+j}} F^\vee) U^{i+j}=\\
    \sum_i \sum_j  (-1)^j \dim H^j(X,{\textstyle\bigwedge^{i+j}} F^\vee)(-1)^{i+j} U^{i+j}.
  \end{multline*}
  Condensing the inner summation gives the desired result.
\end{proof}
% This general result is mirrored in the bigraded results \cref{thm:Kpolyrealizable} and \cref{thm:Z^2 resolution}.

The degree of $Y_E \subset \PP^{m-1}$ is number of points in the
intersection of $Y_E$ with a general linear space whose dimension
compliments $\dim(Y_E)$. This is equal to the (multi)degree of
$\hat Y_E$ in the $\ZZ$-grading.
\begin{proposition}\label{prop:C}
  Assume that $E$ is birational to its collapsing $\hat Y_E \subset \CC^m$. Then
  \[
    \deg(Y_E) = \int_X c_{\dim(X)}(F)
  \]
  where $c_k(-)$ is the $k$th Chern class of $-$.
\end{proposition}
\begin{proof}
  Say that $X$ is $n$-dimensional and write $r=\rk E$. 
  Since the product of the total
  Chern classes of $E$ and $F$ is $1$, we find that the top Chern
  class of $F$ is the top Segre class of $E$. Now
  $\int_X s_{\dim(X)}(E) = \int_{\mathbf{P}E} \zeta^{r+n-1}$ by
  definition of the Segre class, where $\zeta$ is the class of
  $\mathcal{O}(1)$ on $\PP (\underline{\CC}^m)$ (restricted to
  $\mathbf{P}E$). We now compute
  \[
    \int_{\mathbf{P}E} \zeta^{r+n-1} =
    \int_{\PP^{m-1} \times X} [\PP E]\zeta^{r+n-1} =
    \int_{\PP^{m-1}} [Y_E]
  [H]^{r+n-1},
\]
where $H$ is a hyperplane in $\PP^{m-1}$. The last equality is the
push-pull formula, since $q_* [\PP E] = [Y_E]$ (because $E$ is birational to $\hat Y_E$) and $q^*[H] = \zeta$.
\end{proof}

\subsection{Tropical intersection theory and toric Chow rings}\label{sec:chow}
% was: The Chow ring of the permutohedral variety

Write $N = \ZZ^n/\ZZ$ for the ambient lattice of $\Sigma_n$, and $N_\RR:=N\otimes\RR=\RR^n/\RR$.
Our main construction in \Cref{sec:products} will involve 
intersections of tropical varieties within $N_\RR$,
so we restrict our attention to this ambient space from the start.
A general reference for this section is \cite[\S3.6, \S6.7]{maclaganSturmfels}.

For polyhedra $\sigma,\tau$ in $N_\RR$,
let $\sigma-\tau$ be the Minkowski sum $\sigma+(-\tau)=\{x-y:x\in\sigma,y\in\tau\}$.
If $\tau$ is rational, 
define $N_\tau$ to be the intersection of $N$ with the $\RR$-span of $\tau-\tau$.

\begin{definition}
  A \newword{tropical cycle} $(C,X)$ of dimension $d$ in~$N_\RR$ %, with lattice structure given by~$N$,
  is a rational polyhedral complex $C$ in~$N_\RR$, pure of dimension~$d$,
  with a \newword{weight} function $X: C_d \to \ZZ$
  satisfying the balancing condition: For all
  $\tau \in C_{d-1}$ we have
  \[
    \sum_{\tau \subset \sigma \in C_d} X(\sigma)\,
    v_{\sigma/\tau} = 0 \in N/N_\tau
  \]
  where $v_{\sigma/\tau}$ is the primitive vector in $N/N_\tau$
  on the ray generated by the image of $\sigma-\tau$ 
  in $N_\RR/(\tau-\tau) = (N/N_\tau)\otimes\RR$. 
  Two tropical cycles are equal if they differ only up to refinement of~$C$,
  with $X$ being pulled back to the refinement,
  and/or deletion of $d$-faces of weight 0
  (and if needed deletion of maximal faces of dimension $<d$).
\end{definition}

If we call $(C,X)$ an \newword{unbalanced tropical cycle} 
we mean that it need not satisfy the balancing condition.
Unbalanced tropical cycles will not appear outside of the present section;
they are only for convenience of setup.

In most tropical cycles of importance in this paper, all weights will be~$1$.
If we specify a tropical cycle by giving only the polyhedral complex, this is what is intended.

Write $TC^{n-1-d}(N)$, with indexing by codimension, 
for the set of tropical cycles of dimension $d$ in~$N_\RR$.
Then $TC^{n-1-d}(N)$ is an abelian group with the following addition.
Given two unbalanced tropical cycles $(C_X,X)$ and $(C_Y,Y)$ of dimension~$d$,
we may construct a $d$-dimensional complex $C$ which has subcomplexes refining both $C_X$ and $C_Y$,
produce new weight functions $X'$ and $Y'$ with $(C,X')=(C_X,X)$ and $(C,Y')=(C_Y,Y)$,
and then set the sum to be $(C,X'+Y')$.
If $(C_X,X)$ and $(C_Y,Y)$ are both tropical cycles, then so is $(C,X'+Y')$.

Let $TC^\bullet(N) = \bigoplus_{k\ge0} TC^k(N)$,
with $TC^k(N)=0$ when $k>n-1$.
The additive group $TC^\bullet(N)$ has a product given by \newword{stable intersection} that makes it into a graded ring.
We define only the cases of stable intersection that we use
(and defer to \cite{maclaganSturmfels} for the rest), 
beginning with the transverse case.
The \newword{support} of a tropical cycle $(C,X)$ is
\[|(C,X)|:=\bigcup_{\sigma\in C_d,\, X(\sigma)\ne0}\sigma.\]

\begin{definition}[{\cite[Definition 3.4.9]{maclaganSturmfels}}]
Let $(C_X,X)$ and $(C_Y,Y)$ be tropical cycles.
Suppose that for every point $x\in |(C_X,X)|\cap |(C_Y,Y)|$,
if $\sigma$ and $\tau$ are the unique cells of $C_X$ and $C_Y$ respectively
containing $x$ in their relative interior,
then the Minkowski sum $\sigma+\tau$ is full-dimensional in~$N_\RR$.
We then say that $(C_X,X)$ and $(C_Y,Y)$ \newword{intersect transversely}.
\end{definition}

\begin{remark}\label{rem:tropical moving lemma}
If $(C_X,X)$ and $(C_Y,Y)$ are any two tropical cycles,
then the translates $(C_X+w_1,X)$ and $(C_Y+w_2,Y)$ intersect transversely
for generic $w_1,w_2\in N_\RR$.
It suffices to choose $w_1-w_2$ lying off of finitely many hyperplanes with rational coordinates
\cite[proof of 3.6.12]{maclaganSturmfels}.
\end{remark}

\begin{definition}[{\cite[after 3.6.11]{maclaganSturmfels}}]\label{def:transverse tropical intersection}
Suppose that $(C_X,X) \in TC^{c_1}(N)$ and $(C_Y,Y) \in TC^{c_2}(N)$
intersect transversely.
Then $(C_X,X)\cdot (C_Y,Y)\in TC^{c_1+c_2}(N)$ is defined as the following tropical cycle $(C_Z,Z)$.
The cells of $C_Z$ are all nonempty intersections $\sigma\cap\tau$
of cells $\sigma\in C_X$ and $\tau\in C_Y$.
If $\dim(\sigma\cap\tau)=n-1-(c_1+c_2)$ then
\[Z(\sigma\cap\tau) = [N:N_\sigma+N_\tau]\cdot X(\sigma)\,Y(\tau).\]
\end{definition}
%This implies the claim about dimension and that the intersection is tropically stable by \cite[Lemma 2.6]{jensenYu}.

The other case in which we will define stable intersection is for fans,
where it agrees with the \emph{fan displacement rule} of Fulton and Sturmfels {\cite[Proposition~4.1]{fultonSturmfels}}.

If $\sigma$ is the polyhedron of \eqref{eq:Ax<=b}, then its \emph{recession cone} is
\[\rec(\sigma)=\{x\in N_\RR:\langle x,a_i\rangle\le 0\mbox{ for }i=1,\ldots,s\}.\]
Taking recession cone is an idempotent operation,
and induces an idempotent group endomorphism of~$TC^\bullet(N)$ preserving the grading, which we also denote $\rec$.
Let $(\{\sigma\},k)$ denote the unbalanced tropical cycle with one maximal face $\sigma$ given weight $k\in\ZZ$. 
For $(C,X)$ a tropical cycle of dimension~$d$, we set
\[\rec((C,X)) = \sum_{\sigma\in C_d}(\rec(\sigma),X(\sigma)).\]
This sum, which we call the \newword{recession cycle} of $(C,X)$, is always a tropical cycle.
The image of~$\rec$ is the group of tropical cycles whose underlying polyhedral complexes are fans.
Observe that $\rec((C,X))=\rec((C+w,X))$ for~$w\in N_\RR$.

\begin{proposition}[{cf.\ \cite[Propositions 3.5.6, 3.6.12]{maclaganSturmfels}}]\label{prop:tropical moving fans}
The map $\rec:TC^\bullet(N)\to TC^\bullet(N)$ is a ring endomorphism.
In particular, if $(C_X,X)$ and $(C_Y,Y)$ are fans,
\[(C_X,X)\cdot(C_Y,Y) = \rec((C_X+w_1,X)\cdot(C_Y+w_2,Y))\]
for any $w_1,w_2\in N_\RR$.
\end{proposition}

By \Cref{rem:tropical moving lemma} the intersection on the right hand side may be made transverse with a generic choice of $w_1$ and~$w_2$, 
and then computed by \Cref{def:transverse tropical intersection}.

The fan displacement rule was first introduced
in the setting with a single fixed rational fan $\Sigma$ on~$N$, to model the toric Chow ring.
Let $MW^\bullet(\Sigma)$, the group of \newword{Minkowski weights} on~$\Sigma$,
be the subgroup of $TC^\bullet(N)$ generated by tropical cycles $(C,X)$ where $C$ is a subfan of~$\Sigma$.

\begin{theorem}[{\cite[Theorem~3.1]{fultonSturmfels}, \cite[Corollary~17.4]{fultonIT}}]\label{thm:MW is Chow}
Let $\Sigma$ be a unimodular complete fan on~$N$.
Then $MW^\bullet(\Sigma)$ with the stable intersection product is a subring of~$TC^\bullet(N)$,
isomorphic to the Chow ring $A^\bullet(X_\Sigma)$
of the toric variety of~$\Sigma$.
\end{theorem}

The fan $\Sigma_n$, which the rest of this paper will wholly concentrate on, satisfies the assumptions of \Cref{thm:MW is Chow}.
A convenient property of stable intersections for~$\Sigma_n$ is that 
the lattice indices in \Cref{def:transverse tropical intersection} are always~1 \cite[Lemma~4.5]{BST}.
%This is known (e.g., \cite[Lemma~4.5]{BST}), but we give a proof for ease of reference.
%\begin{lemma}
%Let $\sigma$ and~$\tau$ be cones of $\Sigma_n$ with $\sigma+\tau$ full-dimensional in $N_\RR$.
%Then $[N:N_\sigma+N_\tau]=1$.
%\end{lemma}
%
%\begin{proof}
%We use induction on~$n$, the base case $n=1$ being trivial.
%Since $\sigma+\tau$ has dimension $n-1$, it has at least $n-1$ rays, 
%so one summand, without loss of generality $\sigma$, has at least $\frac12(n-1)$ rays,
%generated by the vectors $e_{S_j}+\RR\mathbf{1}$
%for some chain of subsets $S_\bullet : \emptyset = S_0 \subsetneq S_1 \subsetneq \cdots \subsetneq S_k \subsetneq S_{k+1} = [n]$ with $k\ge\frac12(n-1)$.
%Because $2(k+1)>n$, there is an index $j$ with $|S_j|-|S_{j-1}|<2$,
%i.e.\ $S_j\setminus S_{j-1}$ is a singleton $\{i\}$.
%Now the quotient by $\RR e_i$ is a map $\RR^n/\RR\to\RR^{n-1}/\RR$
%under which $\sigma$ and~$\tau$ map to cones $\sigma', \tau'$ of $\Sigma_{n-1}$ whose sum is still full-dimensional.
%By the inductive hypothesis, $N_{\sigma'}+N_{\tau'}$ contains a generating set for~$N/e_i\ZZ$.
%Lifts of its elements to~$N_\sigma+N_\tau$,
%together with $e_i+\RR\mathbf{1}\in N_\sigma$, generate $N$.
%\end{proof}

\subsection{Tautological classes of matroids}\label{ssec:best stuff}
We consider $\CC^n$ with the action of $T$ by
inversely scaling coordinates. Say that $L \subset \CC^n$ is a linear subspace. Let $\mathcal{S}_L$ and $\mathcal{Q}_L$ be the
\newword{tautological sub- and quotient bundles} over the
permutohedral variety $X_n$ associated to $L$. That is,
$\mathcal{S}_L$ is the unique $T$-equivariant subbundle of the trivial
bundle $\underline{\CC^n}=\CC^n \times X_n$ whose fiber over the image
of the identity of $T$ in $X_n$ is $L \subset \CC^n$. The quotient
bundle is $\mathcal{Q}_L = \underline{\CC^n}/\mathcal{S}_L$. A
systematic treatment of these bundles is undertaken in \cite{BEST}. In this
section we review some key results.

\subsubsection{Tautological classes in $K$-theory} The classes
$[\mathcal{S}_L], [\mathcal{Q}_L] \in K^T_0(X_n)$ can be described
using $T$-equivariant localization as follows.  Let $S_n$ denote the
symmetric group of permutations of $[n]$. 
For $\pi \in S_n$ and $M$ any matroid on~$[n]$, let
$B_M(\pi)$ denote the basis of $M$ one forms by greedily adding
elements in the order they appear in the one-line notation of $\pi = (\pi(1),\dots,\pi(n))$; 
we call $B_M(\pi)$ the lexicographically first basis with respect to~$\pi$.
The formulae for the localization of $[\mathcal{S}_L]$ and $[\mathcal{Q}_L]$ at the $T$-fixed point indexed by $\pi$ are
\begin{align*}
  [\mathcal{S}_L]_\pi = \sum_{i \in B_{M(L)}(\pi)} T_i^{-1},\qquad \qquad
  [\mathcal{Q}_L]_\pi = \sum_{i \notin B_{M(L)}(\pi)} T_i^{-1},
\end{align*}
which are elements of $K^T_0(\textup{pt}) = \ZZ[T_1^{\pm 1},\dots,T_n^{\pm 1}]$,
recalling the notation $M(L)$ for the matroid realized by~$L$.

For any matroid $M$ on $[n]$, we define classes
$[\mathcal{S}_M], [\mathcal{Q}_M] \in K^T_0(X_n)$ by localization:
\begin{align*}
    [\mathcal{S}_M]_\pi = \sum_{i \in B_M(\pi)} T_i^{-1},\qquad \qquad
  [\mathcal{Q}_M]_\pi = \sum_{i \notin B_M(\pi)} T_i^{-1}.
\end{align*}
These are well-defined classes of $K^T_0(X_n)$ by
\cite[Proposition~3.8]{BEST} -- they are the \newword{tautological
  sub- and quotient bundle classes of $M$}.

There are associated dual classes $[\mathcal{S}^\vee_M]$ and
$[\mathcal{Q}^\vee_M]$, in whose local classes $T_i$ appears in place of $T_i^{-1}$.
The ring $K^T_0(X)$ is a $\lambda$-ring and we have
classes $\Sym^k [\mathcal{S}^\vee_M]$ and
$\bigwedge^k [\mathcal{Q}^\vee_M]$ given locally by
\begin{align*}
  \sum_{k \geq 0} \left(\Sym^k [\mathcal{S}^\vee_M]\right)_\pi q^k &= \prod_{i \in B_M(\pi)}\frac{1}{1-T_i q},\\
  \sum_{k \geq 0} \left({\textstyle\bigwedge^k} [\mathcal{Q}^\vee_M]\right)_\pi q^k &=\prod_{i \notin B_M(\pi)}{(1+T_i q)}.
\end{align*}
When $M$ is realized by a linear subspace $L$, we have $\Sym^k [\mathcal{S}^\vee_M] = \Sym^k(\mathcal{S}^\vee_L)$, etc.

Of critical later importance is the following result which occurs as \cite[Proposition~5.8]{BEST}.
\begin{proposition}\label{prop:exterior power is valuative}
  The assignment $M \mapsto [\mathcal{Q}_M]$ is valuative. For a pair of
  matroids $(M_1,M_2)$ and any $i,j$, the assignment
  \[
    (M_1,M_2) \mapsto \textstyle\bigwedge^i [\mathcal{Q}^\vee_{M_1}] \cdot \bigwedge^j [\mathcal{Q}^\vee_{M_2}]
  \]
  is valuative as a function of $M_1$ and, separately, as a function of $M_2$.
\end{proposition}
% \begin{proof}
%  The first claim follows from \cite[Proposition~5.8]{BEST} and the
%  second claim follows from the first it concerns a product of
%  valuative functions.
%\end{proof}
% -- there is nomore second claim here. 

We will use the same symbols $[\mathcal{S}_M]$, $[\mathcal{Q}_M]$, etc.\ %
to denote the restriction of equivariant classes to $K_0(X_n)$.

\subsubsection{Tautological classes in the Chow ring}
The vector bundles $\mathcal{S}_L$ and $\mathcal{Q}_L$ over $X_n$ have
Chern classes in $A^\bullet(X_n)$. These admit matroidal
generalizations such as $c_j(\mathcal{Q}_M) \in A^\bullet(X_n)$, and
these satisfy $c_j(\mathcal{Q}_M) = c_j(\mathcal{Q}_L)$ when $L$
realizes $M$ over $\CC$. 
The \newword{Bergman class} of~$M$ is the top Chern class $c_{n-\rk(M)}(\mathcal{Q}_M)$.
The extension is natural in the sense that
the function that assigns to each matroid $M$ a fixed polynomial
expression in the Chern classes of $M$, is a valuative function. This
follows from the proof of \cite[Proposition~5.8]{BEST} although we
will not specifically draw on this fact.

Via \Cref{thm:MW is Chow}, we identify the Chern classes
$c_j(\mathcal{Q}_M)\in A^\bullet(X_n)$ with Minkowski weights in
$MW^\bullet(\Sigma_n)$.  As a Minkowski weight $c_j(\mathcal{Q}_M)$
takes only the values $0$ and $1$, so in fact we use the symbol
``$c_j(\mathcal{Q}_M)$'' to name the subfan of~$\Sigma_n$ whose
maximal cones are those bearing weight~1.  The cones in question are
described by \cite[Proposition~7.4]{BEST}.
\begin{proposition}\label{prop:chern classes of Q}
  Let $M$ be a matroid on $[n]$. For all $j$, the Minkowski weight
  $c_j(\mathcal{Q}_M)$ takes the value $1$ on the cone $\sigma_{S_\bullet}$
  labeled by a chain
  \[
    S_\bullet : \emptyset = S_0 \subsetneq S_1 \subsetneq S_2 \subsetneq \dots \subsetneq S_k \subsetneq S_{k+1} = [n]
  \]
  if and only if $j+k = n-1$ and, of the matroids
  $M|S_i/S_{i-1}$ for $i=1,\dots,k+1$, exactly $\operatorname{corank}(M) - j$
  of them are loops (i.e., $M|S_i/S_{i-1}$ is rank $0$ on exactly one
  element) and the rest are uniform of rank $1$.
\end{proposition}

\begin{example}\label{ex:chern classes}
  Let $M_1 = U_{2,3}$ and $M_2 = U_{1,3}$. For $j=1,2$,
  $c_0(\mathcal{Q}_{M_j})$ is $1$ on every cone of $\Sigma_3$. The
  first Chern classes of $\mathcal{Q}_{M_1}$ and $\mathcal{Q}_{M_2}$ are the $1$-dimensional fans
  whose support is shown in the middle and right images, respectively, 
  after choosing coordinates on $\RR^3/\mathbf{R}\mathbf{1}$ as in the left image.
  
  \begin{center}   
    \includegraphics{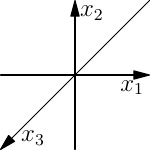}
    \qquad
    \includegraphics{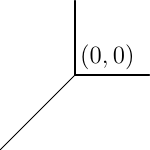}
    \qquad
    \includegraphics{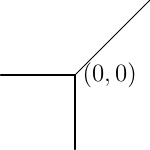}
  \end{center}
  Since $M_1$ has corank $1$ we have $c_2(\mathcal{Q}_{M_1}) = 0$,
  and $c_2(\mathcal{Q}_{M_2})$ consists of a single point --- the
  origin.
\end{example}

\begin{example}\label{ex:S_L ++ O(-beta)}
  We use \cref{prop:K,prop:C} to compute the $\ZZ$-graded degree and
  $K$-polynomial of the collapsing of
  $\mathcal{S}_L \oplus \mathcal{O}(-\beta)$ in the case the
  collapsing is birational  to the bundle (this hypothesis can be relaxed). We will see later that this collapsing is the Schubert variety of a single linear space, studied in \cite{ardilaBoocher}. Here
  $\mathcal{O}(-\beta)$ is the tautological subbundle of the
  one-dimensional subspace $\CC \mathbf{1} \subset \CC^n$. If $L$ has matroid $M$, the degree is computed by \cref{prop:C} to be
  \[
    \int_{X_n} c(\mathcal{Q}_M) \frac{1}{1-\beta},
  \]
  since $1+\beta + \beta^2 + \cdots $ is the total Chern class of
  $\CC^n/\mathcal{O}(-\beta)$. By \cite[Theorem~A]{BEST}, this is the
  number of bases of the the matroid $M$. The
  $K$-polynomial is computed by  \cref{prop:K} to be,
  \[
    \sum_i \chi(X_n,{\textstyle\bigwedge^i}( \mathcal{Q}_L \oplus \mathcal{Q}_{\CC \mathbf{1}})) (-U)^i
    =
    \sum_{i,j} \chi(X_n,{\textstyle\bigwedge^i}( \mathcal{Q}_L ) \otimes  {\textstyle\bigwedge^j}(\mathcal{Q}_{\CC \mathbf{1}})) (-U)^{i+j}
  \]
  By \cite[Theorem~10.5]{BEST}, this is equal to
  \[
    (1-U)^r (1-U)^{n-1} \int_{X_n} c(\mathcal{Q}_M^\vee, -U/(1-U) ) \frac{1}{1+\beta (-U)/(1-U) }.
  \]
  Applying \cite[Theorem~A]{BEST} again we see that the $K$-polynomial can be expressed in terms of the Tutte polynomial of $M$ as  $(1-U)^{n-r} U^r\cdot  T_M(1/U,1)$, where $r$ is the rank of $M$.
\end{example}

\subsection{The tropical permutahedral variety}\label{ssec:Trop X_n}
In \Cref{ssec:bivaluativity} it will be technically desirable to work in a compactification of $\RR^n/\RR$.
The objects involved are tropicalizations of cycles on~$X_n$,
and it turns out that the compactification to which they extend correctly is the tropicalization of~$X_n$
in the sense of \cite[\S3]{payneAnalytification} and \cite[\S6.2]{maclaganSturmfels}.  
We review the definition following the latter reference, continuing in specializing to the permutahedral case.
The symbol $\infty$ means positive infinity throughout,
and $\RR\cup\{\infty\}$ is a topological semigroup under addition with intervals $(a,\infty]$ forming a base of neighborhoods of $\infty$.

As a set, 
\[\Trop X_n = \coprod_{\sigma_{S_\bullet}\in\Sigma} O_{S_\bullet}\]
where 
\begin{equation}\label{eq:O_S}
O_{S_\bullet} = N_\RR / \operatorname{span}(\sigma_{S_\bullet})
= \prod_{i=1}^{\ell(S_\bullet)}\RR^{S_i\setminus S_{i-1}}/\RR e_{S_i\setminus S_{i-1}}.
\end{equation}
(For $O_{S_\bullet}$ \cite{maclaganSturmfels} use the notation $N(\sigma_{S_\bullet})$. $O$ is for tropical torus \underline{o}rbit.)

Let $U_{S_\bullet} = \bigcup_{\sigma_{T_\bullet}\subseteq\sigma_{S_\bullet}} O_{T_\bullet}$.
We regard $U_{S_\bullet}$ as the set of semigroup homomorphisms $\phi$
valued in $\RR\cup\{\infty\}$
with domain the lattice points in the dual cone $\sigma_{S_\bullet}^\vee\subseteq\operatorname{span}(\sigma_{S_\bullet})^*$ 
(that is, a linear functional is in the domain if it sends elements of $N$ to~$\ZZ$).
For each such $\phi$ there is a unique face $\sigma_{T_\bullet}$ of $\sigma_{S_\bullet}$ 
such that $\phi^{-1}(\RR)$ equals $\sigma_{T_\bullet}^\perp$, the set of functionals zero on~$\sigma_{T_\bullet}$;
then $\phi$ lies in $O_{T_\bullet}$,
the perpendicular space to $\sigma_{T_\bullet}^\perp$ in the double dual $(\operatorname{span}(\sigma_{S_\bullet})^*)^*$.
The identification of $U_{S_\bullet}$ as a set of maps to $\RR\cup\{\infty\}$ 
lets us endow it with the pointwise convergence topology,
and gluing all the $U_{S_\bullet}$ along the open inclusions among them 
provides the topology on $\Trop X_n$.

In fact $\Trop X_n$ is covered by the $U_{S_\bullet}$ for which $S$ is a maximal chain.
In these cases $U_{S_\bullet}$ is homeomorphic to $(\RR\cup\{\infty\})^{n-1}$.
The inequality description of~$\sigma_{S_\bullet}$ implies that
the homeomorphism restricts to the ``finite'' points $O_{\emptyset\subset[n]}$ as
\[x+\RR\mathbf{1}\mapsto (x_{\pi(2)}-x_{\pi(1)},\ldots,x_{\pi(n)}-x_{\pi(n-1)})\]
if the permutation $\pi$ of~$[n]$ satisfies $S_j=\{\pi(1),\ldots,\pi(j)\}$.
In other words, on $O_{\emptyset\subset[n]}$, 
all pairwise differences of the tropical projective coordinate functions $(x_1,\ldots,x_n)$ 
have well-defined real values;
those differences $x_{\pi(j)}-x_{\pi(i)}$ where $j\ge i$
extend to $U_{S_\bullet}$, taking values in $\RR\cup\{\infty\}$.

In the next subsection and \Cref{ssec:bivaluativity} we are concerned with sorting coordinates into nondecreasing order,
so we give a name to the restriction $f_\pi:[0,\infty]^{n-1}\to U_{S_\bullet}$ of the inverse of the above homeomorphism.
The images of the $f_\pi$ cover $\Trop X_n$ as $\pi$ ranges over permutations, 
and they intersect only along coordinate subspaces:
if $f_\pi(d)=f_{\pi'}(d')$ then $d=d'$, and $\pi(\{1,\ldots,i\})=\pi'(\{1,\ldots,i\})$ for every $i$ with $d_i\ne0$.
In particular the $f_\pi([0,\infty)\,)$ are the maximal cones of~$\Sigma_n$.

\subsubsection{Tropical convexity}

Tropical convexity was introduced in \cite{develinSturmfels}, 
as a property of subsets of~$\RR^n$ or $(\RR\cup\{\infty\})^n$.
In this section we make an ad hoc extension to subsets of $\Trop X_n$.

A subset $Y\subseteq\RR^n$ is \newword{tropically convex} if, for any $x,y\in Y$, the set
\[
\{\big(\min(a+x_1,b+y_1),\ldots,\min(a+x_n,b+y_n)\big) : a,b\in[0,\infty],\min(a,b)=0\},
\]
called the \newword{tropical line segment} between $x$ and $y$, is contained in~$Y$.
A subset of $\RR^n/\RR$ is tropically convex if its preimage in~$\RR^n$ is.
%(some works, e.g.\ \cite{allamigeonKatz}, distinguish these sets as ``tropical cones'').

To make the definition we will index the cones of~$\Sigma_n$ by total preorders,
now writing $O_\preceq$ for what we named $O_{S_\bullet({\preceq})}$ above.
We set a few notations. 
Given a preorder $\preceq$, let $\sim_{\preceq}$ be its largest sub-equivalence relation, $a\mathrel{\sim_\preceq}b$ iff $a\preceq b$ and $b\preceq a$.
Given a set function $f:A\to B$ and a relation $\preceq$ on~$B$,
define the relation $f^*(\preceq)$ on~$A$ by $a\mathrel{f^*(\preceq)}b$ iff $f(a)\preceq f(b)$.

\begin{definition}\label{def:tropical convexity}
Suppose $Y\subseteq\Trop X_n$.
Take any two points $p_1,p_2\in Y$, with $p_k\in O_{\preceq_k}$.
Choose projective coordinates $(p_{(k,i)})_{i\in[n]}$ for $p_k$:
that is, the $p_{(k,i)}$ are coordinates of lifts of $p_k$ 
from the right hand side of \eqref{eq:O_S} to $\prod_i\RR^{S_i\setminus S_{i-1}}=\RR^n$.

Let $\preceq$ be any total preorder on $[2]\times[n]$ whose restriction to $\{k\}\times[n]$ is $\preceq_k$.
Define a smaller preorder $\leq$ on $[2]\times[n]$ by
\[(k,i)\le(\ell,j)\mbox{ iff either } (k,i)\not\succ(\ell,j),\mbox{ or }
  (k,i)\mathrel{\sim_{\preceq}}(\ell,j)\mbox{ and }p_{k,i}\le p_{\ell,j}.\]
Define $s:[n]\hookrightarrow[2]\times[n]$ so that $s(i)$ is the minimum of $\{(1,i),(2,i)\}$ under $\leq$.
Let $q\in\Trop X_n$ be the point with projective coordinates $(p_{s(i)})_{i\in[n]}$ in~$O_{s^*(\preceq)}$.
If we have $q\in Y$ however the prior choices are made, then we call $Y$ \newword{tropically convex}.
\end{definition}

Observe that the additive actions of $\mathbf R^n$ on every $O_\preceq$ glue to a continuous action on $\Trop X_n$.
We see from the definition that this action preserves convexity: 
if for some $v\in\mathbf R^n$ we replace $p_k$ by $p_k+v$ and $p_{(k,i)}$ by $p_{(k,i)}+v_i$,
then the preorders and the function $s$ are unchanged, so the point finally obtained is $q+v$.

\begin{lemma}\label{lem:tropically convex implies contractible}
A tropically convex set $Y$ in $\Trop X_n$ containing a point of $\mathbf R^n/\mathbf R$ is contractible.
\end{lemma}

\begin{proof}
Using the additive action of~$\mathbf R^n$ we may assume $Y$ contains the origin $\mathbf 0$ of~$\mathbf R^n/\mathbf R$.
Let $p_1\in Y$, and suppose $p_1=f_\pi(d)$ for $d\in[0,\infty]^{n-1}$.
We show using tropical convexity that $f_\pi(d')\in Y$ for every $d'\in[0,\infty]^{n-1}$ such that, for some $1\le j\le n-1$,
we have $d'_j\le d_j$; $d'_i=d_i$ for all $i<j$; and $d'_i=0$ for all $i>j$.
We may assume $d'_j$ is finite; otherwise replace $j$ by $j+1$
(or if $j=n-1$, then $d'=d$ and we need not invoke convexity).
Then to construct $f_\pi(d')$ using the definition of tropical convexity, set $p_2=\mathbf 0$ 
with $p_{(2,i)}=p_{(1,\pi(j))}+d'_j$ for all $i\in[n]$,
and take $\preceq$ such that $(2,i)\mathrel{\sim_\preceq}(1,\pi(j))$ for all $i\in[n]$.

Now let $r:[0,\infty]\times[0,1]\to[0,\infty]$ be a strong deformation retraction of $[0,\infty]$ onto $\{0\}$,
i.e.\ such that $r(0,t)=0$ for all $t\in[0,1]$.
By applying $r$ sequentially to $d_{n-1}$, then $d_{n-2}$, \ldots, then $d_1$ in the coordinates of a point $f_\pi(d)$, simultaneously for all~$\pi$,
we obtain a deformation retraction of $Y$ onto $\{\mathbf 0\}$ (which is well defined and continuous because $r$ is strong).
\end{proof}

\section{The diagonal Dilworth truncation}\label{sec:D}

In this section we introduce and set out properties of the \emph{diagonal Dilworth truncation}, 
a matroid of central interest in this paper constructed from a pair of input matroids.
For an explanation of the name see \Cref{rem:Dilworth}.
Here and in later sections, by a \emph{pair of matroids} we will
always mean an ordered pair $(M_1,M_2)$ where $M_1$ and $M_2$ are
matroids on the same ground set and have no loops in common. A pair of matroids is called realizable if both its constituents are realizable.

\subsection{The construction}
\begin{definition}\label{def:D}
Let $(M_1,M_2)$ be a pair of matroids 
with respective rank functions $\rk_{M_1}$ and $\rk_{M_2}$.
We define the \newword{diagonal Dilworth truncation} $D(M_1,M_2)$ to be the pair $(M_1,M_2)$ together with the matroid $D $ on~$[n]$
whose rank function is defined by
\[\rk_{D}(S) = \min_{T_1\sqcup\cdots\sqcup T_\ell\subseteq S}
\big\{\sum_{i=1}^\ell(\rk_{M_1}(T_i)+\rk_{M_2}(T_i)-1)+|S\setminus(T_1\cup\cdots\cup T_\ell)|\}\]
where no $T_i$ is empty (but $\ell$ may equal~0).
\end{definition}
We will show momentarily that $\rk_D$ is the rank function of a
matroid, though a preliminary comment is in order. Often, we will want
to identify the matroid $D$ with $D(M_1,M_2)$, speaking of independent
sets, circuits, etc.\ of $D(M_1,M_2)$. Generally no confusion will
result from this, since rarely will we want to vary the pair
$(M_1,M_2)$ during a computation. Occasionally, though, we will
want to change the pair and in this case a more robust notation is
needed. In such cases, let $\underline{D(M_1,M_2)}$ denote the matroid
$D$. Thus, $D(M_1,M_2) = D(M_2,M_1)$ is only true when $M_1=M_2$, but
$\underline{D(M_1,M_2)} = \underline{D(M_2,M_1)}$ is always true.

\begin{example}\label{ex:D(M U_1n)}
For any matroid $M_1$ on $[n]$ we have ${D(M_1,U_{1,n})} = M_1$.
The minimum in $\rk_{D}(S)$ is attained by $T_1=S$ for $S$ nonempty.
\end{example}

\begin{proposition}\label{lem:I(D)}
The above function $\rk_{D}$ is the rank function of a matroid~$D$.
A set $I\subseteq[n]$ is independent in $D$ if and only if
for every $j\in[n]$, whether or not in~$I$,
there exist independent sets $I_1$ of~$M_1$ and $I_2$ of~$M_2$
whose multiset sum is $I+\{j\}$.
We have
\[IP(D) = \{x\in\mathbf R^n : x+\Delta\subseteq IP(M_1)+IP(M_2)\}\]
where $\Delta=\conv\{0,e_1,\ldots,e_n\}$.
\end{proposition}

\begin{proof}
The rank function of $D$ is the one in \cite[Theorem 8]{edmonds70}
for the $\beta_0$-function $f(S)=\rk_{M_1}(S)+\rk_{M_2}(S)-1$.
By that theorem it is the rank function of a polymatroid, and $\rk_D(\{i\})\le1$ is easily checked, so $D$ is a matroid.

Edmonds gives for the corresponding polytope the defining inequalities $x(S)\le f(S)$ for all $\emptyset\ne S\subseteq[n]$,
where $x(S)$ is notation for $\sum_{i\in S}x_i$.
Adding 1 to both sides, the inequality becomes
\[\max\{y(S) : y\in x+\Delta\}\le\max\{y(S) : y\in IP(M_1)+IP(M_2)\}.\]
A polymatroid (in Edmonds' sense) is the intersection with the positive orthant 
of a collection of halfspaces of the form $\{y:y(S)\le g_S\}$ for real $g_S$,
and both $\Delta$ and $IP(M_1)+IP(M_2)$ are polymatroids. 
This collection of inequalities is therefore equivalent to the containment of polyhedra 
$x+\Delta\subseteq IP(M_1)+IP(M_2)$,
which proves our equation for~$IP(D)$.

The description of independent sets follows from this by taking $x=e_I$,
as $e_I+\Delta$ is contained in a polytope if and only if all its vertices are.
The vertex $e_I+0$ does not need to be separately tested 
because independence polytopes are closed under changing a coordinate to a lesser non-negative value.
\end{proof}

\begin{remark}\label{rem:Dilworth}
Dilworth introduced the operation now called Dilworth truncation in~\cite{dilworth44}
to show that if the lowest ranks, aside from $\widehat0$, of a geometric lattice are deleted,
what remains can be completed back to a geometric lattice by adding elements in the new ranks \mbox{$\ge2$}.
See \cite{lovaszFlats} for further context.
The version we use, with a restriction to the diagonal, has precedent in \cite{tanigawa} and~\cite{concaTsakiris}. 

The general Dilworth truncation can be defined as follows, staying in the style of~\Cref{def:D} 
(compare \cite[Proposition~7.7.5]{white}).
Let $M$ be a matroid on ground set $E$ with rank function $\rk$.
Let $E'$ be a subset of $\binom E2$ containing no set both of whose elements are loops of~$M$.
Then the (first) \newword{Dilworth truncation} of $M$ on~$E'$ is the matroid on ground set $E'$ with rank function $\rk'$ 
where, for $S\subseteq E'$ nonempty,
\[\rk'(S) = \min_{T_1\sqcup\cdots\sqcup T_\ell\subseteq S}\big\{
\sum_{i=1}^\ell\big(\rk(\bigcup_{e\in T_i}e)-1\big)+|S\setminus(T_1\cup\cdots\cup T_\ell)|\big\}.\]

If we take the ground set of~$M_1\oplus M_2$ to be $[n]\times[2]$,
then $D(M_1,M_2)$ is the Dilworth truncation of $M_1\oplus M_2$ on $\{\{(i,1),(i,2)\}: i\in [n]\}$,
up to the isomorphism relabeling $\{(i,1),(i,2)\}$ as~$i$.
\end{remark}

\subsection{Basic properties of the diagonal Dilworth truncation}
Of the major cryptomorphic descriptions for~$D$, the last one we will need is the circuits.

\begin{proposition}\label{lem:C(D)}
The circuits of $D(M_1,M_2)$ are those nonempty subsets $C\subseteq[n]$ that are minimal with the property that
\[\rk_{M_1}(C) + \rk_{M_2}(C) = |C|.\]
Every such circuit~$C$ can be written as a disjoint union $C=I_1\sqcup I_2$
where $I_k$ is a basis of $M_k|C$ for $k=1,2$.
\end{proposition}

\begin{proof}
If $I$ is an independent set of~$D$, then the minimum in the definition of $\rk_D(I)$ is attained by $\ell=0$.
Therefore, for any set~$S$, the minimum in $\rk_D(S)$ is attained by some set family $T_1,\ldots,T_\ell$ in which none of the $T_i$~are independent in~$D$,
because omitting an independent set from the family will not increase the right hand side.
If $C$ is a circuit of~$D$, so that $\ell=0$ does not attain the minimum in $\rk_D(C)$,
then the foregoing implies that $T_1=C$ must do so: that is,
\[\rk_{M_1}(C)+\rk_{M_2}(C)-1=\rk_D(C)=|C|-1.\]

The matroid union theorem applied to $M_1|C$ and $M_2|C$ shows that the rank of $C$ in the union is $\rk_{M_1}(C) + \rk_{M_2}(C)$.
This implies the disjoint union statement.
\end{proof}

Taking $T_1=[n]$ in \Cref{def:D} we see that $\rk(D)$ is bounded above
by $\rk(M_1)+\rk(M_2)-1$. Equality here merits special consideration.
\begin{definition}
  If $\rk(D) = \rk(M_1) + \rk(M_2) -1$ we say that $D(M_1,M_2)$ has
  \newword{expected rank}.
\end{definition}
If the expectation $\rk(M_1)+\rk(M_2)-1$ exceeds $n$, 
then $D(M_1,M_2)$ cannot be of expected rank.
It is however possible that $D(M_1,M_2)$ is not of
expected rank even when the expected rank is at most $n$.
\begin{example}
  Say that $M_1$ and $M_2$ are both rank $2$ on $3$ elements and both
  have a single loop (not in common). Then $D(M_1,M_2)$ is rank $1$
  and has expected rank $3$. The unique basis of $D(M_1,M_2)$ is
  the singleton that is not a loop in either $M_1$ or $M_2$.
\end{example}
\begin{example}\label{ex:fano}
  Let $F$ be the Fano matroid and $F^{-}$ be the non-Fano matroid, both
  of which are of rank $3$ on $7$ elements. By an exhaustive computation,
  we see that $\underline{D(F,F)} = \underline{D(F^{-},F^{-})}$ is
  a uniform matroid of the expected rank $5$ on $7$ elements.
\end{example}

If $D$ has expected rank, then when $I$ is a basis of~$D$ in
\Cref{lem:I(D)}, both $I_1$ and $I_2$ must be bases of their
respective matroids also.
We record and generalize the consequences for spanning sets:

\begin{lemma}\label{lem:span D implies span M_k}\mbox{}
\begin{enumerate}
\item If $D(M_1,M_2)$ has expected rank and this exceeds~0, then any spanning set of~$D(M_1,M_2)$ is spanning in $M_1$ and~$M_2$.
\item If $M_1$ (resp.\ $M_2$) has no loops, then any spanning set of~$D(M_1,M_2)$ is spanning in~$M_2$ (resp.\ $M_1$).
\end{enumerate}
\end{lemma}

\begin{proof}
For (1), a spanning set $S$ of $D(M_1,M_2)$ contains a basis $B$. 
Choose $j\in B$ and let $I_1$ and $I_2$ be the independent sets with multiset sum $B+\{j\}$ provided by~\Cref{lem:I(D)}.
Then $I_1,I_2\subseteq B\subseteq S$, and the expected rank assumption implies that $I_k$ is a basis of~$B_k$.

For (2), suppose a spanning set $S$ of~$D$ was not spanning in, without loss of generality, $M_1$.
Let $i\in[n]$ be independent of~$S$ in~$M_1$, and let $T_1,\ldots,T_\ell$ be a minimizer in the definition of~$\rk_D(S\cup\{i\})$.
We may assume that no $T_j$ is the singleton $\{i\}$ as follows.
By choice of~$i$ it is not a loop in $M_1$, and by assumption it is not a loop in $M_2$,
so $\rk_{M_1}(\{i\})+\rk_{M_2}(\{i\})-1=1+1-1=1$.
This implies that $T_1,\ldots,\widehat{T_j},\ldots,T_\ell$ also achieves the minimum in $\rk_D(S\cup\{i\})$.

Let $T'_j=T_j\setminus\{i\}\ne\emptyset$.
We use the sets $T'_j$ to provide an upper bound on $\rk_D(S)$.
If $i$ is in some~$T_j$ then $\rk_{M_1}(T'_j)=\rk_{M_1}(T_j)-1$; otherwise, 
\[|S\setminus(T'_1\cup\cdots\cup T'_\ell)|=|S\cup\{i\}\setminus(T_1\cup\cdots\cup T_\ell)|-1.\]
In either case we find that $\rk_D(S)\le \rk_D(S\cup\{i\})-1$, which contradicts that $S$ spans~$D$.
\end{proof}

We continue with lemmas on the interaction of~$D$ with other matroid operations.
%
%\begin{lemma}\label{lem:D weak increasing}
%If $M_1$ and $M_2$ are relaxations of $N_1$ and~$N_2$ respectively,
%then $D(M_1,M_2)$ is a relaxation of $D(N_1,N_2)$.
%
%In particular, assume further that $\rk(M_k)=\rk(N_k)$ for $k=1,2$.
%Then if $D(N_1,N_2)$ has expected rank, $D(M_1,M_2)$ also has expected rank.
%\end{lemma}
%
%\begin{proof}
%Each term of the minimum in \cref{def:D} becomes weakly greater
%if $\rk_{N_1}$ and $\rk_{N_2}$ are replaced by $\rk_{M_1}$ and~$\rk_{M_2}$,
%so this is also true of the minimum itself, $\rk_D(S)$.
%\end{proof}
%
First, the diagonal Dilworth truncation commutes with restrictions.

\begin{lemma}\label{prop:D|S}
For any $S\subseteq[n]$ we have $D(M_1,M_2)|S=D(M_1|S,M_2|S)$.
\end{lemma}

\begin{proof}
This is clear from the form of the rank function in \Cref{def:D},
which when evaluating $\rk_D(S)$ only inspects $\rk_{M_k}(T)$ for subsets $T\subseteq S$.
\end{proof}

\begin{proposition}\label{prop:D directsum}
  Assume that $M_1$ and $M_2$ have ground set $E$, while $N_1$ and
  $N_2$ have ground set $E'$. Then,
  $D(M_1,M_2) \oplus D(N_1,N_2) = D(M_1 \oplus N_1,M_2 \oplus N_2)$.
\end{proposition}
\begin{proof}
  We prove that these matroids have the same set of circuits. 

  A circuit of $D(M_1,M_2) \oplus D(N_1,N_2)$ is a circuit of
  $D(M_1,M_2)$ or a circuit of $D(N_1,N_2)$. If $C$ is a circuit of
  $D(M_1,M_2)$ (the other case handled analogously) then it is minimal with the property that
  $\rk_{M_1}(C) + \rk_{M_2}(C) = |C|$. Since this circuit only meets
  $E$, it is also minimal with the property that
  $\rk_{M_1 \oplus N_1}(C) + \rk_{M_2 \oplus N_2}(C) = |C|$. Thus $C$
  is a circuit of $D(M_1 \oplus N_1,M_2 \oplus N_2)$.

  Assume that $C$ is a circuit of $D(M_1 \oplus N_1,M_2 \oplus
  N_2)$. Then $C$ is minimal with the property that
  $\rk_{M_1 \oplus N_1}(C) + \rk_{M_2 \oplus N_2}(C) = |C|$. It
  follows that at least one of the following is true:
  \begin{align*}
    \rk_{M_1}(C \cap E) + \rk_{M_2}(C \cap E)&\leq |C \cap E|, \\
    \rk_{N_1}(C \cap E') + \rk_{N_2}(C \cap E')&\leq |C \cap E'|.
  \end{align*}
  Without loss of generality, assume the first statement is true. Then
   $C \cap E$ would be dependent in $D(M_1, M_2)$ since
  \[\rk_{D(M_1,M_2)}(C \cap E) \leq \rk_{M_1}(C\cap E) + \rk_{M_2}( C
    \cap E) - 1 < |C \cap E|.
  \]
  Select a circuit of $D(M_1,M_2)$ in $C \cap E$. Since such a circuit
  is a circuit of $D(M_1 \oplus N_1, M_2 \oplus N_2)$, as above, it
  follows from the clutter circuit axiom that $C$ is equal to this
  circuit.
\end{proof}

In the sequel we will often need to reduce to the setting where $D(M_1, M_2)$ has the expected rank.
The next proposition, together with \Cref{ex:D(M U_1n)} as a base case,
shows that, at least when one matroid $M_k$ is loop-free, this reduction is afforded by a sufficient number of truncations of~$M_k$
(as a single matroid, to be distinguished from the diagonal Dilworth truncation). 
Let $\tr M$ denote the (usual) truncation of a matroid~$M$, whose independent sets are those independent sets of $M$ of cardinality less than $\rk(M)$.

\begin{proposition}\label{prop:D tr}
If $D(M_1,M_2)$ does not have the expected rank and $M_2$ has no loops then $\underline{D(M_1,M_2)}=\underline{D(M_1,\tr M_2)}$.
\end{proposition}

\begin{proof}
\Cref{lem:I(D)} implies that independent sets of $\underline{D(M_1,\tr M_2)}$ are independent in $\underline{D(M_1,M_2)}$.
So it is enough to show the converse implication, and enough to handle the bases of~$\underline{D(M_1,M_2)}$.
Let $B$ be a basis of~$\underline{D(M_1,M_2)}$, and $j\in[n]$.
By \Cref{lem:I(D)}, let $B+\{j\}=I_1+I_2$ be such that $I_k$ is independent in~$M_k$.
By assumption, at least one $I_k$ is not a basis of~$M_k$.
If this is true of~$I_2$, then $I_2$ is independent in $\tr M_2$ and we are done.
Otherwise, $B$ is spanning in~$M_1$ by \Cref{lem:span D implies span M_k} (since $M_2$ has no loops),
so there is some $i\in B\setminus I_1$ with $I_1\cup\{i\}$ still independent in~$M_1$. Then
\[(I_1\cup\{i\})+(I_2\setminus\{i\})=B+\{j\}\]
and $I_2\setminus\{i\}$ is independent in $\tr M_2$.
\end{proof}

\subsection{The case $M_1=M_2$}

As the $g$-invariant of a matroid $M$ arises from the case $M_1=M_2=M$ of our constructions,
we now briefly focus on the behavior of $D(M,M)$.
Observe that if $(M,M)$ is a pair of matroids then $M$ has no loops.

\begin{lemma}\label{lem:not expected rank implies D disconnected}
  Let $M$ be a loopless matroid, and assume that we have a partition
  $T_1\sqcup \dots \sqcup T_\ell \sqcup S= [n]$ with the property that
  \[
    \sum_{i=1}^\ell \left(2\rk_M(T_i) -1\right) + |S| = \rk(D(M,M)).
  \]
  Then:
  \begin{enumerate}
  \item The elements of $S$ are all coloops of $D(M,M)$.
  \item Each set $T_i$ is a  component of $D(M,M)$.
  \end{enumerate}
  In particular, $D(M,M)$ is disconnected unless it has expected
  rank.
\end{lemma}
\begin{proof}
  Let $S = \{s_1,\dots,s_m\}$. Note that if we set
  $T_{\ell+i}= \{s_i\}$ then, since $M$ is loopless,
  $2\rk_M(T_{\ell+i})-1 = 1$. Now
  \[
    \sum_{i=1}^{\ell+m} \left(2\rk_M(T_i) -1\right) = \rk(D(M,M))
  \]
  so we have reduced to the case when $S = \emptyset$.
  %
  %Let $T = \bigcup_i T_i$. We claim that $D|T$ has rank equal to
  %$\sum_i \left(2\rk_M(T_i) -1\right)$. Certainly the rank is at most
  %the displayed quantity, by \cref{def:D}. If the rank were smaller
  %then the definition furnishes  a witness $T'_1 \sqcup \dots T'_{\ell'} \sqcup S' = T$ to this and we would have,
  %\[
  %  \sum_i \left(2\rk_M(T'_i) - 1\right) + |S' \sqcup S| <
  %  \sum_j \left(2\rk_M(T_j) -1\right) + |S|,
  %\]
  %which is a contradiction to the right side being equal to
  %$\rk(D)$. Since we have $\rk(D) = \rk(D|T) + |S|$, this means each
  %element of $S$ is a coloop of $D$.
  % Assume now that $S$ is empty.

  For all $i$, we claim that $\rk(D|T_i) = 2\rk_M(T_i) - 1$. If the
  rank were smaller for some $i_0$ then we would have a collection of
  sets $\{T_{i_0,j}\}$ partitioning $T_{i_0}$ satisfying
  $\sum_j (2 \rk_M(T_{i_0,j}) - 1) < \rk(D|T_{i_0})$. However, this
  contradicts our formula for $\rk(D)$, since
  $\{T_{i_0,j}, T_i : i \neq i_0, j\}$ is a partition of $[n]$ and
  \[
    \sum_j (2 \rk_M(T_{i_0,j}) - 1) + \sum_{i \neq i_0}  (2 \rk_M(T_{i}) - 1) <
    \sum_i (2 \rk_M(T_{i}) - 1).
  \]
 
  We now have $\rk(D) = \sum_i \rk(D|T_i)$ and hence each $T_i$ is
  a component of $D$ (and the elements of $S$ are all singleton
  components if we choose the earlier expression for $\rk(D)$).

  For the last claim, if $D(M,M)$ is not expected rank then there is a
  witness $T_1 \sqcup \dots \sqcup T_\ell \sqcup S = [n]$ to this with
  either $\ell >1$ or $S \neq \emptyset$ and both cases imply $D$ is
  disconnected.
\end{proof}

We also need the following result.
\begin{corollary}\label{cor:expected rank implies connected}
  Let $M$ be a loopless matroid. If $D(M,M)$ has expected rank then $M$ is connected.
\end{corollary}
\begin{proof}
  Assume that $M$ is disconnected. Say $M = M' \oplus M''$. Then
  $D(M,M) = D(M',M') \oplus D(M'',M'')$ by \cref{prop:D directsum}. It
  follows that the rank of $D(M,M)$ is at most
  $2\rk(M') - 1 + 2\rk(M'')-1 = 2\rk(M)-2$ which is always less than
  expected.
\end{proof}

\begin{remark}
Our results do not require a structural characterization of the loopless matroids $M$ such that
$D(M,M)$ has lower than expected rank,
and we do not have a substantially different characterization to \Cref{def:D} itself.
But unpacking \Cref{def:D} a little in this case may help the reader.

Write $r=\rk M$ and $E$ for the ground set of $M$.
As in the proof of \cref{lem:not expected rank implies D disconnected}, 
in a witness for $\rk_D(E)$ we may assume $T_1\sqcup\cdots\sqcup T_\ell=E$.
%because adjoining a new singleton $T_i$ doesn't change the quantity being minimized in \Cref{def:D}. 
By this assumption,
\begin{equation}\label{eq:structure of M with low rank D}
2r-1>\rk_D(E)=\sum_{i=1}^\ell(2\rk_M(T_i)-1)
\end{equation}
i.e.
\[\sum_{i=1}^\ell\rk_M(T_i)<r+\frac{\ell-1}2.\]
Necessarily $\ell\ge2$, because if $\ell=1$ then $T_1=E$ and the inequality reads $r<r$.
On the other hand each summand on the right hand side of \Cref{eq:structure of M with low rank D} is at least~1,
so $\ell<2r-1$.

Let's compare $M$ to the matroid direct sum $M'=\bigoplus_{i=1}^\ell M|T_i$,
since the left hand side $\sum_{i=1}^\ell\rk_M(T_i)$ above is $\rk M'$.
For $\ell=2,3$ we have $\rk M'=r$, so $M\cong M'$ is disconnected.
In general $\rk M'-\rk M\in[0,(\ell-2)/2]$;
that is, there is a relaxation of~$M$ of the same rank
which is an at most $\lfloor(\ell-2)/2\rfloor$-fold truncation of a disconnected matroid with $\ell$ components, whose ranks satisfy the inequality above.
This is strong structural information when $|E|$ is large compared to $r$ and therefore to~$\ell<2r-1$.
\end{remark}

\subsection{The external activity complex}\label{ssec:ea complex}
The notion of external activity from \cref{ssec:basis activities} has a significant generalization to the present setting.

Let $A = \CC[x_1,\dots,x_n,y_1,\dots,y_n]$ and fix a sufficiently
generic weight vector $w = (w_x,w_y) \in \RR^n \times \RR^n$, say with
$\QQ$-linearly independent entries. For $S\subseteq[n]$, write $x_S$
and $y_S$ for the respective monomials $\prod_{i\in S}x_i$ and
$\prod_{i\in S}y_i$ in~$A$.

For a basis $B$ of~$D$ and $i \in [n] \setminus B$, 
let $C_{B,i}$ be the fundamental circuit. 
Of all the decompositions $C_{B,i}=I_1\sqcup I_2$ afforded by \Cref{lem:C(D)},
we say that the \newword{$w$-initial decomposition} is the one such that $y_{I_1}x_{I_2}$ has greatest $w$-weight. We will sometimes suppress the dependence on $w$ and simply refer to the \textit{initial decomposition} of $C$.

\begin{definition}\label{def:externally k-active}
We say that $i$ is \newword{externally 1-active}, resp.\ \newword{externally 2-active}, with respect to a basis $B$ of~$D$ and a weight vector $w$,
if $i$ is in $I_1$, resp.\ $I_2$, in the initial decomposition $C_{B,i}=I_1\sqcup I_2$.
For a basis $B$ of $D$, define $E_k(B)$ to be the set of elements of~$[n]\setminus B$ that are externally $k$-active with respect to~$B$.
We write $\ea_w(B;M_1,M_2) = (|E_2|,|E_1|)$ 
and call this the \newword{external activity} statistic of~$B$ with respect to $(M_1,M_2)$.
\end{definition}

The reversed order of the components in the $\ea_w(B;M_1,M_2)$ statistic 
arranges that the contribution of an element~$i$ to the statistic matches the $\mathbb Z^2$-degree of the corresponding variable in~$y_{I_1}x_{I_2}$.

\begin{example}
Let $M_2 = U_{1,n}$, so that $D=M_1$, and suppose $w_{x,1}-w_{y,1}>\cdots>w_{x,n}-w_{y,n}$.
Then the initial decomposition of any circuit $C=I_1\sqcup I_2$ of~$D$ has $I_2=\{\min(C)\}$,
and an element $i$ is externally 2-active for a basis~$B$
if and only if it is externally active for $B$ with respect to $M_1$ alone.
\end{example}

A crucial feature of external activity of a single matroid is that the number of bases of prescribed external activity is independent of the chosen ordering of the ground set.
\Cref{cor:GEA independent of w} will show the same is true of external activity of a pair.

\begin{definition}\label{def:external activity complex}
  The \newword{external activity complex} of the pair $(M_1,M_2)$ is defined to be
  the simplicial complex $\Delta_w(M_1,M_2)$ on $\{x_1,\dots,x_n,y_1,\dots,y_n\}$ with
  facets $x_{B \cup E_1(B)}\, y_{B \cup E_2(B)}$ where $B$  is a basis of $D$.
\end{definition}

Here we are notating a simplicial complex by the monomials of its facets, as described in \cref{ssec:complexes}. The following result is immediate.
\begin{proposition}
  $\Delta_w(M_1,M_2)$ is pure of dimension $\rk(D)+n-1$.
\end{proposition}

\begin{example} %labeling as in M2
  Let $F$ be the Fano matroid and $F^{-}$ be the non-Fano matroid, as
  in \cref{ex:fano}. Label the ground set of $F$ as shown below:
  \[
\begin{tikzpicture}[scale=1.5]
   \node[shape=circle,draw=black,scale=.5] (A) at (0,0) {1};
   \node[shape=circle,draw=black,scale=.5] (B) at (90:1) {2};
   \node[shape=circle,draw=black,scale=.5] (C) at (210:1) {3};
   \node[shape=circle,draw=black,scale=.5] (D) at (-30:1) {4};
   \node[shape=circle,draw=black,scale=.5] (E) at (30:0.5) {5};
   \node[shape=circle,draw=black,scale=.5] (F) at (-210:0.5) {6};
   \node[shape=circle,draw=black,scale=.5] (G) at (-90:0.5) {7};
   \draw (A)--(B)--(F)--(C)--(G)--(D)--(E)--(B);
   \draw (A)--(E);
   \draw (A)--(F);
   \draw (A)--(G);
   \draw (A)--(D);
   \draw (A)--(C);
   \draw (E) edge[bend right=60] (F);
   \draw (F) edge[bend right=60] (G);
   \draw (G) edge[bend right=60] (E);
%   \draw (0,0) circle (0.5);
%   \draw (90:1) -- (-30:1)--(210:1)--cycle;
%
%   \draw (90:1)--(0,0);
%   \draw (210:1)--(0,0);
%   \draw (-30:1)--(0,0);
%
%   \draw (30:0.5)--(0,0);
%   \draw (150:0.5)--(0,0);
%   \draw (270:0.5)--(0,0);
%
%   \fill (-0.866,-0.5) circle (1.5pt) ;
%   \fill (0.866,-0.5) circle (1.5pt);
%   \fill (0,-0.5) circle (1.5pt);
%   \fill (0,1) circle (1.5pt);
%   \fill (0,0) circle (1.5pt);
%   \fill (0.433,0.25) circle (1.5pt);
%   \fill (-0.433,0.25) circle (1.5pt);
\end{tikzpicture}
  \]
  Label $F^{-}$ so that $\{5,6,7\}$ is the unique non-spanning circuit of $F$ that is
  independent in $F^{-}$. Each circuit $C$ of $\underline{D(F,F)} = \underline{D(F^{-},F^{-})}$ has its
  $w$-initial decomposition $I_1 \sqcup I_2$ for $D(F,F)$ and $D(F^{-},F^{-})$. Here we assume that 
  $w$ is such that $I_1$ is the lexicographically first independent set of
  $F$ (resp. $F^{-}$) with $I_2= C \setminus I_1 $ also an independent set of $F$  (resp. $F^{-}$). We show the decomposition of each circuit in $D(F,F)$ and $D(F^{-},F^{-})$, supressing brackets to ease the notation.
  \[
    \begin{array}{c|c|c}
      C & \textup{decomp. of $C$ in }D(F,F) & \textup{decomp. of $C$ in }D(F^{-},F^{-}) \\\hline
      123456 & 123 \sqcup 456 & 123 \sqcup 456\\
      123457 & 123 \sqcup 457 & 123 \sqcup 457\\
      123467 & 123 \sqcup 467 & 123 \sqcup 467\\
      123567 & 125 \sqcup 367 & 123 \sqcup 567\\
      124567 & 125 \sqcup 467 & 124 \sqcup 567\\
      134567 & 135 \sqcup 467 & 134 \sqcup 567\\
      234567 & 235 \sqcup 467 & 234 \sqcup 567\\
    \end{array}
  \]
  Notice that the circuits containing $\{5,6,7\}$ are all broken
  differently in $D(F,F)$ compared to $D(F^{-},F^{-})$. 
  It follows that the external activity complexes 
  $\Delta_w(F,F)$ and $\Delta_w(F^{-},F^{-})$ are not equal.
\end{example}

We denote the Stanley-Reisner ideal of $\Delta_w(M_1,M_2)$ by $I_w(M_1,M_2)$. 
We will see later (as \Cref{prop:ideal in general}) that $I_w(M_1,M_2)$ is equal to
\[
 \left( x_{I_2}y_{I_1} : I_1 \sqcup I_2 \textup{ is the $w$-initial decomposition of a circuit of }D \right).
\]
This will ultimately be a consequence of $\Delta_w(M_1,M_2)$ being
Cohen-Macaulay and bivaluative in an appropriate sense.

We record a simple observation for later use.
\begin{proposition}\label{prop:x_i or y_i}
For each $i\in[n]$, every facet of $\Delta_w(M_1,M_2)$ has $x_i$ or $y_i$ or~both as a vertex.
\end{proposition}

\begin{proof}
By \Cref{def:externally k-active}, the sets $B$, $E_1(B)$ and $E_2(B)$ partition $[n]$ for any basis $B$ of~$D$.
\end{proof}

\Cref{prop:D tr} extends to the external activity complex.

\begin{proposition}\label{prop:Delta tr}
If $D(M_1,M_2)$ is not of the  expected rank and $M_2$ has no loops then $\Delta_w(M_1,M_2)=\Delta_w(M_1,\tr M_2)$.
\end{proposition}

\begin{proof}
Since $\underline{D(M_1,M_2)}$ and $\underline{D(M_1,\tr M_2)}$ are equal matroids, they have the same circuits.
Choosing a circuit~$C$ and comparing the equation in \Cref{lem:C(D)} for the two Dilworth truncations implies that $\rk_{M_2}(C)=\rk_{\tr M_2}(C)$.
This implies that any set $I_2\subseteq C$ is independent in~$M_2$ if and only if it is independent in $\tr M_2$.
So the set of all decompositions of~$C$ in~\Cref{lem:C(D)}, and therefore also the initial decomposition,
is the same in both settings.
\end{proof}

\begin{lemma}\label{lem:D disconnected implies Delta is a join}
  If $D(M_1,M_2)$ is disconnected and $S\sqcup T= [n]$ is a disconnection, then
  \[
    \Delta_w( M_1,M_2) = \Delta_{w_S}(M_1|S, M_2|S) * \Delta_{w_T}(M_1|T, M_2|T).
  \]
  Here $*$ is the simplicial join operation, and $w_S$ and
  $w_T$ are the restriction of $w$ to $S$ and $T$.
\end{lemma}
\begin{proof}
  Let $B$ be a basis of $D(M_1,M_2)$, so that $B \cap S$ and $B \cap T$ are bases of $D(M_1|S,M_2|S)$ and $D(M_1|T,M_2|T)$, respectively. Pick $e \notin B$, so that there is a unique circuit $C \subset B \cup e$. Then either $C \subset S$ or $C \subset T$. Say that $e \in S$. The $w_S$-initial decomposition of $C$ is the same as the $w$-initial decomposition of $C$. Thus $x_e$ appears in the facet of $\Delta_w(M_1,M_2)$ corresponding to $B$ if and only if $x_e$ appears in the facet of $\Delta_{w_S}(M_1|S,M_2|S)$ corresponding to $B \cap S$. The analogous result holds when $C \subset T$. It follows that every facet of $\Delta_w(M_1,M_2)$ is a disjoint union of a facet of $\Delta_{w_S}(M_1|S, M_2|S)$ with a facet of $\Delta_{w_S}(M_1|T, M_2|T)$, and this implies the result.
\end{proof}

\subsection{Bivaluative functions}

Being a valuation is a notion of linearity for a matroid function.
As we are working with a pair of matroids, we will use the corresponding kind of bilinearity.

\begin{definition}\label{def:bivaluative}
A function $f$ taking as input two matroids on~$[n]$, valued in an abelian group $G$, is a \newword{bivaluation} (is \newword{bivaluative})
if there exists a map of abelian groups $\widehat f:\mathbb I(\mathrm{Mat}_{[n]})\otimes\mathbb I(\mathrm{Mat}_{[n]})\to G$
such that $f(M_1,M_2) = \widehat f(\mathbf 1_{P(M_1)}\otimes\mathbf 1_{P(M_2)}).$
\end{definition}

As before, we are only interested in the sequel in pairs of matroids in our narrowed sense, those without common loops.
We will call a function valued on pairs of matroids bivaluative
if its extension by zero to $\mathrm{Mat}_{[n]}^2$ is bivaluative.
\Cref{lem:loops in subdivisions} implies that restricting attention to pairs causes no trouble: 
every linear relation that the bivaluative condition places on $f:\mathrm{Mat}_{[n]}^2\to G$
either involves only values of~$f$ at pairs or involves none of its values at pairs.

\begin{example}
  If $f_1$ and $f_2$ are any two valuative functions
  $\mathrm{Mat}_{[n]} \to R$ taking values in a ring $R$
  then $(M_1,M_2) \mapsto f_1(M_1)\cdot f_2(M_2)$ is bivaluative.

  As a simplest example of a bivaluative function we can take the
  product of the indicator functions
  $(M_1,M_2) \mapsto \mathbf{1}_{P(M_1)} \cdot
  \mathbf{1}_{P(M_2)}$. This example displays
  a phenomenon observed elsewhere in our study: When we specialize to the diagonal, obtaining
  the function
  $M \mapsto \mathbf{1}_{P(M)} \cdot \mathbf{1}_{P(M)} =
  \mathbf{1}_{P(M)}$, we obtain a valuative function of $M$.
  This is at least slightly unusual since the function
  $M \mapsto f_1(M) \cdot f_2(M)$ should be expected to be ``quadratic
  in $M$'', rather than linear.
\end{example}

\begin{example}\label{ex:bigraded Euler char is bivaluative}
  For any $i,j$, the function sending a pair $(M_1,M_2)$ to 
  \[
    \textstyle\bigwedge^i [\mathcal{Q}^\vee_{M_1}] \cdot \bigwedge^j [\mathcal{Q}^\vee_{M_2}] \in K_0(X_n)
  \]
  is bivaluative by \cref{prop:exterior power is valuative}. Restricting this function to the diagonal $M_1 = M_2$
  implies that the resulting function of one matroid is valuative (\cite[Proposition~5.8]{BEST}). A consequence of this is that the equivariant Euler characteristic,
  \[
    \chi^T \left({\textstyle\bigwedge}^i [\mathcal{Q}^\vee_{M_1}]
      \cdot {\textstyle\bigwedge^j} [\mathcal{Q}^\vee_{M_2}] \right)
    \in \ZZ[T_1^{\pm 1}, \dots, T_n^{\pm 1}],
  \]
  is bivaluative. It follows that the $K$-polynomials in \cref{Athm:geometric} and \cref{Athm:main intro} are bivaluative.
\end{example}

Just as in \Cref{cor:realizable determines valuation} for valuations,
\Cref{prop:Schubert expansion} implies that a bivaluation is determined by 
the values it takes when both inputs are Schubert matroids.
Note that it is not sufficient to take ``compatible'' Schubert matroids
$\Omega(S_\bullet,a_\bullet)$ and $\Omega(S_\bullet,b_\bullet)$ with the same chain of subsets.

\begin{remark}\label{rem:expansion into expected rank}
All the Schubert matroids appearing in the expansion of~\cref{prop:Schubert expansion}
are relaxations of the matroid $M$ on the left hand side.
Since the rank of~$D(M_1,M_2)$ cannot decrease on relaxing the arguments,
one can in fact show that 
the values taken by a bivaluation on pairs $(M_1,M_2)$ for which $D(M_1,M_2)$ has expected rank
are determined by its values on pairs of Schubert matroids whose diagonal Dilworth truncation is itself of expected rank.
\end{remark}

%There is a version of this fact that pays attention to the rank of the diagonal Dilworth truncation.
%
%\begin{lemma}\label{lem:expected rank for realizable implies all}
%Let $f$ and~$g$ be bivaluations.
%If $f(M_1,M_2) = g(M_1,M_2)$ when $(M_1,M_2)$ is a pair of Schubert matroids
%for which $D(M_1,M_2)$ has expected rank,
%then $f(M_1,M_2) = g(M_1,M_2)$ for all pairs $(M_1,M_2)$ of matroids
%for which $D(M_1,M_2)$ has expected rank.
%\end{lemma}
%
%\begin{proof}
%Let $(M_1,M_2)$ be any pair for which $D(M_1,M_2)$ has expected rank.
%Expanding both tensor factors in $\mathbf 1_{P(M_1)}\otimes\mathbf 1_{P(M_2)}$ using \Cref{prop:Schubert expansion}
%yields a linear combination of tensors $\mathbf 1_{P(\Omega_1)}\otimes\mathbf 1_{P(\Omega_2)}$,
%in which by \Cref{lem:Schubert summands are relaxations}, for $k=1,2$, 
%$\Omega_k$~is a Schubert matroid that is a relaxation of $M_k$ with $\rk(\Omega_k)=\rk(M_k)$.
%By \Cref{lem:D weak increasing}, $D(\Omega_1,\Omega_2)$ has expected rank.
%So $\widehat f$ and $\widehat g$ have equal images on all of these tensors
%and therefore on $\mathbf 1_{P(M_1)}\otimes\mathbf 1_{P(M_2)}$.
%\end{proof}

\section{The Schubert variety of a pair of linear spaces}\label{sec:schubert}
Let $L_1,L_2$ be linear subspaces of $\mathbf{C}^n$ with associated
matroids $M_1$ and $M_2$. Assume that $(M_1,M_2)$ is a pair of
matroids in our earlier sense, i.e., the two have no loops in common. We call
$(L_1,L_2)$ a pair of linear spaces.

For $i=1,2$, we have the associated tautological subbundle
$\mathcal{S}_{L_i} \subset \underline{\CC^n}$ over the permutohedral
toric variety $X_n$. The total space of
$\mathcal{S}_{L_1} \oplus \mathcal{S}_{L_2}$ has a dense subset
\begin{align}\label{eq:SL1++SL2}
  \{ (u.t, v.t , \overline{t}) : u \in L_1, v \in L_2, t \in T \} \subset \CC^n
  \times \CC^n \times X_n.
\end{align}
Here $\overline{t}$ is the image of $t$ in the dense torus of $X_n$ 
and $.$ is the action of $T$ on $\CC^n$, which is still by inverse scaling,
so e.g.\ $(u.t)_i=u_i/t_i$.
\begin{definition}
The
projection of $\mathcal{S}_{L_1} \oplus \mathcal{S}_{L_2}$ to
$\CC^n \times \CC^n$ is
\begin{align}\label{eq:hatYL1L2}
  \hat Y_{L_1,L_2} = \operatorname{cl} \{ (u.t, v.t) : u \in L_1, v \in                                                       
  L_2, t \in T\} \subset \CC^n \times \CC^n,
\end{align}
where $\operatorname{cl}(-)$ is Zariski closure. This is an
irreducible affine variety which we call the \newword{affine Schubert variety} of the pair $(L_1,L_2)$. 
The image of $\hat Y_{L_1,L_2}$ under the natural rational map $(\CC^n)^2 = (\CC^2)^n \dashrightarrow  (\PP^1)^n$, 
denoted  $Y_{L_1,L_2}\subset(\PP^1)^n$, 
is the \newword{Schubert variety} of the pair $(L_1,L_2)$.
\end{definition}
We remark that it is sufficient to let $t$ range over  $T' = \{t \in T : t_1 = 1\}$ in \Cref{eq:SL1++SL2,eq:hatYL1L2}. Also, when proving proving \cref{Athm:geometric}, there is no real loss of generality in our  restriction to $L_1$ and $L_2$ not having common loops --- the affine Schubert variety is unchanged by deletion of any common loop coordinates.

Let $I({L_1,L_2})$ denote the prime ideal of $\hat Y_{L_1,L_2}$, which
is a $\ZZ^2 \times \ZZ^n$-graded ideal in $A$ since $\hat Y_{L_1,L_2}$ is visibly stable under the corresponding torus actions.

In this section we compute a Gr\"obner basis for $I(L_1,L_2)$, prove
that $\hat Y_{L_1,L_2}$ is normal, Cohen-Macaulay
and has rational singularities, and also compute formulas for the
multidegree and $K$-polynomial of $A/I(L_1,L_2)$ of $Y_{L_1,L_2}$ in
various multigradings.
This will establish \Cref{Athm:geometric} and \Cref{Athm:initial ideal}.

\subsection{Cartwright-Sturmfels${}^*$~ideals}
It is convenient to utilize the language and results of Conca, De
Negri and Gorla \cite{conca18,conca20} to phrase some results in this
section and, later, our main results. As is typical in this paper, we
will view $A$ with the $\ZZ^n$-multigrading, where $x_i$ and $y_i$
have the same degree $e_i \in \ZZ^n$, and leverage our knowledge about
what happens in this multigrading in order to extend our results to
the $\ZZ^2$ or $\ZZ^2 \times \ZZ^n$-gradings.

In this section we will only consider monomial weight orderings on $A$
for weights $w = (w_x,w_y) \in \RR^n \times \RR^n$ that have
$(w_x)_i > (w_y)_i$ for all $i \in [n]$.
\emph{Outside} this section the ideals important to us are $\ZZ^2$-graded,
meaning that by adding $a\gg0$ to each $(w_x)_i$ 
we can assume $w$ satisfies this assumption without changing the $w$-initial ideals appearing elsewhere.

\begin{definition}
  A $\ZZ^n$-graded ideal $I \subset A$ is Cartwright-Sturmfels${}^*$, abbreviated
  {\CS}, if there is an ideal $J \subset A$ whose generators are
  monomials in $x_1,\dots,x_n$ for which the $\ZZ^n$-graded
  $K$-polynomial of $A/I$ is equal to that of $A/J$.
\end{definition}
Note that equality of the $\ZZ^n$-graded $K$-polynomials is equivalent
to equality of the corresponding $\ZZ^n$-graded Hilbert series, which is how the definition of {\CS} was orginally phrased.  A
systematic treatment of {\CS} ideals in a more general setting is given
in \cite{conca20}.

The $\ZZ^n$-graded $K$-polynomial of $A/I$ is unaffected by any
invertible change of coordinates of the form
$x_i \mapsto a_i x_i + b_iy_i$, $y_i \mapsto c_i x_i + d_i y_i$,
$i \in [n]$. It is also unaffected by taking an initial
ideal. Applying a sufficiently general ``upper triangular'' change of
coordinates (i.e., taking each $b_i = 0$ above) and then taking the
initial ideal produces the $\ZZ^n$-graded generic initial ideal ---
the \newword{$\ZZ^n$-graded gin}, denoted $\operatorname{gin}(I)$. Then, $I$ is {\CS} if and only if its
$\ZZ^n$-graded gin is generated by monomials in the variables $x_1,\dots,x_n$.

At several points later we will use the following result.
\begin{proposition}[{\cite[Proposition~1.9]{conca20}}]\label{prop:csstar betti numbers}
  Let $I\subset A$ be a {\CS} ideal. Then there is a unique monomial ideal
  $J\subset A$ extended from $\CC[x_1,\dots,x_n]$ whose $\ZZ^n$-graded
  $K$-polynomial equals that of $I$. Furthermore, $A/I$ and $A/J$ have
  the same $\ZZ^n$-graded Betti numbers.
\end{proposition}

The {\CS} ideals in this section essentially all arise from embedding a
linear space in a product of projective spaces, 
where the cited results suggest we might expect this property. 
However, later we will produce monomial {\CS} ideals for non-realizable pairs of
matroids and there the appearance of the {\CS} property is  more
surprising (\textit{cf.}\ \cite{huangLarson}).

\subsection{A Gr\"obner basis for $I(L_1,L_2)$}\label{ssec:gb}
Our goal is to prove the following result, and ultimately relate the
affine Schubert variety of the pair $(L_1,L_2)$ to its external
activity complex. 
\begin{theorem}\label{thm:initialIdeal}
  The initial ideal of $I(L_1,L_2)$ in the monomial weight order determined by~$w$
  is equal to
\[
  \left( y_{I_1} x_{I_2}  : I_1 \sqcup I_2 \textup{ is the $w$-initial decomposition of a circuit of } D(M_1,M_2)\right).
\]
\end{theorem}
Our notation for the $w$-initial ideal of an ideal $I$ will be $\operatorname{LT}_w(I)$.

To set up the proof, we reconsider the Schubert variety $Y_{L_1,L_2} \subset (\PP^1)^n$ of the pair $(L_1,L_2)$. This is at once seen  to be the closure of the image of
$\PP(L_1 \oplus L_2)$ in $(\PP^1)^n$ under the rational map
\[ (u,v) \mapsto \big((\langle u,e_1 \rangle: \langle v,e_1\rangle),
    \dots , (\langle u,e_n \rangle: \langle v,e_n\rangle)\big)\]
Here $\langle -,-\rangle$ is the standard inner product on $\CC^n$ and
$e_1,\dots,e_n$ is the standard basis of $\CC^n$. Now
$\hat Y_{L_1,L_2}$ is the multiaffine cone over $Y_{L_1,L_2}$ and
$A/I({L_1,L_2})$ is its multihomogeneous coordinate ring.

\begin{proposition}\label{prop:gin}
  The ideal $I(L_1,L_2)$ is {\CS}. The $\mathbf{Z}^n$-graded gin of $I(L_1,L_2)$ is
  \[
    \left( x_C : C \textup{ is a circuit of }D \right) \subset A.
  \]
  The $\mathbf{Z}^n$-graded multidegree of $A/I(L_1,L_2)$ is
  \[
    \sum_{B \textup{ a basis of } D} \prod_{i \notin B} T_i.
  \]
\end{proposition}
\begin{proof}
  \cite[Lemma 4.13]{binglin} asserts that, for a specific term order,
  the $\ZZ^n$-graded gin is the ideal called $I_o$ in \cite[Proposition 4.5]{binglin},
  which equals $( x_C : C$ is a circuit of $D)$ by our \cref{lem:C(D)}.
  The same follows for any term
  order under consideration by \cite[Corollary~1.10]{conca20}.
 The claim about the multidegree follows from
\cite[Theorem~1.1]{binglin}.
\end{proof}

The next lemma gives equations that vanish on $\hat
Y_{L_1,L_2}$.

\begin{lemma}\label{lem:determinant}
  Let $C=\{i_1,\dots,i_c\} \subset [n]$ be any subset. 
  Let $\mathbf{A}$ be a $\rk({M_1}) \times n$ matrix realizing $L_1$ in~$\CC^n$ 
  and $\mathbf{B}$ be a $\rk({M_2}) \times n$ matrix realizing $L_2$ in~$\CC^n$. 
  Let $\mathbf{A}_C$ and $\mathbf{B}_C$ be the column selected submatrices indexed by $C$ and
  let $\mathbf{A}_C^\perp$ and $\mathbf{B}_C^\perp$ denote matrices whose rows form a
  basis for the (right) nullspace of $\mathbf{A}_C$ and $\mathbf{B}_C$, respectively. 
  Say that $\mathbf{A}_C^\perp$ and $\mathbf{B}_C^\perp$
  have columns $\mathbf{a}_1,\dots,\mathbf{a}_c$ and
  $\mathbf{b}_1,\dots,\mathbf{b}_c$. Form the 
  matrix
  \begin{align}\label{eq:detMatrix}
        d_C=\begin{bmatrix}
      \mathbf{a}_1 x_{i_1} & \mathbf{a}_2 x_{i_2} & \dots  & \mathbf{a}_c x_{i_c} \\
      \mathbf{b}_1 y_{i_1} & \mathbf{b}_2 y_{i_2} & \dots  & \mathbf{b}_c y_{i_c} \\
    \end{bmatrix}.
  \end{align}
  Then, the $|C| \times |C|$ minors of $d_C$ lie in $I(L_1,L_2)$.
  
  If $\rk_{M_1}(C) + \rk_{M_2}(C) = |C|$ then the monomials
  $y_{J_1}x_{J_2}$ that show up in the determinant of~$d_C$ 
  are exactly those with $J_k$ an independent set of $M_k$ and $J_1 \sqcup J_2 = C$.
\end{lemma}
\begin{proof}
  In the matrix $d_C$, set $x = u.t$ and $y=v.t$, with $u \in L_1$,
  $v \in L_2$ and $t \in T$. 
  Then a factor of $1/t_{i_j}$ appears in each entry of the $j$th column.
  By construction there is a linear dependence among these columns:
  scaling column~$j$ by $t_{i_j}$ and
  summing over all $1\le j\le c$ yields the zero vector. It follows
  that the $|C| \times |C|$ minors of this specialization of~$d_C$ must vanish
  (these minors are zero by definition if there are more columns than rows). 
  This shows that the $|C| \times |C|$ minors of $d_C$ vanish
  on a dense subset of $\hat Y_{L_1,L_2}$, and hence these minors
  vanish on all of $\hat Y_{L_1,L_2}$. This proves our first claim.
  
  Now assume that $\rk_{M_1}(C) + \rk_{M_2}(C) = |C|$. Then $d_C$ is a
  square matrix. Assume that $y_{J_1} x_{J_2}$ is a monomial appearing
  in $\det(d_C)$. It is immediate that $J_1 \sqcup J_2 = C$. Set
  $x_j = 0$ for $j \in J_1$, and $y_j = 0$
  for $j \in J_2$.
  % Set
  % $x_j = 0$ and $y_j = 1$ for $j \in J_1$, and $y_j = 0$, $x_j = 1$
  % for $j \in J_2$.
  After sorting the columns, 
  $d_C$ becomes block diagonal with a non-zero determinant 
  and hence the two blocks must each be square matrices of full rank. 
  This forces $\{\mathbf{a}_j : j \in J_2\}$ and
  $\{\mathbf{b}_j : j \in J_1\}$ to be linearly independent
  sets. Since $|J_2| \leq |C| - \rk_{M_1}(C)$ and
  $|J_1| \leq |C|- \rk_{M_2}(C)$ it follows that
  \[
    |C| = |J_1| + |J_2| \leq 2|C| - \rk_{M_1}(C) - \rk_{M_2}(C) = |C|
  \]
  and hence the above inequalities are actually equalities.

  It follows that $J_2 \subset C$ is a basis of 
  $(M_1|C)^*$, and hence $J_1 = C\setminus J_2$ is a basis of
  $M_1|C$. That is, $J_1$ is independent in $M_1$. Similarly, $J_2$ is
  independent in $M_2$.

  On the other hand, if $J_1 \sqcup J_2 = C$ is a decomposition with
  $J_k$ independent in $M_k$ then $J_2$ is a basis of $(M_1|C)^*$ and
  $J_1$ is a basis of $(M_2|C)^*$. Setting $x_j = 0$ for $j \in J_1$
  and $y_j = 0$ for $j \in J_2$ in~$d_C$ yields a full rank matrix whose
  determinant is a multiple of $y_{J_1}x_{J_2}$. Since we can
  specialize the variables after taking the determinant, this monomial
  $y_{J_1}x_{J_2}$ must have appeared in the orginal determinant.
\end{proof}

We can prove our description of the initial ideal of $I(L_1,L_2)$:
\begin{proof}[Proof of \cref{thm:initialIdeal}] 
  Let $C = \{i_1, \dots, i_c\}$ be a circuit of $D(M_1,M_2)$ with $w$-initial
  decomposition $I_1 \sqcup I_2$ where $I_k$ is independent in $M_k$, $k=1,2$,
  and $y_{I_1}x_{I_2}$ of greatest $w$-weight. By
  \cref{lem:determinant}, the determinant of the square matrix $d_C$
  will have leading term $y_{I_1}x_{I_2}$ under the monomial weight order determined by $w$. It follows that 
  $y_{I_1}x_{I_2} \in \operatorname{LT}_w I(L_1,L_2)$. 

  Let $I_0$ be the ideal generated by the determinants $\det(d_C)$, where
  $C$ is a circuit of $D$, so that $I_0 \subset I(L_1,L_2)$. For the $\ZZ^n$-graded generic initial
  ideals we have
  \[
    ( x_C : C \textup{ is a circuit of }D) \subset
    \operatorname{gin}(I_0) \subset \operatorname{gin}
    I(L_1,L_2).
  \]
 The left containment
  follows since applying $(x_i,y_i) \mapsto g_i (x_i,y_i)$ for generic upper triangular
  $g_i \in \operatorname{GL}_2(\CC)$, $i \in [n]$, to $d_C$ and taking the
  determinant gives a polynomial with 
  a scalar multiple of~$x_C$ as one term
  and all other terms involving $y$ variables. The ideal on
  the left is equal to the ideal on the right, by \cref{prop:gin}, and
  hence all containments are equalities. It follows that $I_0$ is 
  {\CS}. The $\ZZ$-graded Betti numbers of $I_0$,
  $I(L_1,L_2)$ and $( x_C : C \textup{ is a circuit of }D) $ are all
  equal, by \cref{prop:csstar betti numbers}. Comparing the first $\ZZ^n$-graded Betti numbers we find that $I_0 = I(L_1,L_2)$.

  Finally, any minimal set
  of generators of a {\CS} ideal is a universal
  Gr\"obner basis, by \cite[Proposition~1.12]{conca20}, and this
  proves our theorem.
\end{proof}
We have shown the first part of \cref{Athm:initial ideal}. The second (and final) part, regarding  the $\ZZ^n$-graded Betti numbers, follows from \cref{prop:csstar betti numbers,prop:gin}. We now turn to consequences of these results.
\begin{corollary}\label{cor:initIdealSRComplex}
  Let $(M_1,M_2)$ be a pair of matroids realized by a pair of linear spaces $(L_1,L_2)$ over $\CC$. Then:
  \begin{enumerate}
  \item   $A/I(L_1,L_2)$ is a Cohen-Macaulay normal domain.
  \item   The Stanley-Reisner complex of $\operatorname{LT}_w I(L_1,L_2)$
    is equal to $\Delta_w(M_1,M_2)$.
  \item   $\Delta_w(M_1,M_2)$ is a Cohen-Macaulay complex.
  \end{enumerate}
\end{corollary}
\begin{proof}
  The first item follows directly from \cite[Theorem~1]{brion} and
  \cref{prop:gin} since the $\ZZ^n$-graded multidegree of $Y_{L_1,L_2}$ is
  mutiplicity free. See also \cite[Theorem~4.1]{conca18}.
  
  For the second item, let $\Delta_0$ denote the Stanley-Reisner complex of
  $\operatorname{LT}_w I(L_1,L_2)$. We will show it has the same
  facets as $\Delta_w(M_1,M_2)$. Since $A/I(L_1,L_2)$ is Cohen-Macaulay, $\Delta_0$ is pure and all its facets
  have cardinality $n + \rk(D)$. By \cref{prop:multidegree} and
  \cref{prop:gin}, the generating function for the $\ZZ^n$-degrees of
  the complements of the facets of $\Delta_0$ is
  \[
    \sum_{B \textup{ a basis of }D} \prod_{i \notin B} T_i.
  \]
  It follows that all the facets of $\Delta_0$ are of the form
  $x_{B \cup F_1} y_{B \cup F_2}$, where $B$ is a basis of $D$ and
  $F_1 \sqcup F_2 = [n] \setminus B$. Furthermore, there is exactly
  one such facet for each basis. If $e \in F_1$ then $e\notin B$ and
  hence $e \in E_1(B)$ or $e \in E_2(B)$. If $e \in E_2(B)$, let
  $I_1 \sqcup I_2$ be the $w$-initial decomposition of the fundamental
  circuit of $B \cup e$. Since $e \in I_2$ we find that
  $y_{I_1}x_{I_2}$ divides $x_{B \cup F_1} y_{B \cup F_2}$, 
  which contradicts \cref{thm:initialIdeal}; hence $F_1 \subset E_1(B)$. 
  Similarly, $F_2 \subset E_2(B)$. Since
  $|F_1 \sqcup F_2| = |E_1(B) \sqcup E_2(B)|$ we must have
  $F_k = E_k(B)$ and we are done.

  The third item is an immediate
  consequence of the first and second.
\end{proof}
% The proof of the first item uses realizability in a crucial way. We
% will later demonstrate these results for \emph{any} pair of
% matroids $(M_1,M_2)$.

The following fact appearing in the above argument is worth recording separately.
\begin{corollary}\label{cor:dim hat Y}
The dimension of $\hat Y_{L_1,L_2}$ is $n+\rk(D(M_1,M_2))$.
\end{corollary}

\begin{proof}
We have that $\dim A/I(L_1,L_2)=\dim A/\operatorname{LT}_w I(L_1,L_2)$ is the cardinality of facets of~$\Delta_w(M_1,M_2)$, which is $n+\rk(D(M_1,M_2))$.
\end{proof}

We also offer the following result when $D(M_1,M_2)$ does not have expected rank.
\begin{corollary}\label{cor:hat Y hyperplane}
  If $D(M_1,M_2)$ is not of the  expected rank and $M_2$ has no loops, then $\hat Y_{L_1,L_2}=\hat Y_{L_1,H\cap L_2}$ for a general hyperplane $H$ in~$\mathbf C^n$.
\end{corollary}

\begin{proof}
For general $H$ the matroid of $H\cap L_2$ is $\tr M_2$.
\Cref{prop:D tr} implies that $\underline{D(M_1,M_2)}=\underline{D(M_1,\tr M_2)}$,
so $\hat Y_{L_1,L_2}$ and $\hat Y_{L_1,H \cap L_2}$ have the same dimension.
By construction $\hat Y_{L_1,H \cap  L_2}$ is a subvariety of~$\hat Y_{L_1,L_2}$.
Since $\hat Y_{L_1,L_2}$ is irreducible and reduced it has no proper subvarieties of the same dimension.
\end{proof}

\subsection{Rational singularities and cohomology vanishing}
We have already noted that $A/I(L_1,L_2)$ is normal and
Cohen-Macaulay. In this section we look closer at its
singularities, with the aim of giving a formula for its
$K$-polynomial. When $\mathcal{S}_{L_1} \oplus \mathcal{S}_{L_2}$ is
birational to $\hat Y_{L_1,L_2}$ the situation is simpler;
it is however essential that our results cover the case when these
varieties are not birational. To deal with this, we introduce a larger
auxiliary bundle, containing
$\mathcal{S}_{L_1} \oplus \mathcal{S}_{L_2}$ as a subbundle, which is
always birational to its collapsing. From this bundle we can extract the
needed vanishing results, which we state here.

\begin{theorem}\label{thm:pushforwards} For any pair of linear spaces
  $(L_1,L_2)$, let
  $q : \mathcal{S}_{L_1} \oplus \mathcal{S}_{L_2} \to \hat
  Y_{L_1,L_2}$ be the collapsing map. Then:
  \begin{enumerate}
  \item For all $i > 0$, $R^i q_* \mathcal{O}_{\mathcal{S}_{L_1} \oplus \mathcal{S}_{L_2}} = 0$;
  \item $q_* \mathcal{O}_{\mathcal{S}_{L_1} \oplus \mathcal{S}_{L_2}} = \mathcal{O}_{\hat Y_{L_1,L_2}}$; and furthermore,
  \item if $D(M_1,M_2)$ has expected rank, the natural map
  $q:\mathcal{S}_{L_1} \oplus \mathcal{S}_{L_2} \to \hat Y_{L_1,L_2}$ is
  a rational resolution of singularities.
  \end{enumerate}
\end{theorem}

The following lemma introduces the auxiliary bundle. To conform to the
notation in the literature, we let $\mathcal{O}(-\beta)$ be the
tautological subbundle of the linear subspace spanned by $\CC\mathbf{1}\in\CC^n$,
that we would otherwise denote $\mathcal{S}_{\CC\mathbf{1}}$. 
This is the line bundle on $X_n$ whose corresponding
piecewise polynomial on the fan $\Sigma_n$ takes a point $x$ to
$\max(x)$ \cite[Example~3.10]{BEST}.
\begin{lemma}
  For the subbundle
  $\mathcal{E} = \mathcal{S}_{L_1} \oplus \mathcal{S}_{L_2} \oplus
  \mathcal{O}(-\beta)$ of the trivial bundle
  $\underline{\CC^n}^{\oplus 3}$ let
  $\hat Z \subset (\CC^n)^{\oplus 3}$ denote its collapsing
  and $Z$ denote the rational image of $\hat Z$ in $(\PP^2)^n$.
  Then, $Z$ is normal and the natural map $\pi : \mathcal{E} \to \hat Z$ is a
  rational resolution of singularities.
\end{lemma}
\begin{proof}
  First, $\mathcal{E}$ is smooth since it is a vector bundle over the
  smooth variety $X_n$. To see that $\pi : \mathcal{E} \to \hat Z$ is
  birational we produce an inverse over a dense subset of $\hat Z$: 
  On the subset
  $\{(u.t, v.t, z\mathbf{1}.t) : u \in L_1,v \in L_2,z\in\CC^\times,t \in T'\}$ 
  an inverse to~$\pi$ is given by
  $(u.t, v.t, z\mathbf{1}.t) \mapsto ((u.t, v.t, z\mathbf{1}.t), t)$,
  because $z=(z\mathbf{1}.t)_1$ 
  and thus $t = ((z\mathbf{1}.t)/z)^{-1}$ are functions of the input.

  If we can show that $\hat Z$ has rational singularities then it
  follows that every resolution of singularities is rational, which
  will prove our result. To see this, since the
  multidegree of $\hat Z$ is multiplicity free by
  \cite[Theorem~1.1]{binglin}, it follows that the Chow class of~$Z$
  is multiplicity free. Now \cite[Theorem~0.1]{brion2} implies $\hat Z$ is
  normal.

  It follows from the results in \cite{brion2} (specifically
  Remarks 2 and 3 therein; see also \cite[Theorem~4.3]{ratsing}) that
  $Z$ has rational singularities. If
  $\mathcal{L}_1,\dots,\mathcal{L}_n$ are the restrictions to~$Z$
  of the tautological line bundles on the factors of~$(\PP^2)^n$ then
  \cite[Theorem~0.1]{brion2} implies that for all non-positive
  $a_1,\dots,a_n$ and $i >0$,
  \[
    H^i(Z, \mathcal{L}_1^{ a_1} \otimes \dots \otimes  \mathcal{L}_n^{ a_n} ) = 0
  \]
  and additionally, for strictly positive $b_1,\dots,b_n$ and $j< \dim(Z)$,
  \[
        H^j(Z, \mathcal{L}_1^{b_1} \otimes \dots \otimes  \mathcal{L}_n^{b_n} ) = 0.
  \]
  We may now apply \cite[Theorem~1]{kempf} and conclude
  that the multicone $\hat Z$ has rational singularities.
\end{proof}
What is essentially the above argument for a different subbundle appears as
\cite[Theorem~4.1]{ratsing}. The following result can be extracted
from the above proof.
\begin{corollary}
  For any pair $(L_1,L_2)$, $\hat Y_{L_1,L_2}$ has rational
  singularities.
\end{corollary}
\begin{proof}
  The key here is that $Y_{L_1,L_2} \subset (\PP^1)^n$ is multiplicity
  free, and so the same argument used in the lemma can be used to show
  that the multicone $\hat Y_{L_1,L_2}$ has rational singularities.
\end{proof}

\begin{remark}
  For a general pair $(L_1,L_2)$, take $\ell$ such that
  $\rk(D(M_1,M_2)) = \rk(M_1) + \rk(M_2) - 1 - \ell$. Let $H^\ell$
  denote a general linear space of codimension $\ell$. Then,
  $\mathcal{S}_{L_1} \oplus \mathcal{S}_{L_2 \cap H^\ell} \to
  \hat Y_{L_1,L_2}$ will be a rational resolution of singularities.
\end{remark}
\begin{proof}[Proof of \cref{thm:pushforwards}]
  We maintain the notation of the lemma. Say that $(\CC^n)^{\oplus 3}$ has coordinate ring $B$. We can identify the higher direct
  images $R^i \pi_* \mathcal{O}_{\mathcal{E}}$ with the graded $B$-module
  $H^i(X_n,\Sym(\mathcal{E}^\vee))$ by \cite[Theorem~5.1.2(b)]{weyman}.

  Since $\mathcal{E} \to \hat Z$ is a rational resolution of
  singularities, for all $i>0$ we have
  \[
    H^i(X_n,\Sym(\mathcal{E}^\vee)) = \bigoplus_{j,k} H^i(X_n,\Sym^j(\mathcal{S}^\vee_{L_1} \oplus \mathcal{S}^\vee_{L_2}) \otimes \mathcal{O}(k\beta)) =0.
  \]
  In particular, all direct summands vanish and hence, for all $i>0$,
  \[
   H^i(X_n,\Sym(\mathcal{S}^\vee_{L_1} \oplus \mathcal{S}^\vee_{L_2})) = 0.
 \]
 Identifying the left side with
 $R^i q_* \mathcal{O}_{\mathcal{S}_{L_1} \oplus \mathcal{S}_{L_2}}$ we obtain the first item in the theorem.

 For the second item, we use the fact that for a proper morphism of
 varieties with a normal target, the pushforward of the structure sheaf
 of the source is equal to the structure sheaf of the target if and
 only if the fibers of the morphism are connected (see
 \cite[Section~1.13]{debarre}). Since $Z$ is normal, Zariski's main theorem implies the fibers
 of $\pi$ are connected. This means that the fiber of a point of
 the form $(a,b,0) \in Z$ is connected. However, this is naturally
 identified with the fiber $q^{-1}((a,b))$. It follows that $q$ has
 connected fibers and hence
 $q_* \mathcal{O}_{\mathcal{S}_{L_1} \oplus \mathcal{S}_{L_2}} =
 \mathcal{O}_{\hat Y_{L_1,L_2}}$ as desired.

 For the last item in the theorem, assume that $D(M_1,M_2)$ has
 expected rank. This implies that the varieties
 $\mathcal{S}_{L_1} \oplus \mathcal{S}_{L_2}$ and $\hat Y_{L_1,L_2}$
 have the same dimension, since the latter variety has dimension
 $\rk(D) + n$ by \cref{cor:dim hat Y}. Since $q$ is dominant, there is a dense open subset
 $U \subset \hat Y_{L_1,L_2}$ where its fibers are $0$-dimensional
(this is a theorem of Chevalley; see \cite[Exercise~II.3.22(e)]{hartshorne}). One can explicitly construct
 such a $U$ by taking the collection of points where the matrices from
 \cref{ssec:gb}, whose determinants generate $I(L_1,L_2)$, all have
 nullity exactly $1$. Since the fibers of $q$ are connected they are single
 points, and hence $q: q^{-1}(U) \to U$ is bijective. This implies $q$
 is a birational isomorphism. Since
 $\mathcal{S}_{L_1} \oplus \mathcal{S}_{L_2}$ is smooth we know that
 $q$ is a resolution of singularities. That the resolution is rational
 follows from the first two items of the theorem.
\end{proof}

\subsection{$K$-polynomials of Schubert varieties}
In the previous subsections we have already proved the first two items of~\cref{Athm:geometric}. 
In this subsection we prove the third and final part.
\begin{theorem}\label{thm:Kpolyrealizable}
  The $\ZZ^2 \times \ZZ^n$-graded $K$-polynomial of $A/I(L_1,L_2)$ is
  \[
    \sum_{i,j} \chi^T( \textstyle\bigwedge^i \mathcal{Q}^\vee_{L_1} \otimes
    \bigwedge^j \mathcal{Q}^\vee_{L_2} ) (-U_1)^i (-U_2)^j.
  \]
  To obtain the $\ZZ^2$-graded $K$-polynomial replace $\chi^T$ with
$\chi$. 
\end{theorem}
This is a $\ZZ^2 \times \ZZ^n$-graded (i.e., $S \times T$-equivariant)
version of \cref{prop:K}. We will give a relatively self-contained
proof using localization, but it is possible to avoid this by
observing that the resolution described in
\cite[Theorem~5.1.3]{weyman} is equivariant with respect to the
$(n+2)$-torus $S \times T$, where $S$ acts trivially on $X_n$. Some further comments on this extension occur in \cref{ssec:cohomology of vb}. 
\begin{proof}
  We compute the multigraded Hilbert series of $A/I(L_1,L_2)$ and
  multiply by the relevant denominator. Using equivariant
  localization, we will see this yields the desired formula.
  
  Let $[V]$ denote the character of a $T$-representation $V$. By the
  second part of \cref{thm:pushforwards}, the $\ZZ^2 \times \ZZ^n$-graded
  Hilbert series of $A/I(L_1,L_2)$ is equal to
  \[
    \sum_{i,j} [H^0 \left(X_n, \Sym^i(\mathcal{S}_{L_1}^\vee ) \otimes  \Sym^j(\mathcal{S}_{L_2}^\vee )\right)]\, U_1^i U_2^j 
  \]
  By the first part of \cref{thm:pushforwards} this is equal to
  \[
    \sum_{i,j} \chi^T \left( \Sym^i(\mathcal{S}_{L_1}^\vee ) \otimes  \Sym^j(\mathcal{S}_{L_2}^\vee )\right) U_1^i U_2^j.
  \]
  We compute the Euler characteristic using equivariant localization; see
  \cite[Theorem~10.2]{BEST} for a treatment focused on the
  permutohedral variety. Using the formulas from \cref{ssec:best stuff}, the Euler characteristic, and hence the Hilbert series, is
  \[
    \sum_{\pi \in S_n} \prod_{i \in B_\pi(M_1)}
    \frac{1}{1-T_i U_1}\cdot \prod_{i \in B_\pi(M_2)}
    \frac{1}{1-T_i U_2} \cdot \prod_{i=1}^{n-1}
    \frac{1}{1-T_{\pi(i+1)}/T_{\pi(i)}  }.
  \]
  We now multiply by the appropriate denominator $\prod_{i =1}^n (1-T_i U_1) \cdot \prod_{i =1}^n (1-T_i U_2)$ to produce the
  $K$-polynomial of $A/I(L_1,L_2)$ as
  \[
    \sum_{\pi \in S_n} \prod_{i \notin B_\pi(M_1)} (1-T_i U_1) \cdot
    \prod_{i \notin B_\pi(M_2)} (1-T_iU_2) \cdot \prod_{i=1}^{n-1}
    \frac{1}{1-T_{\pi(i+1)}/T_{\pi(i)}  }.
  \]
  This is exactly the localization formula for
  \[
    \sum_{i,j} \chi^T( \textstyle\bigwedge^i \mathcal{Q}^\vee_{L_1} \otimes
    \bigwedge^j \mathcal{Q}^\vee_{L_2} ) (-U_1)^i (-U_2)^j,
  \]
  the summand for $\pi\in S_n$ coming from the local class at~$\pi$, 
  which proves our result.
\end{proof}

\begin{corollary}\label{cor:K poly realizable}
  Assume that the pair of matroids $(M_1,M_2)$ is realizable over $\CC$. Then
  the $\ZZ^2 \times \ZZ^n$ graded $K$-polynomial of
  $\Delta_w(M_1,M_2)$ is
  \[
   \sum_{i,j} \chi^T( \textstyle\bigwedge^i [\mathcal{Q}^\vee_{M_1}] \cdot
   \bigwedge^j [\mathcal{Q}^\vee_{M_2} ] (-U_1)^i (-U_2)^j.
 \]
\end{corollary}
\begin{proof}
  Say that $(M_1,M_2)$ is realized by the pair of linear spaces $(L_1,L_2)$. The $K$-polynomial of $\Delta_w(M_1,M_2)$ is that of
  $A/I_w(M_1,M_2)$, by definition. Since the $K$-polynomial does not
  change when taking an initial ideal this is the $K$-polyomial of
  $A/I(L_1,L_2)$, which is given by \cref{thm:Kpolyrealizable}.
\end{proof}

\section{Products of Chern classes}\label{sec:products}
In this section we compute the products of Chern classes of
$\mathcal{Q}_{M_1}$ and $\mathcal{Q}_{M_2}$. There are two main
results: First we give a combinatorial description of these products
related to the external activity complex. Second, we leverage this
connection to prove that the finely graded $K$-polynomial of the
external activity complex is bivaluative.

There is a quick way to anticipate some of the results in this
section, which we sketch here. By \cref{prop:C} the degree of the
Schubert variety $\hat Y_{L_1,L_2}$ is equal to
$\int_{X_n} c_{n-1}(\mathcal{Q}_{L_1} \oplus \mathcal{Q}_{L_2})$. 
There is a bigraded version of this, whose proof we omit, 
which gives the $\ZZ^2$-graded multidegree of $\hat Y_{L_1,L_2}$ as
$\int_{X_n} c(\mathcal{Q}_{L_1},u_1)c( \mathcal{Q}_{L_2} ,u_2)$. 
(We could also write the formula to sum only over 
products of individual Chern classes whose indices sum to~$n-1$.)
Since multidegrees do not change on initial degeneration, this is the
$\ZZ^2$-graded multidegree of the Stanley-Reisner ring
$A/I_w(M_1,M_2)$. By \cref{prop:multidegree}, the multidegree of the
Stanley-Reisner ring is the facet complement enumerator of the
associated simplicial complex $\Delta_w(M_1,M_2)$. Here we have
required $(M_1,M_2)$ to be realizable over $\CC$; we prove the same statement for
arbitrary pairs $(M_1,M_2)$ in \cref{cor:int cc}. Since the product
$c(\mathcal{Q}_{M_1},u_1)c(\mathcal{Q}_{M_2},u_2)$ is bivaluative, one
should anticipate a form of bivaluativity for \emph{all} faces of
$\Delta_w(M_1,M_2)$, not just facets. This is made precise in the
results that follow.

\subsection{The graph and monomial of a point}\label{ssec:graph and monomial}
We will multiply our Chern classes using the fan displacement rule
in $\mathbf R^n/\mathbf R $, and later work in the compactification $\Trop X_n$. 
Since this quotient by
$\mathbf R $ is a tropical projectivization, we will say
that an element of $\mathbf R^n$ mapping to
$p\in\mathbf R^n/\mathbf R $ provides \newword{projective
  coordinates} for~$p$. Fix vectors
$w_x=(w_{x,1},\ldots,w_{x,n}) \in \RR^n$,
$w_y=(w_{y,1},\ldots,w_{y,n}) \in \RR^n$ and write $w = (w_x,w_y)$. We will
translate the fans of the Chern classes $c_i(\mathcal Q_{M_1})$ by
$w_x$, and those of the Chern classes $c_i(\mathcal Q_{M_2})$ by
$w_y$.
By \Cref{prop:tropical moving fans}, stably intersecting 
$c_i(\mathcal Q_{M_1})+w_x$ and $c_j(\mathcal Q_{M_2})+w_y$
will produce a tropical cycle whose recession cycle is 
$c_i(\mathcal Q_{M_1})c_j(\mathcal Q_{M_2})\in A^\bullet(X_n)$.

Given $p\in\mathbf R^n/\mathbf R $,
we define an edge-labeled bipartite graph $G(p) = (V_1(p)\sqcup V_2(p), E(p))$ 
to encode its position with respect to the families of fans 
$c_i(\mathcal Q_{M_1})+w_x$ and $c_j(\mathcal Q_{M_2})+w_y$.
These graphs have appeared independently in \cite{ardilaEurPenaguiao}, where they are called \textit{intersection graphs} associated to a pair of partitions.
The definition runs as follows.
Choose projective coordinates $(p_i)$ for~$p$.
Let $V_1(p)$ be the set $\{p_i-w_{x,i}:i\in[n]\}$ of real numbers that appear as a coordinate of $p-w_x$,
and similarly $V_2(p)=\{p_i-w_{y,i}:i\in[n]\}$ the set of coordinates of $p-w_y$.
Let $E(p) = \{e_i : i\in[n]\}$ where $e_i = (p_i-w_{x,i},p_i-w_{y,i})=:(e_{i,1},e_{i,2})$;
when we must make $p$ explicit we write this edge as $e_i(p)=(e_{i,1}(p),e_{i,2}(p))$.
The sets $V_1(p)$ and $V_2(p)$ depend on the projective coordinates
but only up to adding a constant to the elements of either set, 
and our further constructions will be independent of this choice.
For many arguments, only the total order on $V_1(p)$ and~$V_2(p)$ will concern us. 
\begin{example}\label{ex:6.1}
  We give an example of~$G(p)$ and how it changes when $p$ moves along a segment. 
  Choose $p \in \mathbf{R}^4/\mathbf{R}$ with projective
  coordinates $(w_{x,1},w_{x,2},p_3,t)$. Assume that $p -w_x$ has
  entries $(0,0,e,f(t))$ with $0<e<f(t)$, and $p-w_y$ has entries
  $(a,b,c,t)$ with $b< a < c$. 
  In \Cref{fig:6.1} we draw $G(p)$ as $t$ increases from
  $a$ to $c$, placing $V_1(p)$ on the top and $V_2(p)$ on the bottom.
  \begin{figure}[htb]
   \[\begin{tikzpicture}[scale=.75]
      \node () at (2,-2) {$t=a$};
      \node[shape=circle,draw=black,scale=.5] (A) at (0,1) {1,2};
      \node[shape=circle,draw=black,scale=.5] (B) at (0,1) {1,2};
      \node[shape=circle,draw=black,scale=.5] (C) at (1,1) {3};
      \node[shape=circle,draw=black,scale=.5] (D) at (3,1) {4};
      \node[shape=circle,draw=black,scale=.5] (x) at (1,-1) {1,4};
      \node[shape=circle,draw=black,scale=.5] (y) at (0,-1) {2};
      \node[shape=circle,draw=black,scale=.5] (z) at (3,-1) {3};
      \node[shape=circle,draw=black,scale=.5,fill=red] (w) at (1,-1) {1,4};
      \path [-] (A) edge  (x);
      \path [-] (B) edge  (y);
      \path [-] (C) edge  (z);
      \path [-] (D) edge  (w);
    \end{tikzpicture}\qquad\qquad
    \begin{tikzpicture}[scale=.75]
      \node () at (2,-2) {$a<t<c$};
      \node[shape=circle,draw=black,scale=.5] (A) at (0,1) {1,2};
      \node[shape=circle,draw=black,scale=.5] (B) at (0,1) {1,2};
      \node[shape=circle,draw=black,scale=.5] (C) at (1,1) {3};
      \node[shape=circle,draw=black,scale=.5] (D) at (4,1) {4};
      \node[shape=circle,draw=black,scale=.5] (x) at (1,-1) {1};
      \node[shape=circle,draw=black,scale=.5] (y) at (0,-1) {2};
      \node[shape=circle,draw=black,scale=.5] (z) at (3,-1) {3};
      \node[shape=circle,draw=black,scale=.5,fill=red] (w) at (2,-1) {4};
      \path [-] (A) edge  (x);
      \path [-] (B) edge  (y);
      \path [-] (C) edge  (z);
      \path [-] (D) edge  (w);
    \end{tikzpicture}\qquad\qquad
    \begin{tikzpicture}[scale=.75]
      \node () at (2,-2) {$t=c$};
      \node[shape=circle,draw=black,scale=.5] (A) at (0,1) {1,2};
      \node[shape=circle,draw=black,scale=.5] (B) at (0,1) {1,2};
      \node[shape=circle,draw=black,scale=.5] (C) at (1,1) {3};
      \node[shape=circle,draw=black,scale=.5] (D) at (5,1) {4};
      \node[shape=circle,draw=black,scale=.5] (x) at (1,-1) {1};
      \node[shape=circle,draw=black,scale=.5] (y) at (0,-1) {2};
      \node[shape=circle,draw=black,scale=.5] (z) at (3,-1) {3,4};
      \node[shape=circle,draw=black,scale=.5,fill=red] (w) at (3,-1) {3,4};
      \path [-] (A) edge  (x);
      \path [-] (B) edge  (y);
      \path [-] (C) edge  (z);
      \path [-] (D) edge  (w);
    \end{tikzpicture}
  \]
    \caption{The graphs $G(p)$ of \Cref{ex:6.1}.}\label{fig:6.1}
  \end{figure}
  The vertices in each graph are drawn so that the coordinates of
  $p-w_x$ and $p-w_y$ increase when read left-to-right, and we have
  labeled each vertex according to which edges $e_i$ emanate from it
  in order to enhance the readability of these figures. We will
  maintain these conventions in future examples. One imagines that as
  the last entry of $p$ increases from $t=a$ to $t=c$ the edge $e_4$
  is dragged to the right, so that $e_{4,2}$ (shown in red) changes in
  the displayed way.
  In the proofs in this section, we will in fact often be varying $p$,
  while $w$ remains fixed. In pictures like these, this corresponds to
  edges translating horizontally while maintaining their slope.
\end{example}
That the graphs in the examples are forests is a phenomenon that
persists in general.
\begin{lemma}\label{lem:G(p) forest}
Assume that $w$ has $\QQ$-linearly independent entries. For any $p\in\mathbf R^n/\mathbf R $, $G(p)$ is a forest.
\end{lemma}

\begin{proof}
The incidences of successive edges in a cycle $e_{i_1}, \ldots, e_{i_{2\ell}}$ in~$G(p)$, 
say with the common vertex of $e_{i_1}$ and~$e_{i_{2\ell}}$ in $V_1(p)$, yield equations
\begin{align*}
p_{i_1}-w_{y,i_1} &= p_{i_2}-w_{y,i_2} \\
p_{i_2}-w_{x,i_2} &= p_{i_3}-w_{x,i_3} \\
&\vdots \\
p_{i_{2\ell}}-w_{x,i_{2\ell}} &= p_{i_1}-w_{x,i_1}.
\end{align*}
The sum of these equations, after subtracting $p_{i_1}+\cdots+p_{i_{2\ell}}$ from each side, is a $\mathbf Q$-linear dependence among the coordinates of~$w$,
a contradiction.
\end{proof}

Fix a pair of matroids $(M_1,M_2)$ for the rest of this section. Given $p\in\mathbf R^n/\mathbf R$ and $k=1,2$, define the chain of sets
\[S^k_\bullet(p) = \{\{j\in[n]:e_{j,k}\ge a\} : a\in \RR\}.\]
The length of $S^k_\bullet(p)$ is $|V_k(p)|$,
and each vertex $v\in V_k(p)$ 
is the transitional value of~$a$ between the successive elements
$S_{i-1}=\{j\in[n]:e_{j,k}>v\}$ and $S_i=\{j\in[n]:e_{j,k}\ge v\}$ of~$S^k_\bullet(p)$. For a vertex $v \in V_k(p)$,
we say that $M(v):=M_k|S_i/S_{i-1}$ is \emph{the minor determined by~$v$}.
Its ground set $E(M(v))\subseteq[n]$ is the set of labels $j$ of edges~$e_j$ that are incident to~$v$.

\begin{example}
  We continue our example above and compute $S^1_\bullet(p)$
  and $S^2_\bullet(p)$ when $t=a$. We have
  \begin{align*}
    S^1_\bullet(p) &: \{4\} \subset \{3,4\} \subset \{1,2,3,4\},\\
    S^2_\bullet(p) &: \{3\} \subset \{1,3,4\} \subset \{1,2,3,4\}.
  \end{align*}
  The minors determined by the vertices of $G(p)$ are
    \begin{align*}
    % \begin{tabular}{ll}
    %   $M_1/\{3,4\} $ & $M_2/\{1,3,4\}$\\
    %   $M_1|\{3,4\}/ \{4\} $ & $M_2|\{1,3,4\}/ \{3\} $\\
    %   $M_1|\{4\} $ & $M_2|\{3\} $
    % \end{tabular}
    \begin{tabular}{lll}
      $M_1/\{3,4\} $, & $M_1|\{3,4\}/ \{4\} $, & $M_1|\{4\}, $ \\
      $M_2/\{1,3,4\}$, & $M_2|\{1,3,4\}/ \{3\}, $ & $M_2|\{3\} $,
    \end{tabular}
  \end{align*}
  where the minor determined by a vertex has the same relative positioning as in
  our earlier drawing of $G(p)$.
\end{example}

\Cref{prop:chern classes of Q} shows that, for $k=1,2$,
the ranks in~$M_k$ of the sets $\{i\in[n]:e_{i,k}\ge a\}$ 
determine how the point $p$ sits with respect to the translated fans $c_i(\mathcal Q_{M_k})$.
We encode this data concretely by assigning to~$p$ a squarefree monomial $m(p)$ in~$A$.
Let $x_i\mid m(p)$ if and only if $i$ is independent of $\{j\in[n]:e_{j,1}>e_{i,1}\}$ in~$M_1$,
and symmetrically $y_i\mid m(p)$ if and only if $i$ is independent of $\{j\in[n]:e_{j,2}>e_{i,2}\}$ in~$M_2$.
That is, $x_i\mid m(p)$, respectively $y_i\mid m(p)$, if and only if $i$ is not a loop in~$M(e_{i,1})$, respectively $M(e_{i,2})$. 

\begin{example}\label{ex:6.4}
  We consider $M_1$ to be rank $2$ on $[4]$ with $1,2$ parallel and
  $M_2$ to be rank $2$ on $[4]$ with $1,3$ parallel (all other elements are pairwise not parallel). Continuing with $p$
  as above, we have
  \begin{align*}
    \begin{tabular}{lll}
      $M_1/\{3,4\}= U_{0,\{1,2\}} $, & $M_1|\{3,4\}/ \{4\}= U_{1,\{3\}} $, & $M_1|\{4\} = U_{1,\{4\}}, $ \\
      $M_2/\{1,3,4\}  = U_{0,\{2\}}$, & $M_2|\{1,3,4\}/ \{3\}= U_{0,\{1\}} \oplus U_{1,\{4\}}$, & $M_2|\{3\} = U_{1,\{3\}} $,
    \end{tabular}
  %  \begin{tabular}{ll}
  %    $M_1/\{3,4\} = U_{0,\{1,2\}}$ & $M_2/\{1,3,4\} = U_{0,\{2\}}$\\
  %    $M_1|\{3,4\}/ \{4\} = U_{1,\{3\}}$ & $M_2|\{1,3,4\}/ \{3\} = U_{0,\{1\}} \oplus U_{1,\{4\}}$\\
  %    $M_1|\{4\} = U_{1,\{4\}}$ & $M_2|\{3\} = U_{1,\{3\}}$
  %  \end{tabular}
  \end{align*}
  and we see that $m(p)= x_3x_4 y_3 y_4$.

  Varying $p$ so that its graph is drawn below,
  \[
  \begin{tikzpicture}[scale=.75]
      \node[shape=circle,draw=black,scale=.5] (A) at (0,1) {1,2};
      \node[shape=circle,draw=black,scale=.5] (B) at (0,1) {1,2};
      \node[shape=circle,draw=black,scale=.5] (C) at (-1,1) {3,4};
      \node[shape=circle,draw=black,scale=.5] (D) at (-1,1) {3,4};
      \node[shape=circle,draw=black,scale=.5] (x) at (1,-1) {1,3};
      \node[shape=circle,draw=black,scale=.5] (y) at (0,-1) {2};
      \node[shape=circle,draw=black,scale=.5] (z) at (1,-1) {1,3};
      \node[shape=circle,draw=black,scale=.5] (w) at (-3,-1) {4};
      \path [-] (A) edge (x);
      \path [-] (B) edge (y);
      \path [-] (C) edge (z);
      \path [-] (D) edge (w);
    \end{tikzpicture}\] we encourage the reader to check that
  $m(p) = x_1 x_2 x_3 x_4 y_1 y_2 y_3$.
\end{example}

If $m$ is a squarefree monomial in~$A$, let $\sigma(m)=\{p\in\mathbf R^n/\mathbf R:m(p)=m\}$.
We will refer to  $\sigma(m)$ as the \newword{cell} associated to $m$.

\begin{proposition}\label{prop:stable intersection}
We have
\[(c_i(\mathcal Q_{M_1})+w_x)\cap(c_j(\mathcal Q_{M_2})+w_y)
= \bigcup_{\deg_{\ZZ^2}(m)\ge(\rk(M_1)+i,\rk(M_2)+j)} \sigma(m).\]
These fans intersect transversely as tropical cycles.
In particular $\sigma(m)$ has dimension $n-1+\rk(M_1)+\rk(M_2)-\deg(m)$ if nonempty.
\end{proposition}

\begin{proof}
A point $p\in\mathbf R^n/\mathbf R$ lies in the support of $c_i(\mathcal Q_{M_1}) + w_x$
if and only if a chain $S'_\bullet:S'_0\subset S'_1\subset\cdots$ refining $S^1_\bullet(p)$ is of the form given in \cref{prop:chern classes of Q},
i.e.\ all minors $M_1|S'_{k+1}/S'_{k}$ are either loops or rank~1 uniform matroids,
and $\operatorname{corank}(M_1)-i$ of them are loops.

One coarsest refinement of~$S^1_\bullet(p)$ with minors of the correct form
inserts between any two successive sets $S_{k-1}$ and~$S_k$ of $S^1_\bullet(p)$
\begin{enumerate}
\item a maximal chain of sets between $S_{k-1}$ and $S_{k-1}\cup F_0$,
where $F_0$ is the set of elements of~$S_k$ dependent on $S_{k-1}$ in~$M_1$; and
\item the sets $S_{k-1}\cup F_\ell$ where $F_0,\ldots,F_r$
is a maximal chain of flats of $M_1|S_k/S_{k-1}$,
with $r=\rk_{M_1}(S_k)-\rk_{M_1}(S_{k-1})$.
\end{enumerate}
The minimal number of loops in a coarsest refinement of~$S^1_\bullet(p)$ is the sum of the sizes of all the sets~$F_0$;
these elements must each appear as a loop in any suitable refinement, even if it's not of the form above.
By repeating the operation of inserting a set $S'_{k+1}\setminus\{j\}$ between a pair $S'_k\subset S'_{k+1}$ that determine a rank~1 uniform matroid, 
it is also possible to produce refinements of~$S^1_\bullet(p)$ with any greater number of loops up to $\operatorname{corank}(M_1)$.

The elements $j\in[n]$ appearing in one of the sets $F_0$ above are exactly those for which $x_j\nmid m(p)$.
Therefore $n-\deg_{\ZZ^2}(m(p))_1$ is the minimum number of elements of~$[n]$ that can be a loop minor in a chain $S'_\bullet$.
It follows that $p$ lies on $c_i(\mathcal Q_{M_1}) + w_x$ if and only if
$\operatorname{corank}(M_1)-i\ge n-\deg_{\ZZ^2}(m(p))_1$, i.e.\ $\deg_{\ZZ^2}(m(p))_1\ge\rk(M_1)+i$.
This and the corresponding argument for $M_2$ proves the proposition as a claim about intersection of supports.

The tropical cycles $c_i(\mathcal Q_{M_1})+w_x$ and $c_j(\mathcal Q_{M_2})+w_y$
intersect transversely by \Cref{rem:tropical moving lemma},
because the choice of $w$ with $\mathbf Q$-linearly independent coordinates
guarantees that $w$ lies on no rational hyperplane.
The claimed dimension of~$\sigma(m)$ follows by linear algebra from the definition of transversality.
\end{proof}

\begin{example}\label{ex:cells}
  Let $M_1 = U_{2,3}$ and $M_2 = U_{1,3}$. The cells
  $\sigma(m) \subset \RR^3/\RR$, although they are not all balls, 
  give the structure of a stratified space to $\RR^3/\RR$
  as shown in \Cref{fig:cells}, 
  after choosing coordinates as done in \cref{ex:chern classes}. 
  Here $D(M_1,M_2) = U_{2,3}$. We have labeled the $2$-dimensional cells by
  their indexing monomial. The monomials of other cells are the least
  common multiple of their bounding faces.

  The $0$-dimensional intersection
  $(c_1(\mathcal{Q}_{M_1})+w_x) \cap (c_1(\mathcal{Q}_{M_2})+w_y)$
  consists of the points $F$ and $G$, whose associated monomials are
  $x_{123}y_{23}$ and $x_{123} y_{13}$, respectively. The
  $0$-dimensional intersection
  $(c_0(\mathcal{Q}_{M_1})+w_x) \cap (c_2(\mathcal{Q}_{M_2})+w_y)$ is
  the singleton $\{w_y\}$, whose associated monomial is $x_{12}y_{123}$.
  \begin{figure}[h]
    \centering
    \includegraphics{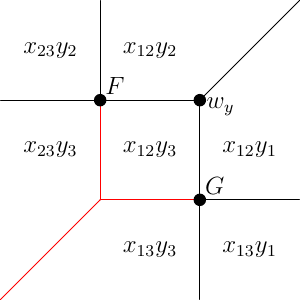}   
    \caption{Intersecting Chern classes of $\mathcal{Q}_{M_1}$ and $\mathcal{Q}_{M_2}$.}\label{fig:cells}
  \end{figure}
  In total, of the nonempty cells,
  $7$ are $2$-dimensional, $8$ are $1$-dimensional, and $3$ are $0$-dimensional.
  Note that the $1$-dimensional cell $\sigma(x_{123}y_3)$,
  displayed in red (with its apex at the point in $\RR^2$
  corresponding to $w_x$), is not a traditional line segment. That
  this cell is an open subset of a tropical line serves as a strong
  motivation for results to come. The perturbed Chern classes
  $c_i(\mathcal{Q}_{M_1}) + w_x$ and $c_j(\mathcal{Q}_{M_2})+w_y$ from
  \cref{ex:chern classes} are plainly visible in this image.
\end{example}

\subsubsection{Chern-generic points} Let us say that $p\in\mathbf R^n/\mathbf R$ is \newword{Chern-generic}
if all of the minors determined by vertices of~$G(p)$ are loops or uniform of rank~1.
The proof of \Cref{prop:stable intersection} shows that $p$ is Chern-generic if and only if 
it sits in a top-dimensional face of the smallest intersection of Chern classes containing it.
Observe that if $p$ is Chern-generic, then $x_i\mid m(p)$ if and only if $i$ is an element of a rank~1 uniform minor $M(v)$ determined by a vertex $v\in V_1(p)$, and symmetrically for $y_i$ and~$V_2(p)$.

For $i\in[n]$, define the relation $\to_i$ on $\mathbf R^n/\mathbf R$ by $p\mathrel{\to_i}q$ if in some projective coordinates $q_i=p_i$ and $q_j\le p_j$ for all~$j$. Then $\mathrel{\to_i}$ is a partial order.

\begin{lemma}\label{lem:make Chern-generic}
For every $p\in\mathbf R^n/\mathbf R$ there are Chern-generic points $q$ arbitrarily close to~$p$ such that $m(q)=m(p)$.
For any $i\in[n]$, we may choose $q$ so that also $p\mathrel{\to_i}q$.
\end{lemma}

\begin{proof}
Suppose we partition $E(M(v))=P_1\sqcup\cdots\sqcup P_\ell$ arbitrarily,
and form a graph $G'$ from~$G(p)$ by splitting $v$ into $\ell$ vertices, the $k$\/th of them incident only to the edges $\{e_j:j\in P_k\}$.
Then by acyclicity (\Cref{lem:G(p) forest}) no two of these new vertices belong to the same connected component of~$G'$.
So we may realise $G'$ as $G(p')$ for a $p'$ close to~$p$ by, for each $j\in[n]$,
setting $p'_j=p_j+\varepsilon_k$ if $e_j$ lies in the component of~$G'$ containing~$P_k$ and $p'_j=p_j$ otherwise,
for any sufficiently small pairwise distinct real numbers $\varepsilon_1,\ldots,\varepsilon_\ell$.

We prove the lemma by applying the previous paragraph successively to each vertex $v$ of~$G(p)$ that determines a minor other than a loop or a uniform matroid of rank~1,
to refine it in a coarsest way to a sequence of vertices that do determine such minors.
If $p'$ is obtained from~$p$ by this refinement, we check that $m(p')=m(p)$.
For any $j\in E(M(v))$, if $j$ is dependent on the union of $E(M(v'))$ over vertices $v'$ greater than~$v$ and in the same part as~$v$ in~$G(p)$,
this remains true in $G(p')$ because the union becomes weakly larger.
If not, then because we use a coarsest refinement, $j$ belongs to a rank~1 uniform minor in $G(p')$,
and this minor is obtained by contracting $E(M(v'))$ for all greater vertices~$v'$, so $j$ is still independent of the union of this contracted material.

All that is left to show in the statement is that we can arrange $p\mathrel{\to_i}q$.
For this we note that, given a minor $M'$ of $M_1$ or~$M_2$ with ground set $N\ni i$,
there is a coarsest chain $\emptyset\subset N_1\subset N_2\subset\cdots\subset N$ of subsets of~$N$
determining loops or uniform matroids of rank~1 in which $i\in N_1$.
Namely, if $i$ is a loop in~$M'$, let $N_1=\{i\}$; otherwise, let $N_1$ be set of all nonloops parallel to~$i$ in~$M'$.
(Then fill in the rest of the chain as in the proof of \Cref{prop:stable intersection} on $M'/N_1$.)
So if $i\in E(M(v))$ in the first paragraph, we refine using such a chain and take $\varepsilon_k=0$ for the component containing $e_i$ and $\varepsilon_k<0$ for other $k$;
if $i\not\in E(M(v))$, we take $\varepsilon_k<0$ for all~$k$.
\end{proof}

\begin{example}
  We continue \Cref{ex:6.4}, where $G(p)$ is shown here:
  \[
    \begin{tikzpicture}[scale=.75]
      \node[shape=circle,draw=black,scale=.5] (A) at (0,1) {1,2};
      \node[shape=circle,draw=black,scale=.5] (B) at (0,1) {1,2};
      \node[shape=circle,draw=black,scale=.5] (C) at (1,1) {3};
      \node[shape=circle,draw=black,scale=.5] (D) at (3,1) {4};
      \node[shape=circle,draw=black,scale=.5] (x) at (1,-1) {1,4};
      \node[shape=circle,draw=black,scale=.5] (y) at (0,-1) {2};
      \node[shape=circle,draw=black,scale=.5] (z) at (3,-1) {3};
      \node[shape=circle,draw=black,scale=.5] (w) at (1,-1) {1,4};
      \path [-] (A) edge  (x);
      \path [-] (B) edge  (y);
      \path [-] (C) edge  (z);
      \path [-] (D) edge  (w);
    \end{tikzpicture}
    \]
    The respective vertex minors are
  \begin{align*}
    \begin{tabular}{lll}
      $M_1/\{3,4\}= U_{0,\{1,2\}} $, & $M_1|\{3,4\}/ \{4\}= U_{1,\{3\}} $, & $M_1|\{4\} = U_{1,\{4\}}, $ \\
      $M_2/\{1,3,4\}  = U_{0,\{2\}}$, & $M_2|\{1,3,4\}/ \{3\}= U_{0,\{1\}} \oplus U_{1,\{4\}}$, & $M_2|\{3\} = U_{1,\{3\}} $,
    \end{tabular}
  \end{align*}
  and $m(p)= x_3x_4 y_3 y_4$. The point $p$ is not Chern-generic,
  since the minors of the top-left and bottom-middle vertices are not
  uniform rank one or loops. We wish to perturb $p$ to become
  Chern-generic. There are a few ways to do this and we demonstrate
  two ways by drawing the graphs corresponding to the perturbations here:
  \[    \begin{tikzpicture}[scale=.75]
      \node[shape=circle,draw=black,scale=.5] (A) at (-.5,1) {1};
      \node[shape=circle,draw=black,scale=.5] (B) at (0,1) {2};
      \node[shape=circle,draw=black,scale=.5] (C) at (1,1) {3};
      \node[shape=circle,draw=black,scale=.5] (D) at (3,1) {4};
      \node[shape=circle,draw=black,scale=.5] (x) at (.5,-1) {1};
      \node[shape=circle,draw=black,scale=.5] (y) at (0,-1) {2};
      \node[shape=circle,draw=black,scale=.5] (z) at (3,-1) {3};
      \node[shape=circle,draw=black,scale=.5] (w) at (1,-1) {4};
      \path [-] (A) edge  (x);
      \path [-] (B) edge  (y);
      \path [-] (C) edge  (z);
      \path [-] (D) edge  (w);
    \end{tikzpicture}\qquad\qquad
        \begin{tikzpicture}[scale=.75]
      \node[shape=circle,draw=black,scale=.5] (A) at (0,1) {1};
      \node[shape=circle,draw=black,scale=.5] (B) at (-.5,1) {2};
      \node[shape=circle,draw=black,scale=.5] (C) at (1,1) {3};
      \node[shape=circle,draw=black,scale=.5] (D) at (2.5,1) {4};
      \node[shape=circle,draw=black,scale=.5] (x) at (1,-1) {1};
      \node[shape=circle,draw=black,scale=.5] (y) at (-.5,-1) {2};
      \node[shape=circle,draw=black,scale=.5] (z) at (3,-1) {3};
      \node[shape=circle,draw=black,scale=.5] (w) at (.5,-1) {4};
      \path [-] (A) edge  (x);
      \path [-] (B) edge  (y);
      \path [-] (C) edge  (z);
      \path [-] (D) edge  (w);
    \end{tikzpicture}
  \]
  On the left we have $G(p')$ where $p \to_i p'$ for $i=2,3,4$
  (decrease $p_1$ a small amount to see this) and on the right we have
  $G(p'')$ where $p \to_1 p''$ (decrease $p_2$ and $p_4$ a small
  amount to see this). The minors for the truant vertices are resolved
  to two single loops for the top left, and $U_{1,\{4\}}$,
  $U_{0,\{1\}}$ for the bottom-middle. One checks that $G(p')$ and
  $G(p'')$ both have the same monomial as $G(p)$.
\end{example}

For $p\in\mathbf R^n/\mathbf R$ and $i\in[n]$,
we will associate to the edge $e_i$ of~$G(p)$ a pair of elements $(h_{i,1},h_{i,2})$ 
with $h_{i,k}\in V_k(p)\cup\{\infty\}$.
If $i$ is a loop in~$M_k$, set $h_{i,k}=\infty$.
Otherwise let $v\in V_k(p)$ be the greatest vertex such that $e_{i,k}(p)$ is in the closure of $\{j\in[n]:e_{j,k}(p)\ge v\}$ in~$M_k$, and we set $h_{i,k}=v$.
We call $h_{i,k}$ the $k$-\newword{holster} of~$e_i$. 

Note that $h_{i,k}\ge e_{i,k}$.
We have $x_i\mid m(p)$, respectively $y_i\mid m(p)$, 
if and only if $h_{i,1}=e_{i,1}$, respectively $h_{i,2}=e_{i,2}$.
So if $p$ is Chern-generic, $h_{i,k}=e_{i,k}$ if and only if $e_{i,k}$ determines a rank~1 uniform minor.

\begin{lemma}\label{lem:pull i up}
Let $p\in\mathbf R^n/\mathbf R$ and $i\in[n]$.
There exists $q$ Chern-generic with $p\mathrel{\to_i}q$
such that one of the following holds:
\begin{enumerate}
\item $m(q)=zm(p)$ for a variable~$z$; or
\item $m(q)=m(p)$, and if $C$ is the connected component of $G(q)$ containing $e_i(q)$,
then for all vertices $v$ of~$C$ and $v'$ of~$G(q)\setminus C$ in the same part of the bipartition,
we have $v>v'$.
\end{enumerate}

If either holster $h_{i,k}$ is finite and not in the connected component of~$G(p)$ that contains $e_i$,
then there always exists $q$ such that (1) holds.
\end{lemma}

\begin{proof}
We prove the lemma by iterative construction, 
producing a sequence $p\mathrel{\to_i}p^0\mathrel{\to_i}p^1\mathrel{\to_i}\cdots$ of points in turn, with $m(p^\ell)=m(p)$ for each~$\ell$, 
until we find a point meeting one of the conditions of the lemma.
Let $p^0$ be the Chern-generic point with $m(p^0)=m(p)$ and $p\mathrel{\to_i}p^0$ provided by \Cref{lem:make Chern-generic}.

Given some $p^\ell$, 
let $C^\ell$ be the connected component of~$G(p^\ell)$ containing $e_i$.
Choose projective coordinates $(p^\ell_j)$ for~$p^\ell$.
For $a>0$ real, define $p(a)$ so that $p(a)_j = p^\ell_j$ if $e_j\in C^\ell$, and otherwise $p(a)_j = p^\ell_j - a$.
Note that $p^\ell\mathrel{\to_i}p(a)$.

If there exists no $a>0$ such that $G(p(a))$ has fewer vertices than~$G(p^\ell)$,
then $q=p^\ell$ satisfies~(2).
If there does exist $a>0$ with this property, fix the least such~$a$.
By $\mathbf Q$-linear independence of the entries of~$w$,
$G(p(a))$ is obtained from $G(p^\ell)$ by identifying just two vertices $v_1>v_2$, 
which are in the same part of the vertex bipartition and consecutive in the total order,
into a single vertex~$v$.
If $G(p(a))$ is Chern-generic, then $v$ must determine a rank~1 uniform minor in~$G(p(a))$,
while $v_2$ must determine a loop minor in~$G(p^\ell)$, whose singleton ground set let us write $\{j\}$.
Let $z=x_j$ if $v\in V_1(P^a)$ and $z=y_j$ if $v\in V_2(P^a)$.
Then $m(p(a)) = zm(p)$, so $q=p(a)$ satisfies~(1).
If $G(p(a))$ is not Chern-generic, then let $p^{\ell+1}$ be the Chern-generic perturbation of~$p(a)$ with $p(a)\mathrel{\to_i}p^{\ell+1}$
given by \Cref{lem:make Chern-generic}, and iterate.
The iteration eventually terminates because there are fewer vertices in~$G(p^{\ell+1})$ than in~$G(p^\ell)$ 
which are greater than the endpoint of $e_i$ in the same part of the bipartition.

It remains to prove the last sentence of the lemma.
In the iterative case, if the $k$-holster $h_{i,k}$ of $e_i$ in~$G(p^\ell)$ is not in~$C^\ell$, we claim that the counterpart for~$C^{\ell+1}$ in~$G(p^{\ell+1})$ also holds.
Let $H=G(p(a))-v$. Then $H$ is a subgraph of $p^\ell$ and of~$p^{\ell+1}$ with the same edge labelling,
and $C^{\ell+1}\cap H\subseteq C^\ell\cap H$.
If $h_{i,k}$ is in $H$, then the identified vertex of~$H$ is the $k$-holster in~$G(p^{\ell+1})$, so our claim is true.

If $h_{i,k}$ is not in~$H$, then it equals $v_1$, because $v_2\in C^\ell$.
Since $p^\ell$ is not in the case that ended with $q=p(a)$ satisfying~(1),
both $v_1$ and $v_2$ determine rank~1 uniform minors in~$G(p^\ell)$,
so $v$ determines a loopless rank~2 minor~$M'$ in $G(p(a))$.
If $v'$ is the greatest of the vertices in $p^{\ell+1}$ refining $v$, 
then $v'$ is the only vertex of $G(p^{\ell+1})\setminus H$ in~$C^{\ell+1}$.
Since $i$ is independent of $\{j:e_{j,k}(p(a))>v\}$ in~$M_k$ but dependent on $\{j:e_{j,k}(p(a))\ge v\}$,
the flat partition axiom for~$M'$ implies that $i$ is dependent on precisely one of the rank~1 flats of~$M'$.
Among these flats are $E(M(v_1))$ and $E(M(v'))$. 
That $h_{i,k}=v_1$ implies that $i$ is dependent on~$E(M(v_1))$.
So $i$ is not dependent on $E(M(v'))$, implying that the $k$-holster of $e_i$ in~$G(p^{\ell+1})$ is not $v'$ and is therefore not in~$C^{\ell+1}$. This proves the claim.

By induction, $h_{i,k}(p^\ell)\not\in C^\ell$ for any $\ell\ge0$.
If condition~(2) of the lemma holds, then $h_{i,k}(q)\ge e_{i,k}(q)$ implies that $h_{i,k}(q)\in C$, which is a contradiction.
So the construction must have instead produced $q$ satisfying~(1).
\end{proof}

\begin{example}
  We continue our running example. Starting with $p$, we see that the holster $h_{1,1}$ is
  $e_{3,1}$ and this is in a different connected component from
  $e_1$. By decreasing $p_2$, $p_3$ and $p_4$ we can we
  can arrive at a Chern-generic $q$ whose graph is shown here:
  \[
        \begin{tikzpicture}[scale=.75]
      \node[shape=circle,draw=black,scale=.5] (A) at (0,1) {1,3};
      \node[shape=circle,draw=black,scale=.5] (B) at (-1,1) {2};
      \node[shape=circle,draw=black,scale=.5] (C) at (0,1) {1,3};
      \node[shape=circle,draw=black,scale=.5] (D) at (2.5,1) {4};
      \node[shape=circle,draw=black,scale=.5] (x) at (1,-1) {1};
      \node[shape=circle,draw=black,scale=.5] (y) at (-1,-1) {2};
      \node[shape=circle,draw=black,scale=.5] (z) at (2,-1) {3};
      \node[shape=circle,draw=black,scale=.5] (w) at (.5,-1) {4};
      \path [-] (A) edge  (x);
      \path [-] (B) edge  (y);
      \path [-] (C) edge  (z);
      \path [-] (D) edge  (w);
    \end{tikzpicture}
  \]
  The monomial of $q$ is $x_1 x_3 x_4 y_3 y_4 = x_1 m(p)$. This
  monomial is not maximal, and a witness to this is that
  $h_{2,2} = e_{4,2}$. If we decrease $q_1,q_3,q_4$ we can obtain
  a Chern-generic $q'$ as follows:
  \[
        \begin{tikzpicture}[scale=.75]
      \node[shape=circle,draw=black,scale=.5] (A) at (0,1) {1,3};
      \node[shape=circle,draw=black,scale=.5] (B) at (.5,1) {2};
      \node[shape=circle,draw=black,scale=.5] (C) at (0,1) {1,3};
      \node[shape=circle,draw=black,scale=.5] (D) at (2.5,1) {4};
      \node[shape=circle,draw=black,scale=.5] (x) at (1,-1) {1};
      \node[shape=circle,draw=black,scale=.5] (y) at (.5,-1) {2,4};
      \node[shape=circle,draw=black,scale=.5] (z) at (2,-1) {3};
      \node[shape=circle,draw=black,scale=.5] (w) at (.5,-1) {2,4};
      \path [-] (A) edge  (x);
      \path [-] (B) edge  (y);
      \path [-] (C) edge  (z);
      \path [-] (D) edge  (w);
    \end{tikzpicture}
  \]
  We now have $q \mathrel{\to_2} q'$ and $m(q') = y_2 m(q) = x_1 x_3 x_4 y_2 y_3 y_4$. Finally, decreasing $q'_2$ and $q'_4$ we arrive at $q''$ whose graph is here:
  \[
    \begin{tikzpicture}[scale=.75]
      \node[shape=circle,draw=black,scale=.5] (A) at (0,1) {1,2,3};
      \node[shape=circle,draw=black,scale=.5] (B) at (0,1) {1,2,3};
      \node[shape=circle,draw=black,scale=.5] (C) at (0,1) {1,2,3};
      \node[shape=circle,draw=black,scale=.5] (D) at (2,1) {4};
      \node[shape=circle,draw=black,scale=.5] (x) at (1,-1) {1};
      \node[shape=circle,draw=black,scale=.5] (y) at (0,-1) {2,4};
      \node[shape=circle,draw=black,scale=.5] (z) at (2,-1) {3};
      \node[shape=circle,draw=black,scale=.5] (w) at (0,-1) {2,4};
      \path [-] (A) edge  (x);
      \path [-] (B) edge  (y);
      \path [-] (C) edge  (z);
      \path [-] (D) edge  (w);
    \end{tikzpicture}
  \]
  Here we have $m(q'') = x_2 m(q') = x_1 x_2 x_3 x_4 y_2 y_3 y_4$, and
  this monomial is maximal in $\Delta_w(M_1,M_2)$.
\end{example}

\subsection{Agreement with $D(M_1,M_2)$ and $\Delta_w(M_1,M_2)$}

We will see in this subsection that the monomials appearing as $m(p)$
form an order filter of the simplicial complex $\Delta_w(M_1,M_2)$,
containing all of its facets. 
The main result is a computation of the degree of the product of Chern classes
\[\int_{X_n} c(\mathcal{Q}_{M_1},u_1)\,c(\mathcal{Q}_{M_1},u_2)\]
in terms of external activity. Throughout this subsection we will write
$D$ and $\Delta_w$ for $D(M_1,M_2)$ and $\Delta_w(M_1,M_2)$.

\begin{proposition}\label{prop:m(p) basis of D}
Let $m$ be a monomial in~$A$ maximal under divisibility among those appearing as $m(p)$.
Then either $x_i$ or $y_i$ divides~$m$ for every~$i$, and 
\[\{i\in[n]:x_iy_i\mid m\}\] 
is a basis of~$D$.
In particular $\deg(m)=\rk(D)+n$.
\end{proposition}

\begin{proof}
Let $p$ be Chern-generic with $m(p)=m$.
If an edge $e_i$ of~$G(p)$ had $h_{i,k}\ne e_{i,k}$ for both $k=1,2$,
then both $M(e_{i,1})$ and~$M(e_{i,2})$ would be loops, 
so both $e_{i,1}$ and $e_{i,2}$ would be leaves of~$G(p)$, i.e.\ $e_i$ would be disconnected from the rest of~$G(p)$.
At least one of the $h_{i,k}$ must be finite because $i$ is not a common loop of $M_1$ and~$M_2$.
Then a $q$ meeting condition~(1) in \Cref{lem:pull i up} would exist, contradicting maximality of~$m$.

Let $B=\{i\in[n]:x_iy_i\mid m\}$.
We first show that $B$ is independent in~$D$.
Let $G(p)|B$ be the subgraph of~$G(p)$ containing exactly the edges $\{e_i:i\in{B}\}$ (and no isolated vertices).
For any $i\in[n]$, define a bipartite multigraph $G_i$ which adds to $G(p)|B$ a single edge $h_i=(h_{i,1},h_{i,2})$ whose endpoints are the holsters of~$e_i$.
The vertex set $V_1(G_i)\sqcup V_2(G_i)$ of~$G_i$ is either the same as that of~$G(p)|B$ or contains exactly one more vertex $\infty$.
In the former case, $G_i$ has exactly one cycle; in the latter, $G_i$ remains a forest.
Therefore, $G_i$ has an orientation where each vertex has outdegree at most~1,
and $\infty$, if present, has outdegree~0.
For $k=1,2$, define $I_k$ to be the multiset of labels of edges of~$G_i$ oriented out of~$V_k(G_i)$,
where $h_i$ has label~$i$ and any $e_j$ has label~$j$.
In fact $I_k$ is a set, because if $e_i$ is an edge of~$G_i$ then $h_i$ and~$e_i$ are parallel edges so they have been given opposite orientations.
By construction $I_1+I_2=B+\{i\}$.
Also $I_k$ is independent in $M_k$, 
because if we label $\{e_{i,k}:i\in I_k\}$ as $\{v_1,\ldots,v_{|I(k)|}\}$ with $v_1>v_2>\cdots$,
then the $M(v_i)$ are an ascending chain of rank~1 uniform minors of~$M_k$,
each containing one element of~$I_k$.
So $B$ is independent in~$D$ by \Cref{lem:I(D)}.

Now, because $m$ is maximal, for every $i$ the point $q_i$ produced by \Cref{lem:pull i up} must satisfy condition~(2).
Let $C_i$ be the connected component of $G(q_i)$ containing $e_i(q_i)$.
In the sequence $(C^\ell)$ of connected components containing $e_i$ produced in the proof of \Cref{lem:pull i up},
if $C^{\ell+1}$ differs from $C^\ell$, then $C^\ell$ has had its edges partitioned among multiple subgraphs, 
and $C^{\ell+1}$ is chosen as the one of these subgraphs containing $e_i$.
Therefore if $j\in C_i$, we have that $C_j$ and~$C_i$ are isomorphic as edge-labeled graphs, as using $e_j$ instead of~$e_i$ will choose the same $C^{\ell+1}$ at each step.
It follows that, writing $T_i=\{j:e_j\in C_i\}$, the set $\{T_i:i\in[n]\}$
is a partition of~$[n]$.

Since $m(q_i)=m$, we can read from~$m$ whether a vertex of $G(q_i)$ determines as a minor a loop or a rank~1 uniform matroid.
Since the vertices of~$C_i$ are an upper set in $V_k(G(q_i))$ for $k=1,2$,
the value of $\rk_{M_k}(T_i)$ is 
the number of rank~1 uniform minors determined by vertices in $V_k(G(q_i))\cap C_i$.
If $e_j\in C_i$ with $j\not\in B$, then $e_{j,k}\ne h_{j,k}$ for some~$k$, so $e_{j,k}$ is a leaf.
Deleting all these leaves, we obtain a subtree $C_i|B$ whose edge labels are $B\cap T_i$
and whose vertices are exactly the vertices of~$C_i$ determining rank~1 uniform minors.
Therefore
\[
\rk_{M_1}(T_i)+\rk_{M_2}(T_i)-1 = \#V(C_i|B)-1
= \#E(C_i|B)
=|B\cap T_i|.
\]
Using the partition $\{T_i\}$ in the definition of $\rk_D([n])$, it follows that
\[\rk(D)\le\sum_{T_i}(\rk_{M_1}(T_i)+\rk_{M_2}(T_i)-1)
=\sum_{T_i}|B\cap T_i|
=|B|.\]
Therefore $B$ is a basis of~$D$.
\end{proof}

\begin{lemma}\label{lem:h_0(G(p))}
If $p$ is Chern-generic, then the number of components of~$G(p)$ is $\rk(M_1)+\rk(M_2)-\deg(m(p))+n$.
\end{lemma}
By \Cref{prop:m(p) basis of D}, if moreover $m(p)$ is maximal so $\deg(m)=\rk(D)+n$, then the number of components of~$G(p)$ is $\rk(M_1)+\rk(M_2)-\rk(D)$.
In particular, $G(p)$ is a tree when $D$ has expected rank.

\begin{proof}
If $p$ is Chern-generic, then every vertex determines as a minor either a loop or a rank~1 uniform matroid.
The number of vertices determining loops is $2n-\deg(m(p))$,
while the number determining rank~1 uniform matroids is $\rk(M_1)+\rk(M_2)$.
Since $G(p)$ is acyclic, its number of components is its Euler characteristic
\[|V(G(p))|-|E(G(p))|=2n-\deg(m(p))+\rk(M_1)+\rk(M_2)-n.\qedhere\]
\end{proof}

\begin{proposition}\label{prop:bases of D give m(p)}
For every basis $B$ of~$D$, there is a point $p\in\mathbf R^n/\mathbf R$ such that $x_By_B\mid m(p)$
and $m(p)$ is maximal under divisibility.
If $D$ has expected rank, then $p$ is unique.
\end{proposition}

\begin{proof}
  We begin by showing that the product of the Bergman classes of $M_1|B$ and $M_2|B$ is non-zero.
By \cite[Proposition~3.1]{speyerTLS}, the Bergman classes of two matroids have nonzero product in the Chow ring of the permutohedral variety 
if the intersection of the two matroids in the sense described there (that is, with spanning sets given by intersections of constituent spanning sets; not Edmonds' matroid intersection) has no loops.
% We take the matroids $M_1|B$ and~$M_2|B$.

We first handle the case that $D$ has expected rank. 
By \Cref{lem:I(D)}, for every $i\in B$, there exist independent sets $I_1$ of~$M_1$ and $I_2$ of~$M_2$ such that $I_1 \cup I_2=B \cup \{i\}$.
Then $I_1$ and $I_2$ are already bases of $M_1$ and $M_2$, respectively, by \cref{lem:I(D)}. It is impossible that $M_1|B$ and $M_2|B$ have  disjoint bases, since then $|B| \geq \rk(M_1) + \rk(M_2)$ and this contradicts $D$ being expected rank. 
It follows that the intersection of $M_1|B$ and $M_2|B$ has rank one, with $\{i\} = I_1 \cap I_2$ as a basis. 
This means the intersection of $M_1|B$ and $M_2|B$ is loopless and so the Bergman classes $c_{\rm top}(\mathcal Q_{M_1|B})$ and~$c_{\rm top}(\mathcal Q_{M_2|B})$ have nonzero product.

For the case that $D$ is not expected rank 
we begin by restricting to the non-loop elements $S$ of $M_2$, which
does not change that $B$ is a basis of $D(M_1,M_2)|S = D(M_1|S,M_2|S)$
(\cref{prop:D|S}) since the elements not in $S$ are all loops of
$D$. Since $(M_k|S)|B = M_k|B$, we can reduce to the case that $M_2$ is
loopless when checking $c_{\rm top}(\mathcal Q_{M_1|B}) c_{\rm top}(\mathcal Q_{M_2|B}) \neq 0$. To complete the argument, we will show
that the product of the Bergman classes with a power of the divisor
$\alpha = c_1( \mathcal{Q}_{U_{|B|-1,B}})$ is nonzero. Since
\[
  c_{\rm top}(\mathcal Q_{M_1|B}) c_{\rm top}(\mathcal Q_{M_2|B})\cdot \alpha
  = c_{\rm top}(\mathcal Q_{M_1|B}) c_{\rm top}(\mathcal Q_{\tr M_2|B})
\]
by \cite[Lemma~5.1]{huhKatz} and
$\underline{D(M_1,M_2)} = \underline{D(M_1, \tr M_2)}$ by \cref{prop:D tr}, we may reduce to the case when $D$ has expected rank by induction on $\rk(M_1)+\rk(M_2)-\rk(D)$.

Reunifying the two cases,
we compute the product of Bergman classes by the fan displacement rule and see that
\[Y:=(c_{\rm top}(\mathcal Q_{M_1|B})+w_{x,B})\cap (c_{\rm top}(\mathcal Q_{M_2|B})+w_{y,B})\]
is a nonempty tropical linear space whose dimension in $\mathbf R^B/\mathbf R$
is
\begin{equation}\label{eq:dim Berg Berg}
  \dim Y = \rk_{M_1}(B)+\rk_{M_2}(B)-|B|-1,
\end{equation}
where $w_{x,B}$ is the vector $(w_{x,i}:i\in B)$ and similarly for~$w_{y,B}$.

Choose a point $p_B\in Y$. Fixing projective coordinates $(p_i:i\in B)$ for $p_B$, 
we will extend them to a coordinate vector $(p_i:i\in[n])$ determining the $p\in\mathbf R^n/\mathbf R$ sought in the statement of the proposition.
But first we examine a point $p'$ where $p'_i=p_i$ for $i\in B$ and
$p'_j$ is sufficiently small for $j\not\in B$ that $e_{j,k}(p')<e_{i,k}(p')$ for all $i\in B$, $j\not\in B$, and $k=1,2$.
The minor of $G(p')$ determined by any vertex $e_{i,k}(p')$ with $i\in B$ is a minor of $M_k|B$,
so by definition of the Bergman fan, all these minors are uniform of rank~1.
Other vertices of $G(p')$ determine rank~0 minors, so $m(p')=x_By_B$.

Now for $j\not\in B$ set
\begin{equation}\label{eq:holster e_j at one end}
p_j = \min(h_{j,1}(p')+w_{x,j},h_{j,2}(p')+w_{y,j}).
\end{equation}
If the minimum is attained by the $k$\/th term then we have $e_{j,k}(p)=h_{j,k}(p')=h_{j,k}(p)$,
so $j$ is a nonloop in the minor~$M(e_{j,k}(p))$, which is still of rank~1.
The two terms are not equal by $\mathbf Q$-linear independence of the coordinates of~$w$,
so the other vertex of $e_j(p)$ is strictly less than its holster and determines a loop minor.
Therefore exactly one of $x_j$ and~$y_j$ divides $m(p)$.  
We still have $x_By_B\mid m(p)$, so by the degree statement in~\Cref{prop:m(p) basis of D}, $m(p)$ is maximal.

Finally, for uniqueness when $\rk(D)=\rk(M_1)+\rk(M_2)-1$, suppose $p$ is any point as in the proposition.
As in the last proof, if $x_jy_j\nmid m(p)$ then one end of~$e_j$ is a loop,
and deletion of all these loops yields a tree with edges labeled by~$B$ whose vertices determine rank~1 uniform minors.
This implies that the projection of~$p$ to $\mathbf R^B/\mathbf R$ lies in~$Y$.
By \Cref{lem:span D implies span M_k}(1) and \eqref{eq:dim Berg Berg} 
\begin{equation*}
\dim Y=\rk(M_1)+\rk(M_2)-\rk(D)-1,
\end{equation*}
which is to say that $Y$ is a 0-dimensional tropical linear space, to wit a single point, and so the projection of~$p$ must equal this point.
Equation \eqref{eq:holster e_j at one end} is then the only choice for the coordinates $p_j$ with~$j\not\in B$.
If we had say $p_j>h_{j,1}(p')+w_{x,j}$, then at the vertex $v=h_{j,1}(p')$ of~$G(p)$, at least one $i\in E(M(v))\cap B$ would become dependent in~$M_1$ on the $E(M(v'))$ for $v'>v$, which would imply $x_i\nmid m(p)$.
If $p_j$ was greater than the right hand side of~\eqref{eq:holster e_j at one end} then both its vertices would determine loop minors and neither $x_j$ nor~$y_j$ would divide~$m(p)$.
\end{proof}

A closer study of the graphs $G_i$ and their orientations from the proof of \Cref{prop:m(p) basis of D}
allows us to make the connection between the monomials $m(p)$ and external activity.

\begin{proposition}\label{prop:maximal m(p)}
Let $p\in\mathbf R^n/\mathbf R$ be such that $x_By_B\mid m(p)$ for a basis $B$ of~$D$, and $m(p)$ is maximal under divisibility.
Then $m(p)=x_{B\cup E_1(B)}\,y_{B\cup E_2(B)}$.
\end{proposition}

\begin{proof}\label{thm:m(p) Delta_w}
We may assume $p$ is Chern-generic.
Then the rank of a set $S\subseteq[n]$ in~$M_k$, for $k=1,2$,
is bounded below by the number of vertices incident to $\{e_j:j\in S\}$ in~$V_k(p)$ that determine rank~1 uniform minors.
Let $i\in[n]\setminus B$. Then we get the same bounds for subsets of~$B\cup\{i\}$ using sets of vertices in~$G_i$ instead with $h_i$ standing in for~$e_i$,
where if a vertex $\infty$ is present we do not count it as determining a rank~1 uniform minor.

Let $G_{B,i}$ be the subgraph of $G_i$ containing the edges labeled by $C_{B,i}$, and no isolated vertices.
%For notational convenience rename the edge $h_i$ of~$G_i$ as $e_i$ within~$G_{B,i}$.
\Cref{lem:C(D)} says that
\[\rk_{M_1}(C_{B,i})+\rk_{M_2}(C_{B,i})=|C_{B,i}|.\]
By the last paragraph, the left hand side is at least the number of vertices of~$G_{B,i}$, excluding any~$\infty$, while the right hand side is the number of edges of~$G_{B,i}$.
The graph $G_i$ contains at most one cycle, and if a vertex $\infty$ is present it is acyclic;
so $G_{B,i}$ must be a connected subgraph containing, respectively, the whole cycle or the vertex $\infty$.
Moreover, the bound on the left hand side must be tight: $\rk_{M_k}(C_{B,i})=|V_k(G_{B,i})|$ for $k=1,2$.

Temporarily let $J_k$ be a set containing one element of~$E(M(v))$ for each $v\in V_k(p)\setminus V_k(G_{B,i})$ determining a rank~1 uniform minor. Then
\[\rk_{M_k}(J_k\cup C_{B,i}) = |J_k| + |V_k(G_{B,i})| 
= \rk_{M_k}(J_k) + \rk_{M_k}(C_{B,i}).\]
This implies that a subset $S\subseteq C_{B,i}$ is dependent in~$M_k$ if $J_k\cup S$ is dependent.
In particular, let $v_{k,1}>v_{k,2}>\cdots$ be the vertices of $V_k(G_{B,i})$ in decreasing order.
Then $S$ is dependent in~$M_k$ if, for any $\ell$,
\begin{equation}\label{eq:set in G_B,i dependent}
S\cap\Big(E(M(v_{k,1}))\cup\cdots\cup E(M(v_{k,\ell}))\Big)>|\ell|.
\end{equation}

Now, as in the proof of \Cref{prop:m(p) basis of D},
orient the edges of~$G_{B,i}$ so that each vertex aside from $\infty$ has outdegree~$1$,
and $\infty$ has outdegree~$0$ if present.
We further require that $h_i$ be directed away from $V_{k^*}(G_{B,i})$ for the $k^*\in\{1,2\}$ such that $e_{i,k^*}(p)=h_{i,k^*}(p)$.
Then there is a unique such orientation.
Let $I_k^{\rm in}$ be the set of $i\in[n]$ such that $e_i$ or~$h_i$ is oriented away from $V_k(G_{B,i})$, so that $C_{B,i}=I_1^{\rm in}\sqcup I_2^{\rm in}$.
We show this is the initial decomposition of~$C_{B,i}$.

Let $C_{B,i}=I_1\sqcup I_2$ be any decomposition with $I_k$ independent in~$M_k$.
The $w$-weight of the corresponding monomial $y_{I_1}x_{I_2}$ is
\begin{align}
\operatorname{wt}(y_{I_1}x_{I_2}) 
&= \sum_{j\in I_1}w_{y,j}+\sum_{j\in I_2}w_{x,j} 
\notag
\\&= \sum_{j\in I_1}(p_j-e_{j,2}(p))+\sum_{j\in I_2}(p_j-e_{j,1}(p))
\notag
\\&= \sum_{j\in C_{B,i}}(p_j-e_{j,1}(p)-e_{j,2}(p))+\sum_{j\in I_1}e_{j,1}(p)+\sum_{j\in I_2}e_{j,2}(p).
\label{eq:wt in G_B,i}
\end{align}
Note that the first subsum, $\sum_{j\in C_{B,i}}(p_j-e_{i,1}(p)+e_{i,2}(p))$, is independent of the decomposition. Let 
\[a_{k,\ell} = |\{j\in I_k : j\in B\text{ and }e_{j,k}(p)=v_{k,\ell}\text{, or }
j=i\text{ and }h_{i,k}(p)=v_{k,\ell}\}|.\]
That is, $a_{k,\ell}$ is the number of edges incident to~$v_{k,\ell}$ in~$G_{B,i}$
whose label is in~$I_k$.
Since $e_{i,k^*}(p)=h_{i,k^*}(p)$ and $e_{i,k}(p)<h_{i,k}(p)$ for $k\ne k^*$,
if we replace $e_{i,k}(p)$ by $h_{i,k}(p)$ in~\eqref{eq:wt in G_B,i} and then collect equal summands we get
\[\operatorname{wt}(y_{I_1}x_{I_2}) \le
\sum_{j\in C_{B,i}}(p_j-e_{i,1}(p)-e_{i,2}(p))+\sum_\ell a_{1,\ell}\,v_{1,\ell}+\sum_\ell a_{2,\ell}\,v_{2,\ell}.\]
with equality if and only if $i\in I_{k^*}$.

For the moment write $r=|V_1(G_{B,i})|$.
Equation~\eqref{eq:set in G_B,i dependent} says that $a_{k,1}+\cdots+a_{k,\ell}\le\ell$ for each~$\ell$.
Because the quantities $v_{1,\ell}-v_{1,\ell+1}$ are positive, this implies
\begin{align*}
\sum_\ell a_{1,\ell}\,v_{1,\ell}
&= \sum_{\ell=1}^r a_{1,\ell}\left(\sum_{k=\ell}^{r-1}(v_{1,\ell}-v_{1,\ell+1}) + v_{1,r}\right)
\\&= \sum_{k=1}^{r-1}\left(\sum_{\ell=1}^k a_{1,\ell}\right)(v_{1,k}-v_{1,k+1}) 
+ \left(\sum_{\ell=1}^r a_{1,\ell}\right)v_{1,r}
\\&\le \sum_{k=1}^{r-1}k(v_{1,k}-v_{1,k+1}) + rv_{1,r}
\\&= \sum_{k=1}^rv_{1,k}.
\end{align*}
The same is true of the sum involving $v_{2,\ell}$,
so we have
\[\operatorname{wt}(y_{I_1}x_{I_2}) \le
\sum_{j\in C_{B,i}}(p_j-e_{i,1}(p)-e_{i,2}(p))+\sum_k v_{1,k}+\sum_k v_{2,k}.\]
By construction, $C_{B,i}=I_1^{\rm in}\sqcup I_2^{\rm in}$ achieves equality in this bound,
and is the unique decomposition to do so.

Finally, we have $x_i\mid m(p)$ if and only if $h_{i,1}(p)=e_{i,1}(p)$, 
equivalently $i\in I_1^{\rm in}$,
equivalently $i\in E_1(B)$.
In the same way, $y_i\mid m(p)$ if and only if $i\in E_2(B)$.
Letting $i$ vary, the proposition follows.
\end{proof}

The main theorem of this subsection is an immediate corollary.

\begin{theorem}\label{thm:m(p) Delta_w}
The elements of $\Delta_w(M_1,M_2)$ are the monomials dividing $m(p)$ for at least one $p\in\mathbf R^n/\mathbf R$.
\end{theorem}

\begin{proof}
Let $\Delta$ be the simplicial complex of monomials dividing $m(p)$ for at least one~$p$.
Then \Cref{prop:m(p) basis of D} and \Cref{prop:maximal m(p)} imply that all facets of $\Delta$ are facets of $\Delta_w$, 
and \Cref{prop:bases of D give m(p)} and \Cref{prop:maximal m(p)} imply the converse.
\end{proof}

\begin{corollary}\label{cor:int cc}
We have
\begin{multline*}
\int_{X_n} c(\mathcal Q_{M_1},u_1)\, c(\mathcal Q_{M_2},u_2) = \\
\begin{cases}
u_1^{\rk(M_2)-1}\,u_2^{\rk(M_1)-1} &\\
\qquad\cdot\sum_{B\in\mathcal B(D)}(u_1,u_2)^{\ea_w(B;M_1,M_2)} &
\rk(D)=\rk(M_1)+\rk(M_2)-1 \\
0 & \mbox{otherwise.}
\end{cases}
\end{multline*}
\end{corollary}

\begin{proof}
The zero-dimensional products of individual Chern classes in the integral are 
$c_i(\mathcal Q_{M_1},u_1)\, c_j(\mathcal Q_{M_2},u_2)$ where $i+j=n-1$.
By \Cref{prop:stable intersection}, one such term counts the $p\in\mathbf R^n/\mathbf R$ with 
$\deg_{\ZZ^2}(m(p))\ge(\rk(M_1)+i,\rk(M_2)+j)$.
Coarsening to the $\mathbf Z$-grading we have $\deg(m(p))\ge\rk(M_1)+\rk(M_2)+n-1$,
which by \Cref{prop:m(p) basis of D} can only occur if $\rk(D)=\rk(M_1)+\rk(M_2)-1$, 
in which case equality is attained for both the $\mathbf Z$ and $\mathbf Z^2$ gradings.
That proposition and \Cref{prop:bases of D give m(p)} imply that, as we let $i$ and $j$ vary subject to $i+j=n-1$,
there is exactly one such $p$ for each basis $B$ of~$D$.
By \Cref{prop:maximal m(p)},
\begin{align*}
(\rk(M_1)+i,\rk(M_2)+j)&=\deg_{\ZZ^2}(m(p))
\\&=(\rk(D),\rk(D))+\ea_w(B;M_1,M_2).
\end{align*}
The equation follows.
\end{proof}

\begin{remark}
  We explain the geometry of the second case of \Cref{cor:int cc},
  when $D(M_1,M_2)$ is not of expected rank, for realizable matroids.
  By~\cref{cor:dim hat Y}, this means that the collapsing map
  $q:\mathcal{S}_{L_1} \oplus \mathcal{S}_{L_2} \to \hat Y_{L_1,L_2}$ is
  not dimension preserving. 
  Then the pushforward of the class of the projectivized
  bundle $q_*[\PP(\mathcal{S}_{L_1} \oplus \mathcal{S}_{L_2})]$ to
  $A^\bullet(\PP ({\CC^n \oplus \CC^n}))$ is zero. 
  Taking the pushforwards along the other side of the commutative square
  \[\xymatrix{
    \PP(\mathcal{S}_{L_1} \oplus \mathcal{S}_{L_2})\ar[r]\ar[d] &
	\PP ({\CC^n \oplus \CC^n})\ar[d] \\
	X_n\ar[r] &
	\mathrm{pt}\ ,
  }\]
  this means that
  $\int_{X_n} c(\mathcal{Q}_{L_1} \oplus \mathcal{Q}_{L_2}) = 0$
  (\textit{cf.}\ \cref{prop:C}). Since the bundles $\mathcal{Q}_{L_k}$
  are nef, this implies that when $i+j=n-1$, each of the non-negative integers
  $\int_{X_n} c_i(\mathcal{Q}_{L_1}) c_j(\mathcal{Q}_{L_2})$ is zero, in agreement with \cref{cor:int cc}.
\end{remark}
\begin{corollary}\label{cor:GEA independent of w}
The number of bases of~$D(M_1,M_2)$ of prescribed external activity is independent of~$w$.
\end{corollary}

\begin{proof}
If $\rk(D)=\rk(M_1)+\rk(M_2)-1$, the right hand side of \Cref{cor:int cc} is a generating function for bases by external activity,
so equality to the left hand side shows it independent of~$w$.

For the general case, we first account for loops in $M_2$. If $i$ is a loop of $M_2$ then $i$ is a loop, i.e.\ a circuit, of $D$. There is only one decomposition of $\{i\} = I_1 \sqcup I_2$ with $\rk_{M_1}(I_1) + \rk_{M_2}(I_2) = 1$ and this  is $I_1 = \{i\}$ and $I_2$ empty. This means that $i$ is externally $1$-active for every basis of $D$. Applying \Cref{prop:D|S}, it follows that we may discard loops of $M_2$ without changing the enumeration (although the external $1$-activity of every basis decreases by a fixed amount). We may now apply \cref{prop:D tr}, repeatedly truncating $M_2$ until $D$ becomes expected rank, and by \Cref{prop:Delta tr} this does not alter the enumeration.
%
% The general case is reduced to this one 
% by using \Cref{prop:D|S} to delete loops of~$M_2$, as these are also loops of~$D$ so appear in none of its bases,
% and then replacing $M_2$ with its truncation until $\rk(D)=\rk(M_1)+\rk(M_2)-1$ attains,
% which by \Cref{prop:Delta tr} does not change the enumeration.
\end{proof}

% The following result extends to unexpected rank easily, modulo the fact that the product of Chern classes is the multidegree of $\Delta_w$, which we have not yet shown.
\begin{corollary}\label{cor:log-concave}
  Assume that $D(M_1,M_2)$ has expected rank. For any non-negative integer
  $k$, let $a_k$ be the number of bases $B$ of $D(M_1,M_2)$ where
  $|E_1(B)| = k$. Then the sequence $a_0,a_1,a_2,\dots$ is
  log-concave. That is, for all non-negative integers $k$,
  $a_{k+1} a_{k-1} \leq a_k^2$.
\end{corollary}
\begin{proof}
  It follows from \cite[Theorem~9.13]{BEST} that
  $\int_{X_n} c(\mathcal{Q}_{M_1},u_1) c(\mathcal{Q}_{M_2},u_2)$ is a
  denormalized Lorentzian polynomial, and so its coefficients form a log-concave
  sequence. Since we have just shown that the coefficient of
  $u_1^{\rk(D)-k}u_2^k$ in the product of total Chern classes is $a_k$, we are done.
\end{proof}

\subsection{Bivaluativity}\label{ssec:bivaluativity}
The central theorem of this subsection is the formula \Cref{thm:Hilb from sigma}
for the $\mathbf Z^{2n}$-graded $K$-polynomial of $A/I_w(M_1,M_2)$.
We then show the formula bivaluative in the input matroids.

In the proof we would like unions of the form $\bigcup_{m\mid m'}\sigma(m)$,
which are closed in $N_\RR$, to be compact.
For the aesthetics of the technical apparatus we do this by compactifying $N_\RR$ to~$\Trop X_n$,
meaning that our first task is to extend the intersection graphs and associated monomials defined in \Cref{ssec:graph and monomial}
to points $p\in\Trop X_n$.
(We could have avoided the need for $\Trop X_n$ by confining $p$ to a sufficiently large closed box.)

Given any $p\in\Trop X_n$, 
suppose $p=f_\pi(d)$ for $d\in[0,\infty]^{n-1}$ in the notation of \Cref{ssec:Trop X_n}.
Let $p'=f_\pi(d')$
where $d'$ is obtained by replacing any infinite coordinates of~$d$
with finite positive real numbers sufficiently large with respect to the finite coordinates $d_i$ and the entries of~$w$.
Then define $G(p) = G(p')$ and $m(p) = m(p')$.
The monomial $m(p)$ is independent of the choice of~$p'$,
while $G(p)$ is well-defined up to isomorphism of edge-labeled bipartite graphs with total orders of the vertices in each part.
Let $\sigma_\infty(m)$ be the set of all $p\in\Trop X_n$ with $m(p)=m$.
So $\sigma_\infty(m)\cap N_\RR = \sigma(m)$.
The subscript $\infty$ is for the new points at infinity i.e.\ on the boundary of $\Trop X_n$.

\begin{example}\label{ex:compactification}
  We continue \cref{ex:cells} and compactify the picture shown in \Cref{fig:cells}.
  \begin{figure}[h]
    \centering
    \includegraphics{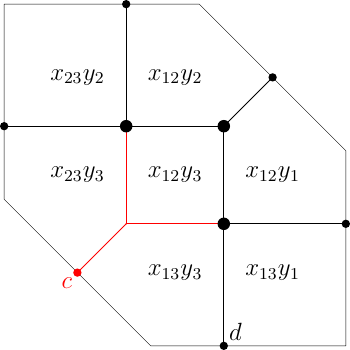}
    \caption{The compactified tropical cell complex encoding $\Delta_w(U_{2,3},U_{1,3})$.}\label{fig:compactification}
  \end{figure}
  
  The six newly added boundary points ${\color{red}\bullet}$ and ${\bullet}$ 
  belong to the $1$-dimensional cell they bound, 
  and the connected components of the boundary minus these points
  belong to the $2$-dimensional cell they bound. For example,
  $\sigma_{\infty}(x_{123}y_3)$ is the red cell, including the point
  $c$, and the other previously unbounded $1$-dimensional cells 
  become half-open line segments. As another example, the
  cell $\sigma_{\infty}(x_{13}y_3)$ is an open pentagon together
  with two of its open edges and the one vertex they share,
  these making up the bent line from $c$ to~$d$
  (but excluding $c$ and $d$ themselves).

  The compactly supported Euler characteristic of the cells $\sigma_\infty(m)$ appear in \cref{thm:Hilb from sigma} below. In \Cref{fig:compactification}, every
  $2$-dimensional cell gives $0$, except for
  $\chi_c(\sigma_{\infty}(x_{12}y_3)) = 1$. For the $1$-dimensional cells, those that are open intervals have compactly supported Euler characteristic $-1$, those that are half-open line segments give $0$, and the red cell gives $-1$. All three zero dimensional cells give $\chi_c(\textrm{pt}) = 1$. 
\end{example}
\begin{proposition}\label{prop:tropical polyhedron}
For any monomial $m$, the set $\bigcup_{m|m'}\sigma_\infty(m')\subseteq\Trop X_n$ is closed and tropically convex.
\end{proposition}

\begin{proof}
Give $\bigcup_{m|m'}\sigma_\infty(m')$ the name $Y$.
For closure, suppose $(p_i)$ is a sequence of points in~$Y$ converging in~$\Trop X_n$ to a point~$q$.
Because there are finitely many possible graphs $G(p)$ up to isomorphism of edge-labeled graphs with ordering of each part of the vertex set,
by passing to a subsequence we may assume all $G(p_i)$ are isomorphic in this sense to the same graph~$G$.
This implies also that all $m(p_i)$ are equal, by assumption to some $m'$ with $m|m'$.
Label the vertices of each~$V_k(p_i)$ as $v_{k,1}(p_i)>v_{k,2}(p_i)>\cdots$.
We can interpret each difference $v_{k,\ell}(p_i)-v_{k,\ell+1}(p_i)$ as an element of $(0,\infty]$.
Then the corresponding difference in $q$ lies in $[0,\infty]$.
That is, $G(q)$ is obtained from $G$ by merging some sets of consecutive vertices, those for which the difference tends to~0.
This implies that $m'\mid m(q)$, 
because if $i$ is independent of $\{j\in[n]:e_{j,k}>e_{i,k}\}$ in~$G$,
it is also independent of the corresponding set for~$G(q)$, which is no larger.
We conclude $q\in Y$.

For tropical convexity, suppose $p_1,p_2\in Y$, and suppose we have chosen the other objects defined in~\Cref{def:tropical convexity} with the notations used there.
For some $i\in[n]$, suppose $s(i)=(k,i)$. Then for any $j\in [n]$, $s(i)\le s(j)$ implies $(k,i)\le(k,j)$.
If $z_i=x_i$, respectively $z_i=y_i$, then $z_i\mid m(p_k)$ implies $z_i\mid m(q)$,
since, again, $z_i\mid m(q)$ requires that $i$ be independent in $M_1$, respectively $M_2$, 
of a subset of the set involved in~$z_i\mid m(p_k)$.
We conclude $\gcd(m(p_1),m(p_2))\mid m(q)$.  Because $m\mid\gcd(m(p_1),m(p_2))$, this implies $m\mid m(q)$, so $m(q)\in Y$ as desired.
\end{proof}

\begin{remark}\label{rem:tropical polyhedral complex}
A tropical polytope \cite{develinSturmfels} is the tropical convex hull of a finite set of points, the points in a minimal such set being called vertices.
Having extended tropical convexity to~$\Trop X_n$, we can speak of tropical polytopes in $\Trop X_n$.
It can be shown using \Cref{prop:tropical polyhedron} that each closure $\overline{\sigma_\infty(m)}$ is a tropical polytope.
This closure is the union of the sets $\sigma_\infty(m')$ where $m\mid m'$, and
its vertices are all intersections of the form $\sigma_\infty(m')\cap O_{S_\bullet}$ that consist of a single point.

In this way, 
the $\sigma_\infty(m)$ are the open cells of a ``tropical polyhedral complex'' structure on~$\Trop X_n$,
whose poset of cells is opposite to an up-set of the face poset of~$\Delta_w$ containing all its maxima.
As usual, the dimension of each cell of this tropical polyhedral complex
is the length of a maximal chain of cells descending from it (\Cref{prop:stable intersection}).

In all nontrivial cases this up-set of faces of $\Delta_w$ is proper.
For example, given any $p\in\mathbf R^n/\mathbf R$, there exist bases $B_k$ of $M_k$, for $k=1,2$, 
such that $x_{B_1}y_{B_2}\mid m(p)$.
(Choose $B_k$ to be the lexicographically first basis of $M_k$
with respect to any total order compatible with the total preorder $i\le j$ $\Leftrightarrow$ $e_{i,k}\ge e_{j,k}$.)
So if $m\in\Delta_w$ does not have such a factor $x_{B_1}y_{B_2}$, 
then $\sigma_\infty(m)=\emptyset$ and $m$ does not appear in the up-set.

In future work we intend to develop the connection between this tropical polyhedral complex and the minimal free resolution of $A/I(\Delta^\vee)$.
We note that, since $\Delta$ is Cohen-Macaulay, 
the complex given in \cite[\S3]{reinerWelker} as the linear strand of this minimal free resolution
is in fact the complete resolution.
\end{remark}

%Andy: I fixed notation for the equivariant K rings, so that $U1, U2, T1 .. T_n$ and $T_{ij}$ all have a fixed meaning. And I’ve tried to reserve $t$ and $s$ for elements of the $n$ torus.  But I did sneak in a $u_1$ and $u_2$ for the variables in the chern classes of $M_1$ and $M_2$. Certainly it would make sense to capitalize these

\begin{theorem}\label{thm:Hilb from sigma}
The $\mathbf Z^{2n}$-graded $K$-polynomial of $A/I_w(M_1,M_2)$ is
\[\sum_m\chi_{\rm c}(\sigma_\infty(m))\prod_{x_i\nmid m}(1-T_{1,i})\prod_{y_i\nmid m}(1-T_{2,i})\]
%the Hilbert series:
%\[\sum_m \frac{\chi_{\rm c}(\sigma_\infty(m))}{\prod_{x_i\mid m}(1-T_{1,i})\prod_{y_i\mid m}(1-T_{2,i})}\]
where $\chi_{\rm c}$ is Euler characteristic with compact support.
\end{theorem}

\begin{proof}
Given a monomial $m$ in~$A$,
let $\mathbf C[m]$ be the subring of~$A$ generated by the variables dividing $m$
(not just by $m$ itself).
The $\mathbf Z^{2n}$-graded $K$-polynomial of $\mathbf C[m]$ is
\[\prod_{x_i\nmid m}(1-T_{1,i})\prod_{y_i\nmid m}(1-T_{2,i}).\]
The monomial basis of~$A/I_w$ consists of the monomials lying in $\mathbf C[m]$ for some facet $m$ of $\Delta_w = \Delta_w(M_1,M_2)$.
We have $\mathbf C[m]\cap\mathbf C[m'] = \mathbf C[\gcd(m,m')]$,
and all gcds of facets of $\Delta_w$ are faces.
So we can write the inclusion-exclusion expansion for the Hilbert series of~$A/I_w$ with respect to the facets as
\[\sum_{m\in \Delta_w} \left(f(m)\prod_{x_i\nmid m}(1-T_{1,i})\prod_{y_i\nmid m}(1-T_{2,i})\right)\]
where $f:\Delta_w\to\mathbf Z$ is defined by the recurrence
$\sum_{m\mid m'}f(m')=1$ for all $m\in \Delta_w$.

The theorem is proved by showing that $m\mapsto\chi_{\rm c}(\sigma_\infty(m))$ obeys the same recurrence.
Excision for $\chi_{\rm c}$ gives
\[\sum_{m\mid m'}\chi_{\rm c}(\sigma_\infty(m))=
\chi_{\rm c}(\bigcup_{m\mid m'}\sigma_\infty(m)).\]
Because $m\in \Delta_w$, there is a facet $m'$ with $m\mid m'$, and then the union above includes a point of $\sigma(m')\in\mathbf R^n/\mathbf R$, so it is nonempty.
By~\Cref{prop:tropical polyhedron} the union is closed in the compact space $\Trop X_n$, therefore compact, and is tropically convex, therefore contractible
(\Cref{lem:tropically convex implies contractible}).
Since it is compact, its compactly supported Euler characteristic equals its usual Euler characteristic;
since it is contractible, this equals~$1$.
\end{proof}

\begin{proposition}\label{prop:m(p) bivaluative}
For fixed $p\in\Trop X_n$ and $m$, the integer function taking value $1$ if $m(p)=m$ and $0$ otherwise
is bivaluative in $M_1$ and~$M_2$.
\end{proposition}

\begin{proof}
We show that the function taking value $1$ if $\gcd(x_{[n]},m(p)) = \gcd(x_{[n]},m)$ and $0$ otherwise is valuative in~$M_1$;
note that $\gcd(x_{[n]},m)$ is the product of the $x$ variables dividing~$m$.
The analogue for $y$ variables and~$M_2$ is true symmetrically,
and we conclude that the function in the proposition is bivaluative, being the product of a valuative function in each matroid.

Let $v_1>\cdots>v_\ell$ be the ordering of~$V_1(m(p))$, and for $1\le k\le\ell$ let $S_k=\{i\in[n]:e_{1,i}\ge v_k\}$, with $S_0=\emptyset$.
Then our function of $M_1$ can be written
\[\prod_{k=1}^\ell\begin{cases}
1 & \mbox{the non-loops in $M_1|S_k/S_{k-1}$ are exactly $\{i\in S_k\setminus S_{k-1}:x_i\mid m\}$}\\
0 & \mbox{else}.
\end{cases}\]
The factors are valuations on $M_1|S_k/S_{k-1}$, because in any matroid subdivision every internal face has the same set of loops as the total space.
So the product is a valuation on~$M_1$, because it is the product of these valuations 
under the multiplication dual to the coproduct on matroid polytope indicator functions of \cite[\S7]{df}.
\end{proof}

\begin{corollary}\label{thm:GEAC bivaluative}
The $\mathbf Z^{2n}$-graded $K$-polynomial of $A/I_w(M_1,M_2)$ 
is a bivaluation on pairs of matroids $(M_1, M_2)$.
\end{corollary}

\begin{proof}
This is a consequence of \Cref{thm:Hilb from sigma};
the implication can be thought of as linearity of Euler integration \cite{baryshnikovGhrist}, 
but we give an argument avoiding that.

Suppose we have a relation 
\[\sum_{i=1}^s a_i\mathbf 1_{P(M_1^i)}\otimes \mathbf 1_{P(M_2^i)} = 0\]
in $\mathbb I(\mathrm{Mat}_{[n]})\otimes\mathbb I(\mathrm{Mat}_{[n]})$.
By \Cref{lem:loops in subdivisions}, 
$M_1^i$ has the same set of loops as~$M_1$, and $M_2^i$ as~$M_2$, for every $i$,
implying that every $(M_1^i,M_2^i)$ is a pair of matroids.

Let $\{Y_j:j\in J\}$ be the coarsest common refinement of 
the partitions $\{\sigma^i_\infty(m)\}$ of~$\Trop X_n$ for each pair
$(M_1^i, M_2^i)$.
Then each $Y_j$ is a polyhedrally constructible set,
i.e.\ a boolean combination of finitely many polyhedra (equivalently, halfspaces).
The refinement structure provides, for each $i=1,\ldots,s$, 
a function $m_i$ on~$J$ valued in monomials in~$A$ so that 
$Y_j\subseteq\sigma^i_\infty(m_i(j))$.
By additivity of $\chi_{\rm c}$, the sum in \Cref{thm:Hilb from sigma} can be rewritten
\[\sum_{j\in J}\chi_{\rm c}(Y_j)\prod_{x_i\nmid m_i(j)}(1-T_{1i})\prod_{y_i\nmid m_i(j)}(1-T_{2i}).\]

\Cref{prop:m(p) bivaluative} implies that, for each $j\in J$,
\[\sum_{i=1}^s a_i\prod_{x_i\nmid m_i(j)}(1-T_{1i})\prod_{y_i\nmid m_i(j)}(1-T_{2i}) = 0.\]
Therefore
\[\sum_{j\in J}\chi_{\rm c}(Y_j)\sum_{i=1}^s a_i \prod_{x_i\nmid m_i(j)}(1-T_{1i})\prod_{y_i\nmid m_i(j)}(1-T_{2i}) = 0,\]
which, after reversing the order of summation, proves the needed bivaluativity.
\end{proof}

\begin{corollary}\label{cor:indep in D is bivaluative}
For fixed $I\subseteq[n]$, the function taking value 1 if $I$ is independent in~$D$ and 0 otherwise is bivaluative in the pair $(M_1,M_2)$.
\end{corollary}

\begin{proof}
The set $I$ is independent in~$D$ if and only if $x_Iy_I\in\Delta_w$.
By the proof of \Cref{thm:Hilb from sigma}, 
$\bigcup_{x_Iy_I\mid m'}\sigma_\infty(m')$ 
has compactly supported Euler characteristic 1 if $x_Iy_I\in\Delta_w$
and is empty otherwise. 
We conclude by \Cref{prop:m(p) bivaluative} and the linearity of~$\chi_{\rm c}$.
\end{proof}

\section{The Cohen-Macaulay property and surrounding results}\label{sec:consequences}
In this section, we establish the rest of the main results of our
paper. We prove that $\Delta_w(M_1,M_2)$ is Cohen-Macaulay, describe
the generators for the Stanley-Reisner ideal and give a formula for
the $K$-polynomial. All of these results extend their analogues in the
realizable case from \cref{sec:schubert}. We use these results to
prove a formula for the bigraded Euler characteristic of
$\bigwedge^n \left( [\mathcal{Q}_{M_1}] + [\mathcal{Q}_{M_2}] \right)$
in terms of homology groups of links in $\Delta_w(M_1,M_2)$.

\subsection{Consequences of bivaluativity}
The bivaluativity of the finely graded $K$-polynomial of
$\Delta_w(M_1,M_2)$ allows us to extend results about realizable pairs
$(M_1,M_2)$ to non-realizable pairs; it is a powerful result. In this
section we deduce the remaining parts of \cref{Athm:main intro} from this property.
\begin{lemma}\label{lem:I is CSstar}
For any pair of matroids $(M_1,M_2)$ on $[n]$, the ideal $I_w(M_1,M_2)$ is {\CS} in the $\ZZ^n$-grading on $A$.
\end{lemma}
\begin{proof}
  Let $I(D)$ be the ideal in $A$ generated by monomials
  $x_C$ where $C$ is a circuit of $D$. We will show that
  $A/I_w(M_1,M_2)$ and $A/I(D)$ have the same $\ZZ^n$-graded
  $K$-polynomial, which proves that $I_w(M_1,M_2)$ is {\CS} in the $\ZZ^n$ grading on $A$.

  It follows from \Cref{cor:indep in D is bivaluative} that the
  $\ZZ^n$-graded $K$-polynomial of $I(D)$ is bivaluative, since it is
  equal to
  \[
    \sum_{J \textup{ indep.\ in }D} \prod_{j\in J}T_j \prod_{j \notin J}(1-T_j).
  \]
  Coarsening the grading in \cref{thm:GEAC bivaluative}, 
  we also see that the
  $\ZZ^n$-graded $K$-polynomial of $A/I_w(M_1,M_2)$ is
  bivaluative.

  By \cref{cor:realizable determines valuation}, it suffices to verify
  that the two $K$-polynomials are equal when $M_1$ and $M_2$ are
  realized by linear subspaces $L_1,L_2 \subset \CC^n$. In this case,
  equality follows from \cref{thm:initialIdeal,prop:gin}.
\end{proof}
This immediately yields the second part of \cref{Athm:main intro}.
\begin{theorem}
  For any pair of matroids $(M_1,M_2)$, $\Delta_w(M_1,M_2)$ is Cohen-Macaulay.
\end{theorem}
\begin{proof}
  We show that $A/I_w(M_1,M_2)$ is a Cohen-Macaulay ring. Since this property is uneffected by a change of coordinates and taking an initial ideal, it is sufficient to verify that $A/I(D)$ is a Cohen-Macaulay ring. This is a polynomial ring over the Stanley-Reisner ring of $\Delta(D)$, which is Cohen-Macaulay by  \cref{prop:matroid is cm}. This proves our result.
  % -- The above proof is slightly cleaner.
  % By
  % \cite[Corollary~1.11]{conca20} the sequence
  % $\Gamma = \{ x_i - y_i : i \in [n]\}$ is a $A/I_w(M_1,M_2)$-regular
  % sequence. The quotient $A/(I_w(M_1,M_2) + \Gamma)$ is naturally
  % identified with the Stanley-Reisner ring of the independence complex
  % $\Delta(D)$, since $I_w(M_1,M_2) + \Gamma = I(D) + \Gamma$ and the leading term ideal of the latter is generated by the circuit monomials in the $y$-variables and all of the $x$-variables. Since
  % $\Delta(D)$ is Cohen-Macaulay \cref{prop:matroid is cm} we conclude our result by applying \cref{thm:cm conditions}.
\end{proof}

It is noteworthy that, up to this point, we have not needed
to know the generators of $I_w(M_1,M_2)$. Proving the following
result directly from the definitions stands as an
interesting open problem.
\begin{proposition}\label{prop:ideal in general}
    The Stanley-Reisner ideal $I_w(M_1,M_2)$ of $\Delta_w(M_1,M_2)$ is
  \[
     \left( y_{I_1} x_{I_2} : I_1 \sqcup I_2 \textup{ is the $w$-initial decomposition of a circuit of }D \right).
  \]
\end{proposition}
\begin{proof}
  It follows from \cite[Proposition~1.9]{conca20} that $I_w(M_1,M_2)$
  has the same $\ZZ^n$-graded Betti numbers as $I(D)$. Clearly $I(D)$
  is minimally generated by the circuit monomials. By considering the first  Betti
  numbers we find the generators of $I_w(M_1,M_2)$ are of the form
  $x_{S_1}y_{S_2}$ where $S_1 \sqcup S_2=:C$ is a circuit of
  $D$. Assume that $C$ has $w$-initial decomposition $I_1 \sqcup
  I_2$. Pick $e \in S_1$. Since $C\setminus \{e\}$ is independent in
  $D$, it is contained in some basis $B$. The facet of
  $\Delta_w(M_1,M_2)$ corresponding to $B$ has
  $x_{C \setminus e} y_{C \setminus e}$ as a divisor, and $e \in S_1$
  implies $x_e$ does not divide this facet. We conclude that $y_e$
  divides this facet by \Cref{prop:x_i or y_i} which means $e \in E_2(B)$. Thus
  $S_1 \subset I_2$ and similarly $S_2 \subset I_1$. Since
  $S_1 \sqcup S_2 = I_1 \sqcup I_2$, we conclude $S_1 = I_2$ and
  $S_2 = I_1$.
\end{proof}

We can also deduce the last part of \cref{Athm:main intro}.
\begin{theorem}\label{thm:K poly in general}
  The $\ZZ^2 \times \ZZ^n$ graded
  $K$-polynomial of $A/I_w(M_1,M_2)$ is equal to
  \[
   \sum_{i,j} \chi^T ( \textstyle \bigwedge^i [\mathcal{Q}_{M_1}^\vee] \bigwedge^j [ \mathcal{Q}_{M_2}^\vee]) (-U_1)^i (-U_2)^j.
 \]
 To coarsen to the $\ZZ^2$-grading, replace $\chi^T$ with $\chi$.
\end{theorem}

\begin{proof}
  Since the finely graded $K$-polynomial of $A/I_w(M_1,M_2)$ is
  bivaluative we may coarsen the grading and conclude the
  $\ZZ^2 \times \ZZ^n$-graded $K$-polynomial of $A/I_w(M_1,M_2)$ is a
  bivaluative function $f(M_1,M_2)$. Let $g(M_1,M_2)$ denote the bigraded Euler
  characteristic displayed above. The function $g$ is bivaluative by
  \cref{ex:bigraded Euler char is bivaluative}. It suffices to verify that
  $f(M_1,M_2) = g(M_1,M_2)$ when both $M_1$ and $M_2$ are realizable over $\CC$, which is the content of \cref{cor:K poly realizable}.
\end{proof}
We have an immediate corollary of the theorem.
\begin{corollary}\label{cor:expected vanishing in K poly}
  In the $\ZZ^2$-graded $K$-polynomial of $\Delta_w(M_1,M_2)$, the
  coefficient of $U_1^i U_2^j$ is zero whenever $i > n-\rk(M_1)$ or
  $j > n-\rk(M_2)$.
\end{corollary}
\begin{proof}
  The product
  $[\bigwedge^i \mathcal{Q}_{M_1}^\vee] [ \bigwedge^i
  \mathcal{Q}_{M_2}^\vee]$ is zero in either of the stated cases.
\end{proof}
The next result was a crucial ingredient in the proof of Speyer's
tropical $f$-vector conjecture. It strengthens the previous vanishing result in the case when $M$ is a connected matroid and  $D(M,M)$ is not expected rank.
\begin{lemma}\label{lem:D unexpected rank implies omega is zero}
  Assume that $M$ is loopless and connected and $D(M,M)$ is not of expected
  rank. Then the coefficient of $U_1^{\rk(M)} U_2^{n - \rk(M)}$ in
  the $\ZZ^2$-graded $K$-polynomial of $\Delta_w(M,M)$ is zero.
\end{lemma}
\begin{proof}
  Since $D(M,M)$ is not expected rank,  \cref{lem:not expected rank implies D disconnected} implies $D(M,M)$
  is disconnected. Say that $T_1,\dots,T_\ell$ are the components of
  $D$. By \cref{lem:D disconnected implies Delta is a join}, this
  means $\Delta_w(M,M)$ is the join of the restrictions
  $\Delta_{w_{T_i}}(M|T_i,M|T_i)$. Since $M$ is connected we must have
  $\sum_i \rk_M(T_i) > \rk(M)$.

  Since the $K$-polynomial of the join of simplicial complexes is the
  product of their $K$-polynomials, we have
  \[
    \K( \Delta_w(M_1,M_2)) = \prod_{i=1}^\ell \K( \Delta_w(M_1|T_i,M_2|T_i)).
  \]
  We now look at the non-zero coefficient of $U_1^{n-p} U_2^{p}$ in
  $\K( \Delta_w(M_1,M_2))$ where $p$ is maximum. This
  is the product of such coefficients for each factor. It follows from \cref{cor:expected vanishing in K poly} that
  \[p \leq \sum_i (|T_i| - \rk_M(T_i)) < n - \rk(M).\] Thus,
  the coefficient of $U_1^{\rk(M)} U_2^{n-\rk(M)}$ is
  zero in $\K( \Delta_w(M_1,M_2))$.
\end{proof}

\subsection{Multigraded Betti numbers} We will now be concerned with the Betti numbers of $\Delta_w(M_1,M_2)$ in various gradings on $A$. To begin, note that when we compute a multigraded Betti number in a coarsening of the fine grading, we simply sum the finely graded Betti numbers that coarsen to a degree of interest. For example, 
for fixed $\mathbf{b} \in \ZZ^n$ we have
\[
  \beta_{i,\mathbf{b}}(A/I_w(M_1,M_2)) = \sum_{(\mathbf{a}_1,\mathbf{a}_2)\, :\, \mathbf{a}_1 + \mathbf{a}_2 = \mathbf{b}} \beta_{i,(\mathbf{a}_1,\mathbf{a}_2)}(A/I_w(M_1,M_2)).
\]
Since  the quantities involved are all non-negative integers, when the
coarser Betti number is zero it forces all of its refinements to be
zero as well. 
Recall \cref{prop:csstar betti numbers}, that the $\ZZ^n$-graded Betti numbers of $I_w(M_1,M_2)$ are the same as those of $I(D)$. 
These Betti numbers are the same as those of the Stanley-Reisner ring of $\Delta(D)$, and are described by \cref{prop:matroid betti numbers}. By \cref{cor:matroid betti numbers} there is only one possibly non-zero $\ZZ^n$-graded Betti number $\beta_{i,\mathbf{b}}(A/I_w(M_1,M_2))$ 
with $|\mathbf{b}|$, the sum of the entries of $\mathbf{b}$, equal to $n$, 
and it occurs in homological degree $\operatorname{corank}(D(M_1,M_2))$. This yields the following important result.
\begin{proposition}\label{prop:vanishing betti}
  Fix a multidegree
  $\mathbf{a} = (\mathbf{a}_1,\mathbf{a}_2) \in \{0,1\}^n \times
  \{0,1\}^n$ with $|\mathbf{a}| = n$. Then $\beta_{i,\mathbf{a}}(A/I_w(M_1,M_2)) = 0$ unless $\mathbf{a}_1 + \mathbf{a}_2 = \mathbf{1}$ and $i=\operatorname{corank}(D(M_1,M_2))$.

  As a consequence, for any $(p,q) \in \ZZ^2$ with $p+q = n$ we have
  $\beta_{i,(p,q)}(A/I_w(M_1,M_2)) = 0$ unless 
  $i=\operatorname{corank}(D(M_1,M_2))$.
\end{proposition}
From this we can extract  the sign of the 
 coefficients of the $K$-polynomial of $\Delta_w(M_1,M_2)$ with total degree $n$.
\begin{corollary}\label{cor:signs}
  In the finely graded, $\ZZ^2$ or $\ZZ^n$ graded $K$-polynomial of $\Delta_w(M_1,M_2)$, the sign
  of any term of total degree (i.e., $\ZZ$-degree) $n$ is equal to
  $(-1)^{\operatorname{corank}(D(M_1,M_2))}$.
\end{corollary}
\begin{proof}
For the fine grading, we write the $K$-polynomial as
  \[
    \mathcal{K}(\Delta_w(M_1,M_2)) =
    \sum_i \sum_{\mathbf{a} \in \ZZ^n \times \ZZ^n}
    (-1)^i \beta_{i,\mathbf{a}}( A/I_w(M_1,M_2)) T^\mathbf{a}.
  \]
  A term of total degree $n$ can only occur when $i = \operatorname{corank}(D(M_1,M_2))$, by \cref{prop:vanishing betti}. The other gradings are similar.
\end{proof}
From this we can deduce a bivaluative variant of the non-negativity of $\omega(M)$.
\begin{corollary}
  Assume that $(M_1,M_2)$ is a pair of matroids on~$[n]$. When $p+q = n$,
  \[
    (-1)^{\operatorname{rank}(D(M_1,M_2))}\chi ( \textstyle [\bigwedge^p \mathcal{Q}_{M_1}^\vee] [ \bigwedge^q \mathcal{Q}_{M_2}^\vee]) \geq 0.
  \]
\end{corollary}
\begin{proof}
  When $p+q=n$, the coefficient of
  $U_1^p U_2^q$ in the $K$-polynomial of $\Delta_w(M_1,M_2)$ is
  $(-1)^{n + \operatorname{corank}(D(M_1,M_2))}$ times the requisite
  Euler characteristic, by \cref{Athm:main intro}.
\end{proof}

\begin{example}
  Assume that $M_1$ and $M_2$ are matroids on $[n]$ with
  $\rk(M_1) + \rk(M_2) = n$ and that $D(M_1,M_2)$ has expected rank
  $n-1$. Then the ideal $I_w(M_1,M_2)$ is principal and the
  $\ZZ^2$-graded $K$-polynomial of $A/I_w(M_1,M_2)$ is
  $1-U_1^{\rk(M_2)} U_2^{\rk(M_1)}$. Consequently,
  \( \chi ( \textstyle [\bigwedge^i \mathcal{Q}_{M}^\vee] [
  \bigwedge^j \mathcal{Q}_{M}^\vee]) \) is zero unless $(i,j) = 0$ or
  $(i,j) = (\rk(M_2),\rk(M_1))$, in which case the Euler
  characteristics are $1$ and $-1$, respectively. Taking
  $M = M_1=M_2$, it follows that $\omega(M) = 1$ when $2\rk(M) = n$
  and $D(M,M)$ has expected rank. If the rank of $D(M,M)$ is
  unexpected, then $\omega(M) = 0$.
\end{example}
\begin{example}
  Assume that $M$ is a matroid on $[n]$ with $2\rk(M) -1 = n - 2$, and
  that $D = D(M,M)$ has expected rank.  Then $\omega(M)$ is equal to
  $\beta(D)/2$ --- half the beta invariant of $D$ (i.e., the coefficient of the either of the linear terms in the Tutte polynomial of $D$). Indeed,
  the minimal resolution of $\Delta_w(M,M)$ has length $2$ and the
  last $\ZZ$-graded Betti number is $\beta(D)$. In the
  $\ZZ^2$-grading, there are two non-zero Betti numbers in homological
  degree $\operatorname{corank}(D)$, and these occur in bidegrees
  $(\rk(M), \rk(M) - 1)$ and $(\rk(M)-1,\rk(M))$. By symmetry these
  are equal and hence either of these Betti numbers is
  $\beta(D)/2$.
  Compare \cite[Proposition 12.7]{FSS}.
\end{example}

The $K$-polynomial of $\Delta_w(M_1,M_2)$ can be explicitly computed
using the dual formulation of Hochster's formula
(\cref{thm:hochster}). However, since the Alexander dual of
$\Delta_w(M_1,M_2)$ is (generally) not Cohen-Macaulay, the formula for
a particular Betti number is not of a predictable sign. Since
$\Delta_w(M_1,M_2)$ is Cohen-Macaulay, \cref{prop:CM implies K theory
  pos} says that $\K(\Delta_w(M_1,M_2), 1-U_1,1-U_2)$ has alternating
signs in its monomial expansion. The coefficients of terms of highest
possible degree in $\K(\Delta_w(M_1,M_2),U_1,U_2)$ will only change by
a predictable sign upon making the substitution $U_j \mapsto 1-U_j$
in the Alexander inversion formula \cite[Theorem~5.14]{millerSturmfels}.
The results of Bayer, Charalambous and Popescu \cite{BCP99}
systematize this observation and give a relationship between certain Betti numbers of a complex
and its Alexander dual.

\begin{definition}
  Let $I \subset \CC[x_1,\dots,x_m]$ be a square-free monomial ideal. Denote the
   finely graded Betti numbers of $\CC[x_1,\dots,x_m]/I$ by $\beta_{k,\mathbf{a}}$, with
  $k \in \ZZ_{\geq 0}, \mathbf{a} \in \{0,1\}^m$. We say that $\beta_{i,\mathbf{b}}$ is \newword{extremal} if $\beta_{j, \mathbf{c}} = 0$ whenever (i) $j \geq i$ and (ii) $\mathbf{c} \in \{0,1\}^m$ is coordinate-wise weakly larger than $\mathbf{b}$ and (iii) $\mathbf{b} \neq \mathbf{c}$.
\end{definition}

We will use the following result.
\begin{theorem}[{\cite[Theorem~2.8]{BCP99}}]\label{thm:BCP99}
  Let $I \subset \CC[x_1,\dots,x_m]$ be a square-free monomial ideal with finely graded Betti numbers $\beta_{k,\mathbf{a}}$. Let $I^\vee$ denote the square-free monomial ideal associated to the Alexander dual of the Stanley-Reisner complex of $I$, whose finely graded Betti numbers we denote $\beta^\vee_{k,\mathbf{a}}$.
  %Then,
  %\[
  %  \beta_{i,\mathbf{b}} \leq \sum_{\mathbf{b}\preceq \mathbf{c}\preceq \mathbf{1}} \beta^\vee_{|\mathbf{b}|-i+1,\mathbf{c}}
  %\]
  If $\beta^\vee_{i,\mathbf{b}}$ is extremal then $\beta^\vee_{i,\mathbf{b}} = \beta_{|\mathbf{b}| - i +1,\mathbf{b}}$.
\end{theorem}
% Here $|\mathbf{b}|$ denotes the sum of the entries of $\mathbf{b}$.
Note that our Betti numbers are those of the quotient
module $\CC[x_1,\dots,x_m]/I$ and not the module $I$ itself, and thus the
index for homological degree of our Betti numbers is adjusted
appropriately from the convention in \cite{BCP99}.

\begin{proposition}
  Let $(M_1,M_2)$ be a pair of matroids on $[n]$. Write the finely graded Betti numbers of $A/I_w(M_1,M_2)$ as
  $\beta_{k,\mathbf{a}}$. For any $\mathbf{b} \in \{0,1\}^n \times \{0,1\}^n$ with $|\mathbf{b}| = n$,
  the Betti number 
  $\beta_{\operatorname{corank}(D(M_1,M_2)),\mathbf{b}}$ is
  extremal.
\end{proposition}
\begin{proof}
  Since $\Delta_w(M_1,M_2)$ is Cohen-Macaulay we know that
  $\beta_{k,\mathbf{a}}$ is zero when $k$ exceeds the codimension of
  $A/I_w(M_1,M_2)$, which is
  $2n - (n + \rk(D(M_1,M_2)) = \operatorname{corank}(D(M_1,M_2))$. If
  $\mathbf{a} \in \{0,1\}^n \times \{0,1\}^n$ has $|\mathbf{a}| > n$ then when we
  coarsen to the $\ZZ^n$ grading we obtain a degree that is not-square
  free. However, $A/I_w(M_1,M_2)$ has the same $\ZZ^n$ graded Betti
  numbers as $A/I(D)$ and all of these
  occur in square-free degrees, e.g.\ by Hochster's formula. It follows that
  $\beta_{k,\mathbf{a}} = 0$ for $k \geq \operatorname{corank}(D(M_1,M_2))$ and
  $|\mathbf{a}| > n$. The result follows.
\end{proof}

We obtain from this a nice formula for the Betti numbers occuring in
the last step of the free resolution $\Delta_w(M_1,M_2)$.
\begin{proposition}\label{prop:formula for top betti}
  Fix $\mathbf{b} \in \{0,1\}^n \times \{0,1\}^n$ with $|\mathbf{b}| = n$ and let $B$ be the set of positions in the first $n$ coordinates of $\mathbf{b}$ which equal $1$. Then for any pair of matroids $(M_1,M_2)$ on $[n]$, $\beta_{\operatorname{corank}(D(M_1,M_2)),\mathbf{b}}( A/I_w(M_1,M_2) )$ is equal to 
  \[
    \dim_\CC \widetilde{H}_{ \rk(D(M_1,M_2)) -1 } ( \link_{\Delta_w(M_1,M_2)}( x_{[n] \setminus B}\, y_{B} ) ).
  \]
\end{proposition}
Note that if $x_{[n] \setminus B}\, y_{B}$ is not a face of
$\Delta_w(M_1,M_2)$ the link in question is empty and the relevant dimension is zero.
\begin{proof}
  Since the Betti number in question is extremal, we compute the
  corresponding Betti number for $\Delta_w(M_1,M_2)^\vee$ prescribed by
  \cref{thm:BCP99}. For this Betti number, we can apply Hochster's formula which gives this result directly after noting that  $x_{[n] \setminus B} y_{B}$ is the (potential) face whose complementary $\ZZ^n \times \ZZ^n$ degree is $\mathbf{b}$.
\end{proof}

By coarsening to the $\ZZ^2$-grading, we can now prove \cref{thm:positive euler char formula}.
\begin{corollary}
  For any pair of
  matroids $(M_1,M_2)$, if $i+j = n$ then
  \begin{multline*}
        (-1)^{\rk(D(M_1,M_2))} {\textstyle\chi\left( \bigwedge^i [\mathcal{Q}^\vee_{M_1}]
            \bigwedge^j [\mathcal{Q}^\vee_{M_2}] \right)} \\
        = \sum_{B \in \binom{[n]}{i}} \dim_\CC \widetilde{H}_{ \rk(D(M_1,M_2)) -1 } ( \link_{\Delta_w(M_1,M_2)}( x_{[n] \setminus B} y_{B} ) ).
  \end{multline*}
\end{corollary}
\begin{proof}
  We extract the coefficient of $U_1^i U_2^j$, $i+j=n$, from the $\ZZ^2$-graded
  $K$-polynomial of $A/I_w(M_1,M_2)$, applying \cref{thm:K poly in general}, \cref{prop:vanishing
    betti} and \cref{prop:formula for top betti}.
\end{proof}

\subsection{Cohomology of tautological bundles}\label{ssec:cohomology of vb} In this section we
apply Weyman's geometric method for syzygies to deduce some results on
cohomology of tautological bundles of linear spaces.

Assume that $(M_1,M_2)$ is realized by a pair of linear spaces
$L_1,L_2 \subset \CC^n$ and that $D(M_1,M_2)$ has expected rank. This
ensures that the Schubert variety $\hat Y_{L_1,L_2}$ is birational to
the sum of the tautological bundles
$\mathcal{S}_{L_1} \oplus \mathcal{S}_{L_2}$. Since, additionally,
$\hat Y_{L_1,L_2}$  is normal and has rational singularities the minimal free resolution of
$A/I(L_1,L_2)$ in the $\ZZ^2$-grading is can be constructed as follows.
\begin{theorem}\label{thm:Z^2 resolution} For all $i \geq 0$, define
   $\ZZ^2$-graded free $A$-modules
  \[
    F_i = \bigoplus_{(j,k) \in (\ZZ_{\geq 0})^2} H^{j+k - i}\left(X_n, \textstyle\bigwedge^j \mathcal{Q}_{L_1}^\vee \otimes \bigwedge^k \mathcal{Q}_{L_2}^\vee  \right) \otimes A((-j,-k)).
  \]
  There are maps $d_i : F_i \to F_{i-1}$ of bidegree $(0,0)$ making
  $(F_\bullet, d_\bullet)$ 
  the minimal free $\ZZ^2$-graded resolution of $A/I(L_1,L_2)$.
\end{theorem}
This is a $\ZZ^2$-graded analogue of the $\ZZ$-graded result of Weyman \cite[Theorem~5.1.3]{weyman}.
\begin{proof}
  By \cref{thm:pushforwards}, the collapsing map
  $\mathcal{S}_{L_1} \oplus \mathcal{S}_{L_2} \to \hat Y_{L_1,L_2}$
  satisfies the assumptions of \cite[Theorem~5.1.2]{weyman},
  and so that theorem gives a $\ZZ$-graded minimal free resolution of $A/I(L_1,L_2)$.
  We explain the needed modifications to Weyman's proof to account for the
  $\ZZ^2$-grading on $A$, referring the reader to \textit{loc.\ cit.}\
  for details.

  Recall from \cref{ssec:K} that $S = (\CC^\times)^2$ is the $2$-torus
  that acts on $\CC^n \times \CC^n$ by inverse scaling of the
  factors. This gives rise to the positive $\ZZ^2$-grading on $A$. The
  bundle $\mathcal{S}_{L_1} \oplus \mathcal{S}_{L_2}$ is a linearized
  $S$-equivariant bundle when $S$ acts trivially on $X_n$. Thus, the
  objects and maps in the diagram
  \[
    Y_{L_1,L_2} \leftarrow \mathcal{S}_{L_1} \oplus \mathcal{S}_{L_2} \to X_n
  \]
  are $S$-equivariant. Let
  $p : \underline{\CC^n \oplus \CC^n} \to X_n$ be the bundle map,
  which is $S$-equivariant. Then
  $p^*( \mathcal{Q}_{L_1} \oplus \mathcal{Q}_{L_2} )$ is an
  $S$-equivariant bundle on $\CC^n \times \CC^n \times X_n$ with an
  $S$-equivariant section
  $(v,w,x) \mapsto (v+(\mathcal{Q}_{L_1})_x, w +
  (\mathcal{Q}_{L_2})_x)$. The resulting Koszul resolution in
  \cite[Theorem~5.1.1]{weyman} of
  $\mathcal{O}_{\mathcal{S}_{L_1} \oplus \mathcal{S}_{L_2}}$ can be
  taken to be $S$-equivariant. The decomposition
  $\bigwedge^k( \mathcal{Q}_{L_1} \oplus \mathcal{Q}_{L_2})^\vee =
  \bigoplus_{i+j = k} (\bigwedge^i \mathcal{Q}_{L_1}^\vee \otimes
  \bigwedge^j \mathcal{Q}_{L_2}^\vee)$ is an $S$-equivariant
  decomposition into its isotypic components and gives a decomposition
  of the $k$\/th term of the Koszul complex. The differentials in the
  complex then become $S$-equivariant linear forms. It follows that
  the resolution of the Koszul complex in
  \cite[Lemma~5.2.3(a)]{weyman} has the term
  $p^*(\bigwedge^i \mathcal{Q}_{L_1} \otimes \bigwedge^j
  \mathcal{Q}_{L_2})$ with a right resolution whose terms are direct
  sums of terms of the form $A((-i,-j)) \otimes \mathcal{O}_{X_n}(n)$
  (here we have embedded $X_n$ in a projective space). The rest of the
  proof follows \textit{mutatis mutandis}.
  \end{proof}

Since $\hat Y_{L_1,L_2}$ is Cohen-Macaulay, the last non-zero term of the free resolution of $A/I(L_1,L_2)$ is $F_{\operatorname{corank}(D(M_1,M_2))}$ and we have a formula for the $\ZZ^2$-graded Betti numbers $\beta_{\operatorname{corank}(D(M_1,M_2)), (j,k)}(A/I(L_1,L_2))$ as
\[
    \dim_\CC H^{j+k - \operatorname{corank}(D(M_1,M_2))}(X_n, \textstyle\bigwedge^j \mathcal{Q}_{L_1}^\vee \otimes \bigwedge^k \mathcal{Q}_{L_2}^\vee  ) .
  \]

  We know that the only twists $(j,k)$ appearing in the last step of the free resolution have $j+k = n$ (\cref{prop:vanishing betti}) and so we obtain:
\begin{proposition} If $j+k \neq n$ then
  \[
        H^{j+k - \operatorname{corank}(D(M_1,M_2))}(X_n, \textstyle\bigwedge^j \mathcal{Q}_{L_1}^\vee \otimes \bigwedge^k \mathcal{Q}_{L_2}^\vee  )  = 0.
  \]
\end{proposition}
On the other hand, if $j+k = n$ then we obtain a combinatorial formula for the vector bundle cohomology using \cref{thm:positive euler char formula} as follows:
\begin{theorem}\label{thm:higher cohomology}
  If $j+k = n$ then
  \[
    \dim H^{\rk(D(M_1,M_2))}(X_n, \textstyle\bigwedge^j \mathcal{Q}_{L_1}^\vee \otimes \bigwedge^k \mathcal{Q}_{L_2}^\vee  )
\]
is equal to 
\[
  \sum_{B \in \binom{[n]}{j}}\dim \widetilde{H}_{\rk(D(M_1,M_2))-1}( \link_{\Delta_w(M_1,M_2)}(x_{[n]\setminus B} y_B) ).
\]

\end{theorem}

\bibliography{omega}{}
\bibliographystyle{halpha}
\end{document}